\documentclass[11pt,
]{amsart}

\usepackage{
amssymb, graphicx,   cite, bm, color
}
\usepackage[font=small,labelfont=bf]{caption} 
\usepackage[subrefformat=parens]{subcaption}

\usepackage{hyperref}
\date{\today}

\numberwithin{equation}{section}%

\newtheorem{theorem}{Theorem}[section]
\newtheorem{proposition}{Proposition}[section]
\newtheorem{lemma}{Lemma}[section]
\newtheorem{definition}{Definition}[section]
\newtheorem{corollary}{Corollary}[section]

\newtheorem{hypothesis}{Hypothesis}[section]

\theoremstyle{definition}
\newtheorem{remark}{Remark}[section]

\DeclareMathOperator{\tr}{tr}
\DeclareMathOperator{\sinc}{sinc}

\DeclareMathOperator{\STD}{STD}
\DeclareMathOperator{\VAR}{VAR}
\DeclareMathOperator{\MEAN}{MEAN}

\DeclareMathOperator{\supp}{supp}

\DeclareMathOperator{\acor}{ACor}
\DeclareMathOperator{\acov}{ACov}

\DeclareMathOperator{\WFH}{WF_{\it h}}

\newcommand{\eps}{\varepsilon}
\newcommand{\Ra}{\mathcal{R}}

\renewcommand{\sc}{semi\-classical}
\newcommand{\R}{{\bf R}}
\DeclareMathOperator{\Id}{Id}
\renewcommand{\r}[1]{(\ref{#1})}
\renewcommand{\Xi}{B}
\newcommand{\PDO}{$\Psi$DO}
\newcommand{\HPDO}{$h$-$\Psi$DO}
\newcommand{\be}[1]{\begin{equation}\label{#1}}
\newcommand{\ee}{\end{equation}}

\renewcommand{\d}{\mathrm{d}}

\renewcommand{\i}{\mathrm{i}}

\newtheorem{example}{Example}

\newcommand{\B}{B}

\title[Sampling with noise]{Sampling  linear inverse problems with noise}
\author[P. Stefanov]{Plamen Stefanov}
\address{Department of Mathematics, Purdue University, West Lafayette, IN 47907}
\thanks{P.S.\ partially supported by the National Science Foundation under
grant DMS-1900475.}

\author[S. Tindel]{Samy Tindel}
\address{Department of Mathematics, Purdue University, West Lafayette, IN 47907}
\thanks{S.T.\ partially supported by the National Science Foundation under
grant DMS-1952966.}

\begin{document}
\begin{abstract}
We study the effect of additive noise to the inversion of FIOs associated to a diffeomorphic canonical relation. We use the microlocal defect measures to measure the power spectrum of the noise in the phase space and analyze how that power spectrum is transformed under the inversion. In general, white noise, for example, is mapped to noise depending on the position and on the direction. 
In particular, we compute the standard deviation, locally, of the noise added to the inversion as a function of the standard deviation of the noise added to the data. As an example, we study the Radon transform in the plane in parallel and fan-beam coordinates, and present numerical examples. 
\end{abstract} 
\maketitle

\section{Introduction}  
The purpose of this work is to study how noise in discrete measurements affects the reconstruction in linear inverse problems 
\be{1}
Af= g,
\ee 
where $A$ is a Fourier Integral Operator (FIO). Examples are the Radon transform and the geodesic X-ray transforms in two dimensions, at least, thermoacoustic tomography, and the linearization of various non-linear inverse problems like boundary and lens rigidity, inverse scattering problems like inverse back-scattering, etc.  We assume that $A$ is associated with a local diffeomorphism (which condition can be relaxed to the clean intersection condition in principle), and elliptic. Then a parametrix exists, which we will denote by $A^{-1}$, also an FIO of the same type. One can regard the problem as propagation of noise under FIOs, rather than under their inverses but we keep the former point of view. 

We want to emphasize that we are not trying to remove noise. That would be only possible with a priori, say statistical information about $f$, but this is not the goal of this work. On the other hand, understanding well the structure of the noise under the action of the inverse would allow for better understanding of what part of $f$ (in phase space) is most affected by noise and would hopefully allow for more efficient noise reduction. 

We study additive noise first. Such noise is typically created by noisy detectors which add certain constant (but usually low) noise to the signal or by background noise.  
  In section~\ref{sec_NA_noise} we study  examples of non-additive noise:  multiplicative noise, Poisson noise as an example of modulation noise, and noise appearing in CT scan. 
In case of additive noise, we are given the noisy data $g+g_\text{noise}$, where $g_\text{noise}$ (a function) is the noise. Then we are trying to solve 
\be{2}
Af = g+g_\text{noise}
\ee
instead. The right-hand side (r.h.s.) may not be in the range of $A$ so a solution may not even exist. What is often done is to apply the adjoint (assuming some Hilbert structure)
\[
A^*Af = A^*g+A^*g_\text{noise},
\]
which automatically cuts the part of $g_\text{noise}$ perpendicular to the range of $A$, 
and then invert $A^*A$, assuming that $A$ is injective in the first place. If not, we invert $A^*A$ on its  range. This can be viewed also as 
the least squares approximation, and is what the Landweber iteration does, for example. So the inversion is
\be{4}
f_\text{noise} = (A^*A)^{-1}A^*g+(A^*A)^{-1}A^*g_\text{noise} = f+ (A^*A)^{-1}A^*g_\text{noise} . 
\ee
Of course, we could do a different ``inversion''. One way to do it to choose a different Hilbert structure.  What is described above is very common however and it is known as the Moore-Penrose inverse. We do not have to assume that the inversion is the Moore-Penrose inverse; it could be any parametrix of $A$, and the Moore-Penrose inverse is such a parametrix under the assumptions we made on $A$.

With the above considerations in mind, we can think of the added noise as 
\be{5}
f_\text{noise}= A^{-1} g_\text{noise},
\ee
where, as above, $A^{-1}$ is a parametrix, and $A^{-1} g_\text{noise}$ is well-defined even if $g_\text{noise}$ is not in the range of $A$. This is also the so described solution of 
\[
Af_\text{noise}=g_\text{noise}. 
\]
We can drop the $f_\text{noise}$ and the $g_\text{noise}$ notation now and just study \r{1} with $g$ not necessarily in the range of $A$, i.e., $g$ is the noise now. 

\begin{example}
The example we will use in this paper is the Radon transform $\mathcal{R}$ in $\R^2$
\be{7}
\mathcal{R}f(\omega,p)=\int_{x\cdot\omega=p}f(x)\,\d\ell, \qquad p\in\R, \; \omega\in S^1,
\ee
where $\d\ell$ is the Euclidean line measure. 
It is written in ``parallel geometry'' coordinates. We  study this example in more detail in section~\ref{sec_R_p}; and in section~\ref{sec_FB}, we will study the same problem for the Radon transform in fan-beam coordinates. It is known that $\mathcal{R}$ is an FIO of order $1/2$ with a canonical relation a graph of a local diffeomorphism (1-to-2). The most popular inversion formula is the ``filtered back projection''
\be{8}
f = \frac1{4\pi}\mathcal{R}' |D_p|g, \quad g= \mathcal{R} f,
\ee
where $\Ra'$ is the transpose in distribution sense; and its versions with adding an  additional filter. We view \r{8} as a unfiltered inversion and that with an additional filter, see \r{3.8}, as a filtered one. 
Now, one can define a norm in the $g$ space by $\| |D_p|^{1/2}g\|_{L^2(\R\times S^1)}$. Then $\mathcal{R}^*= \mathcal{R}'|D_p|$ and \r{8} takes the form $f=(4\pi)^{-1}\mathcal{R}^*g$. Formula \r{8} is used all the time with noisy data not in the range of $\mathcal R$. In addition, we have $\Ra^*\Ra=4\pi\Id$. Therefore, the relation $f=(4\pi)^{-1}\Ra^*g$ can be recast as $f= (\Ra^*\Ra)^{-1}\Ra^* g$, which is exactly \r{4}. 

On the other hand, we may assume that the natural space for $g$ is $L^2(\R\times S^1)$. Then $\mathcal{R}^*\mathcal{R}=4\pi|D|^{-1}$; then the inversion is 
\[
f = \frac1{4\pi}|D|\mathcal{R}' g, \quad g= \mathcal{R} f.
\]
This inversion formula is equivalent to \r{8}. Note that the inverse is a version of \r{4} again. 
\end{example}

Assume that in discrete measurements, the added noise consists of random variables with a known autocorrelation. The simplest case is independent  identically distributed (i.i.d.) random variables at each ``pixel'' (white noise). The distribution could be Gaussian, uniform, etc. 
We convert the discrete measurements to a function on a ``continuous'', space, i.e., locally a function on $\R^n$. Then we invert the data by applying a parametrix, as in \r{5}. The discretization rate is assumed to be proportional to a small parameter $h>0$ and we are interested in the asymptotic properties as $h\to0$. 
Our main goal is a characterization of the induced noise $f_\text{noise}$ after the inversion. 

The novelties of our approach are the following. First, we view discretization and the inverse process --- interpolation from a given discretization, as the step size tends to $0$, in the \sc\ setting, where the small parameter $h>0$ is proportional to the step size. This point of view was proposed in the first author's paper \cite{S-Sampling}. This allows us to use tools from \sc\ analysis to estimate the sharp sampling rate of $Af$, knowing the band limit of $f$, characterize aliasing artifacts during inversion if $Af$ is undersampled, give a sharp limit of the resolution, etc. In this paper, we assume that we do not undersample $Af$. 

The second novelty is moving the analysis of the spectral character of the noise to the phase space; roughly speaking, instead of localizing in the dual variable $\xi$  only, to localize in both the spatial one $x$, and $\xi$.  
 In the applied literature, there are two main ways to characterize noise: through its standard deviation (which assigns just one number) and through its \textit{power spectral density} (or power spectrum). The latter is $|\hat f(\xi)|^2$, where $f$ is the noise, as a function of the frequency $\xi$. Knowing that, we can recover the standard deviation as well, by Parseval's identity. Even though not always explicitly stated, when the noise is not expected to be homogeneous (translation invariant), one can localize in the base variable $x$ by taking the modulus squared of the windowed Fourier transform $|\widehat{\phi f}(\xi)|^2$ with some $\phi\in C_0^\infty$. We propose going one step further: consider the power spectrum in the phase space of points $x$ and (co)directions $\xi$. With the presence of the small parameter $h$, the natural framework is the \sc\ analysis again. 
The \sc\ version of localizing both in space and momentum is to localize near some $x_0$ in the $x$ space with  a smooth cutoff of size $h^{1/2}$ and then take the Fourier transform with $\xi$ replaced by $\xi/h$, see \cite{Zworski_book} for a discussion. The natural candidate of the power spectrum in the phase space then  would be the so-called \sc\ defect measure $\d\mu(x,\xi)$ which, roughly speaking, measures the spectral content of $f= f_h(x)$ in the phase space. We call that measure power spectrum as well.

The third novelty is looking at the noise in ergodic sense, which we also call ``spatial'', i.e., the noise in one measurement. There are two ways to look at the statistical properties of the noise. First, one might be interested in the expected value of the noise pointwise as we keep repeating the same experiment over and over again (in our context, if we have a series of noisy data sets and do an inversion for each one of them). We call this ``temporal'' view, and the analysis of the temporal properties is easier. 
In applications, we have one such experiment however. 
Our goal is to analyze the statistics of the noise in the inversion for a single experiment, as the sampling rate gets smaller and smaller, hence the term ``spatial''. 
In statistics, an estimate with a single experiment is possible when the variables are i.i.d., and we rely on the ergodic properties of the sequence.  

We start with analysis of discrete white noise. The flatness of its spectrum in temporal sense, see \r{d1-2}, is well-known, which justifies its name. In spatial (ergodic) sense, this is true only in a certain averaged sense, see Theorem~\ref{thm_d1}. For the white noise interpolated to a ``continuous'' function, we show that the defect measure $\d\mu$ is flat as well in Theorem~\ref{thm_m}. In Theorem~\ref{thm_m2} we study the spectrum of more general, correlated noise. 

Next, we study propagation of noise under FIOs $A^{-1}$ (or simply $A$) of the mentioned type. With the \sc\ view of noise and is power spectrum, the analysis of the power spectrum of the result $A^{-1}g$ is reduced to the mapping property of a (\sc) defect measure under a (classical) FIO. The answer is given by the Egorov's theorem with some extra care of the zero section. Then the tools described above would allow us the characterize the spectrum of the resulted noise in the reconstruction. We want to emphasize that even if we start with white noise $g$, which has a flat spectrum, the noise $A^{-1}g$ is not homogeneous in general --- its power spectrum depends on the position $x$ and the codirection $\xi$. In particular, its standard deviation may change from a neighborhood of one point to another.

As we mentioned already, our analysis is not restricted to (additive) white noise only, see also section~\ref{sec_NA_noise} for non-additive noise. We can have data with added non-white noise as well, as long as its power density in the general sense we consider it, is well defined. It could be pink, blue noise, if can be anisotropic noise, varying from point to point, or even noise corresponding to a non-absolutely continuous defect measure. For example, we may have the Radon transform $\Ra f(p,\omega)$ with added noise depending on one of those two variables only, then the associated measure would be singular. Theorems~\ref{thm_m} and \r{thm_dmu} still apply and describe the power density of the noise in the reconstruction.

Instead of developing the general abstract theory further, we present its application to the inversion of the Radon transform in the plane. In ``parallel geometry'', we show that the spectral density of the added noise is independent of the position $x$ and proportional to $|\xi|^{1/2}$ up to the Nyquist limit (and the spectral power, which is the square of the density, is proportional to $|\xi|$). In ``fan-beam coordinates'', the noise depends on the position $x$, on $|\xi|$ proportional to $|\xi|^{1/2}$  again but depends on the direction of $\xi$ (relative to $x$) as well. We present many numerical simulations.

Noise is a major concern in the applied inverse problems and has been considered in the literature; nevertheless, we are not aware of directly related works. We will mention only a few more theoretical works about noise and inverse problems. Reconstruction of Riemannian manifolds with noisy data has been studied in  \cite{FeffermanILN}. Using noise a source for a reconstruction has been studied in   \cite{deVerdiere2011semi, Helin_LOS, deVerdiere2009passive,HelinLO}. 

The structure of the paper is as follows. In section~\ref{sec_SC}, we recall some basic facts about \sc\ analysis, needed for our exposition. We also study the relation between classical and \sc\ FIOs. In section~\ref{sec_sampling}, we summarize and develop further some of the results in \cite{S-Sampling}  about sampling in the \sc\ limit. In Theorem~\ref{thm_m} in section~\ref{sec_4},   we prove that the power spectral density of white noise is uniform, by computing its microlocal defect measure. We also show that more general noise satisfying some assumptions, has a well defined microlocal defect measure as well. 
Then we apply Egorov's theorem to describe how that measure transforms under FIOs associated with a canonical diffeomorphism. Sections~\ref{sec_R_p} and \ref{sec_FB} are devoted to an application of the theory to the Radon transform on the plane in parallel and to fan-beam coordinates. We present many numerical examples as well. Multiplicative noise and other type of noise are analyzed in section~\ref{sec_NA_noise}.  Finally, in section~\ref{sec_8}, we analyze discrete white noise without converting it to noise of a continuous variable. We show that it has flat spectrum \textit{on average}.

\textbf{Acknowledgments.} The authors would like to thank Kiril Datchev for his advice and Magda Peligrad for making us aware of the reference  \cite{hu-taylor97}.

\section{Preliminaries on \sc\ analysis}\label{sec_SC}
We recall some basic facts from \sc\ analysis. For more details, we refer to \cite{DimassiSj_book, Martinez_book, Zworski_book}. Before that, a few words about the notation. All norms $\|\cdot\|$ are in $L^2$ unless indicated otherwise; also $\langle \xi\rangle := (1+|\xi|^2)^{1/2}$. We denote by $\mathcal{S}$ the Schwartz class; and $\mathcal{E}'$ is the space of the compactly supported distributions. For a linear operator $A$, $A'$ is the transpose in distribution sense, while $A^*$ is the $L^2$-adjoint. 

\subsection{Semiclassical wave front set} 
The \sc\ Fourier transform $\mathcal{F}_h f$ in $\R^n$ of a function depending also on $h>0$  is given by
\[
\mathcal{F}_hf(\xi) =\int e^{-\i x\cdot\xi/h} f(x)\,\d x.
\]
 Its inverse is $ (2\pi h)^{-n}\mathcal{F}_h^*$. 
We recall the definition of the semiclassical wave front set of a tempered $h$-depended distribution first. In this definition, $h>0$ can be arbitrary but in \sc\ analysis,  $h\in (0,h_0)$ is a ``small'' parameter and we are interested in the behavior of functions and operators as $h$ gets smaller and smaller. Those functions are $h$-dependent and we use the notation $f_h$ or $f_h(x)$ or just $f$. 
The  Sobolev spaces are the semiclassical ones defined by the norm
\[
\|f\|^2_{H_h^s}= (2\pi h)^{-n}\int \langle\xi\rangle^{2s}|\mathcal{F}_hf(\xi)|^2\,\d\xi. 
\]
Then an $h$-dependent family $f_h\in\mathcal{S}'$ is said to be $h$-tempered (or just tempered) if $\|f_h\|_{H^s_h}=O(h^{-N})$ for some $s$ and $N$. All functions in this paper are assumed tempered even if we do not say so. The \sc\ wave front set of a tempered family $f_h$ is the complement of those $(x_0,\xi_0)\in\R^{2n}$ for which there exists a $C_0^\infty$ function $\phi$ so that $\phi(x_0)\not=0$ so that
\[
\mathcal{F}_h(\phi f_h)= O(h^\infty)\quad \text{for $\xi$ in a neighborhood of $\xi_0$}
\]
in $L^\infty$ (or in any other ``reasonable'' space, which does not change the notion). The \sc\ wave front set naturally lies in $T^*\R^n$ but it is not conical as in the classical case. Elements of the zero section can be in $\WFH(f)$. 

Sj\"ostrand proposed essentially adding the classical wave front set to $\WFH$ by considering the latter in $T^*\R^{n}\cup S^* \R^{n}$, where the second space (the unit cosphere bundle)  represents $T^*\R^n$ as a conic set, i.e., each $(x,\xi)$ with $\xi$ unit is identified with the ray $(x,s\xi)$, $s>0$. Their points are viewed as ``infinite'' ones describing the behavior as $\xi\to\infty$ along different directions.  An infinite point  $(x_0,\xi_0)$ does not belong to the so extended $\WFH(f)$ if we have
\be{infWF}
\mathcal{F}_h(\phi f_h)= O(h^\infty\langle \xi \rangle^{-\infty})\quad  \text{for $\xi$ in a conical neighborhood of $\xi_0$}
\ee
with $\phi$ as above.

\subsection{Semiclassical pseudo-differential operators ($h$-\PDO s)} 
We define the symbol class $S^{m.k}$ of symbols in $ \R^n$  as the smooth functions $p(x,\xi)$ on $\R^{2n}$, depending also on $h$, satisfying the symbol estimates
\be{hpdo1}
|\partial_x^\alpha \partial_\xi^\beta p(x,\xi)|\le C_{\alpha,\beta,K} h^k \langle \xi\rangle^{m},
\ee
for $x$ in any compact set $K$, see, e.g., \cite{Gerard}. In fact, we are going to work with symbols supported in a fixed compact set in the $\xi$ variable, so the behavior in $\xi$ above does not matter; one may also work with the symbol class $h^k S^m(1)$, see \cite{Martinez_book, Zworski_book} where $ S^m(1)$ is defined as \r{hpdo1} with $k=m=0$. 
Given $p\in S^{m}$, we write $P=P_h=p(x,hD)$ with 
\be{hPDO}
Pf(x) = (2\pi h)^{-n}\iint e^{\i (x-y)\cdot\xi/h} p(x,\xi) f(y)\,\d y\, \d\xi,
\ee
where the integral has to be understood as an oscillatory one. This is the standard quantization; sometimes it is convenient to work with the Weyl one $p^{\rm w}(x,hD)$,  where $p(x,\xi)$ is replaced by $p((x+y)/2, \xi)$ in \r{hPDO}. Then real symbols correspond to symmetric operators, in particular. 
Negligible operators are those with $O(h^\infty)$ norms in any pair of Sobolev spaces.

\subsection{Semiclassically band limited functions} 
In  \cite{Zworski_book}, it is said that a tempered  $f_h$  is localized in phase space, if there exists $p\in C_0^\infty(\R^{2n})$ so that
\[
(\Id - p(x,hD))f_h=O(h^\infty),\quad \text{in $\mathcal{S}(\R^n)$}.
\]
All functions in this paper will be of this type.   

It is convenient to introduce the notation $\Sigma_h(f)$ for the \sc\ frequency set of $f$.
\begin{definition}\label{def_S} For each tempered $f_h$ localized  in phase space, set 
\[
\Sigma_h(f) = \{\xi;\; \text{$\exists x$ so that $(x,\xi)\in\WFH(f)$}\}.
\]
\end{definition}
In other words, $\Sigma_p$ is the projection of $\WFH(f)$ to the second variable, i.e.,
\[
\Sigma_h(f)= \pi_2\circ \WFH(f),
\]
where $\pi_2(x,\xi)=\xi$. 
 If $\WFH(f)$ (which is always closed) is bounded and therefore compact, then $\Sigma_h(f)$ is compact.

In \cite{S-Sampling}, we gave the following definition.

\begin{definition}\label{def_B}
We say that   $f_h\in C_0^\infty(\R^n)$  is \sc ly band limited (in $\mathcal{B}$),  
if (i) $\supp f_h$ is contained in an $h$-independent compact set, 
(ii) $f$ is tempered,  and (iii) 
there exists a compact set  $\mathcal{B}\subset\R^n$,  so that  for every open $U\supset \mathcal{B}$, we have 
\[
|\mathcal{F}_h f(\xi)|\le  C_N h^N\langle\xi\rangle^{-N}\quad \text{for $\xi\not\in U$}
\]
for every $N>0$. 
\end{definition}

We showed in \cite{S-Sampling} that $f_h$ is \sc\ band limited if and only if it is localized in phase space and if and only of $\WFH(f)$ is finite (no points of the type \r{infWF}) and compact. 

 In applications, we  take $\mathcal{B}$ to be $[-B,B]^n$ with some $B>0$ or the ball $|\xi|\le B$.

An example of \sc ly band limited functions can be obtained by taking any $f\in\mathcal{E}'(
\R^n)$ and convolving if with $\phi_h=h^n\phi(\cdot/h)$ with $\supp\hat\phi\in C_0^\infty$. Then $\phi_h*f$ is \sc ly band limited with $\mathcal{B}=\supp\hat\phi$.

\subsection{Classical \PDO s as \sc\ \PDO s} In the applications we have in mind, we deal with classical \PDO s and FIOs and want to treat them as \sc\ ones. The negligible operators in the classical calculus are the smoothing ones. We showed in \cite{S-Sampling} that for every $f\in \mathcal{E}'(\R^n)$ and for every smoothing $K$, we have $\WFH(Kf)\subset \R^n\times \{0\}$. Next, every classical \PDO\ of order $m$ can be written as an oscillatory integral of the kind \r{hPDO} with $h=1$ and a symbol $a(x,\xi)$ vanishing for $|\xi|\le1$, plus a smoothing operator. Then formally, that oscillatory integral is an \HPDO\ with symbol $a(x,\xi/h)$. Then we can replace $\langle\xi \rangle$ in \r{hpdo1} by $|\xi|$ to obtain an equivalent estimate, and
\[
|\partial_x^\alpha \partial_\xi^\beta a(x,\xi/h)|\le C_{\alpha,\beta,K} h^{-|\beta|} |\xi/h|^{m-|\beta|}=  C_{\alpha,\beta,K} h^{-m} |\xi|^{m-|\beta|}.
\]
On the support of the symbol, we have $|\xi|\ge h$, therefore the factor $|\xi|^{m-|\beta|}$ is not uniformly bounded near $\xi=0$ when $m<|\beta|$. On the other hand, it is uniformly bounded when $|\xi|\ge\eps$ with $\eps>h$. This allowed us in \cite{S-Sampling}, for every $\eps>0$, to split $a(x,D)$ into an \HPDO\ with symbol $a(x,\xi/h)(1-\chi(\xi/\eps))$ with some cut-off function $\chi\in C_0^\infty$ plus an operator mapping \sc ly band limited functions into functions with \sc\ wave front set in an $O(\eps)$ neighborhood of the zero section. We show below that we can do the same thing for FIOs associated with canonical diffeomorphisms.

Let $A$ be a properly supported FIO with a canonical relation which is a graph of a homogeneous canonical transformation. Then microlocally, $A$ is of the form
\be{A00}
Af(x) = (2\pi)^{-n} \iint e^{\i(\phi(x,\eta)-y\cdot\eta)}a(x,\eta)f(y)\,\d y\,\d\eta,
\ee
see \cite[sec.25.3]{Hormander4}, with $a$ a classical symbol and a phase $\phi(x,\eta)$ homogeneous in $\eta$ of odder $1$, satisfying $\det\phi_{x\eta}\not=0$, $\phi_x\not=0$ for $\eta\not=0$. 
Let $\psi\in C_0^\infty$ have support in $B(0,2)$, $\psi=1$ on $B(0,1)$,  and fix $\eps>0$. Then $A=A_{h,\eps} + R_{h,\eps}$, where
\be{AR}
\begin{split}
A_{h,\eps}f(x) &=  \left[A(\Id-\psi(hD/\eps))f\right] (x)\\
& = (2\pi h)^{-n} \iint e^{\i(\phi(x,\eta)-y\cdot\eta)/h}a(x,\eta/h)(1-\psi(\eta/\eps)) f(y)\,\d y\,\d\eta,\\
R_{h,\eps}f(x) &= [A \psi(hD/\eps)f](x) = (2\pi h)^{-n} \iint e^{\i(\phi(x,\eta)-y\cdot\eta)/h}a(x,\eta/h)\psi(\eta/\eps) f(y)\,\d y\,\d\eta.
\end{split}
\ee

\begin{theorem}\label{thm_c-sc}
Under the assumptions above, 

(a) The operator $A_{h,\eps}$ is an $h$-FIO with a (\sc) canonical relation the same as the (classical) one of $A$. Moreover, for every \sc ly band limited $f$ with $\WFH(f)\cap (\R^n\times B(0,\eps))=\emptyset$, we have $Af=A_{h,\eps}f+O(h^\infty)$. 

(b) 
For every  $f_h\in \mathcal{E}'(\R^n)$ with $|\mathcal{F}_hf_h|\le Ch^{-N}$ for some $N$, we have $\WFH(R_{h,\eps}f )\subset \R^n\times B(0,C\eps)$ with some $C>0$. 
\end{theorem}

\begin{proof}
It follows from \r{hpdo1} that if $a$ is a classical symbol of order $m$, then $a(x,\eta/h)$ is a \sc\ one of order $(m,-m)$ for $|\eta|/h>1$.  Therefore our claim (a) is true for $|\eta|>\eps/2$ and $0<h<\eps$, which is true on the support of the symbol $ a(x,\eta/h)(1-\psi(\eta/\eps)) $ of $A_{h,\eps}$. Hence $\tilde a$ is a \sc\ symbol of order $(-m,m)$. The second part of (a) is immediate. 

To prove (b), multiply $R_{h,\eps}f$ by $\rho\in C_0^\infty$ and apply $\mathcal{F}_h$:
\[
\mathcal{F}_h \rho R_{h,\eps}f(\xi) = (2\pi h)^{-n} \iint e^{\i(\phi(x,\eta)-x\cdot\xi)/h}\rho(x) a(x,\eta/h)\psi(\eta/\eps) \mathcal{F}_h f(\eta)\,\d\eta\,\d x.
\] 
For the phase $\Phi:=\phi(x,\eta)-x\cdot\xi$ we have $\Phi_x= \phi_x(x,\eta)-\xi$. By the homogeneity of $\phi$, for $|\eta|\le 2\eps$, and $|\xi|>C\eps$, we have $\Phi_x\not=0$. Then a stationary phase argument implies $\mathcal{F}_h \rho R_{h,\eps}f(\xi)=O(h^\infty)$ for such $\xi$. This proves (b). 
\end{proof}

\subsection{Semiclassical defect measures} \label{sec_measures}
Given $f_h$ with $\|f_h\|\le C$, one can show that there exists a sequence $h_j\to0$ so that the limit 
\be{9}
 \lim_{h=h_j\to 0+}\left( p(x,hD)f_h,f_h \right)_{L^2} = \int  p(x,\xi)\, \d\mu_f(x,\xi) 
\ee 
exists for every symbol $p\in C_0^\infty$, see \cite{Martinez_book, Zworski_book}, and defines a Borel measure $\d\mu_f(x,\xi)\ge0$ called 
a \sc\ defect measure associated to $f$. That measure may not be unique. Note that $\d\mu_f$ is invariantly defined on $T^*\R^n$. On the other hand, its definition \r{9} depends on the choice of the measure (respectively the coordinates) used to define the $L^2$ space there. We can use every quantization of $p$ in \r{9}, for example the Weyl one $p^{\rm w}(x,hD)$ which guarantees that \r{9} is real when $p$ is real-valued. 

When $f$ is \sc ly band limited,  $\WFH(f)$ is compact, hence $\d\mu_f$ has compact support as well, and 
\be{9a}
\|f_{h_j}\|_{L^2}^2 = \int \d\mu_f +o(1). 
\ee
This in particular implies that our assumption guarantees that $\|f_{h_j}\|_{L^2}$ is asymptotically constant as $h_j\to0$. In fact, some authors require $\|f_h\|=1$, see \cite{Martinez_book}.

\section{Sampling in the \sc\ limit}\label{sec_sampling}
\subsection{Sampling \sc ly band limited functions} We recall some results in \cite{S-Sampling} first. 
The classical Nyquist--Shannon sampling theorem says that a function $f\in L^2(\R^n)$ with a Fourier transform $\hat f$ supported in the box $[-\Xi, \Xi]^n$ can be uniquely and stably recovered from its samples $f(sk)$, $k\in\mathbf{Z}^n$ as long as $0<s\le \pi/\Xi$. More precisely, we have
\be{1a}
f(x) = \sum_{k\in\mathbf{Z}^n} f(sk)\chi_k (x), \quad \chi_k(x):= \prod_{j=1}^n\sinc \Big(\frac{1}{s}(x_j-sk_j)\Big), 
\ee
where we adopt the ``engineering'' definition of the sinc function
\[
\sinc(x)=\sin(\pi x)/\pi x.
\]
Moreover,
\[
\|f\|^2 = s^n\sum_{k\in\mathbf{Z}^n} |f(sk)|^2,
\]
where $\|\cdot\|$ is the $L^2$ norm, see, e.g., \cite{Natterer-book} or \cite{Epstein-book}. 

The proof is based on viewing the samples $f(sk)$ as the (inverse) Fourier coefficients of $\hat f$, extended as $2\pi/s$-periodic function. We reproduce the proof below in the \sc\ case. 

In \cite{S-Sampling}, we formulated this, and related results in the \sc\ setting. One of those theorems is the following. Recall that $\Sigma_h(f)$ is defined in Definition~\ref{def_S}.

\begin{theorem}\label{thm_sc}
Let $f_h$ be \sc ly band limited with $\Sigma_h(f) \subset  \prod(-B_j,\B_j)$ with some $B_j>0$. Let $\hat \chi_j \in L^\infty(\R)$ be supported in $[-\pi,\pi]$, and $\hat \chi_j(\pi\xi_j/B_j)=1$ for $\xi\in \Sigma_h(f) $. If $0<s_j \le \pi/B_j$, then 
\be{2.3}
f_h(x) = \sum_{k\in \mathbf{Z}^n} f_h(s_1 hk_1,\dots, s_nhk_n) \prod_j \chi_j\left(\frac{1}{s_j h}(x_j-s_jhk_j)\right) + O_{\mathcal{S}}(h^\infty) \|f\| ,
\ee
and 
\be{2.4}
\|f_h\|^2 = s_1\dots s_n h^n\sum_{k\in \mathbf{Z}^n} |f_h(s_1hk_1,\dots, s_nhk_n)|^2 + O(h^\infty)\|f\|^2.
\ee
\end{theorem}

One could think of $\chi_j$ as somewhat better versions of the sinc function: they decay faster if we choose $\hat\chi_j$ to be smooth. We can do this because $\Sigma_h(f) $ (which is compact) is assumed to be included in the interior of  the closed  $\prod[-B_j,\B_j]$. 
In case of an equality, we must take $\chi_j(x)=\sinc(x)$.

Assume now for simplicity that all $B_j$ and $s_j$ are equal to some $B$ and $s$, respectively. 
We can always choose a linear transformation $y=Wx$ to get back to \r{2.3} or even more general sampling grids, and the dual one $\xi=W^*\eta$ for the dual variables. 
Set $\chi(x)=\chi_1(x_1)\cdots\chi_n(x_n)$. Then \r{2.3} and \r{2.4} take the form
\be{2.5x}
f_h(x)= \sum_{k\in\mathbf{Z}^n} f_h(shk)\chi_k(x) + O_{\mathcal{S}}(h^\infty) \|f\|, \quad \chi_k(x) :=  \chi\left(\frac{1}{s h}(x-shk)\right)
\ee
and
\be{2.5xx}
\|f_h\|^2 = (sh)^n\sum_{k\in \mathbf{Z}^n} |f_h(shk)|^2 + O(h^\infty)\|f\|^2.
\ee

The proof of Theorem~\ref{thm_sc} is based on the following observation. Since $\mathcal{F}_hf$ is supported in 
$\prod(-B_j,\B_j)$  up to an $O(h^\infty)$ error,  and $\pi/s >B_j$, we have
\be{S1}
(\mathcal{F}_h f)_\text{ext}(\xi) =   (sh)^n \sum_k f(shk)e^{-\i s x \cdot \xi} +  O_{\mathcal{S}}(h^\infty), 
\ee
where $(\mathcal{F}_h f)_\text{ext}(\xi) $ is the periodic extension of $\mathcal{F}_h f(\xi) $ with period $2\pi/s$ in each 1D variable. 
Multiply this by $ \hat \chi(s\xi)$ to get 
\[
\mathcal{F}_h f(\xi) = (sh)^n \hat \chi(s\xi)  \sum_k f(shk)e^{-\i s\xi\cdot k} +  O_{\mathcal{S}}(h^\infty).
\]
If $\chi_k$ is the interpolating function in \r{2.5x}, then
\be{2.6}
\mathcal{F}_h\chi_k(\xi)= (sh)^n \hat \chi(s\xi) e^{-\i s\xi\cdot k}.
\ee
Take $\mathcal{F}_h^{-1}$ to complete the proof. The full details can be found in \cite{S-Sampling}. Also, $\chi$ does not need to be of product type, as shown there. 

\begin{remark}\label{rem_1}
In the limit case $\Sigma_h(f) =  (-B,B)^n$, which is not allowed by the theorem since $\Sigma_h(f) $ is compact, we have $\hat\chi=\mathbf{1}_{[-\pi, \pi]^n}$, where $\mathbf{1}_{[-\pi, \pi]^n}$ stands for the characteristic function of $[-\pi, \pi]^n$. Then $\chi$ is a product of sinc functions, see \r{1a}. We will use the notation 
\be{sinc_k}
\sinc_k(x):= \prod_{j=1}^n\sinc \Big(\frac{1}{sh}(x_j-shk_j)\Big).
\ee
Then \r{2.6} takes the form 
	\be{2.8}
\mathcal{F}_h\sinc_k(\xi) =  (sh)^n  \mathbf{1}_{[-\pi/s, \pi/s]^n} (\xi)  e^{-\i s k\cdot \xi}.
	\ee
The functions $\sinc_k$ 
form an orthogonal system, and  
\be{phi_k}
\phi_k:= (sh)^{-n/2}\sinc_k
\ee
 is an orthonormal basis in the subspace  $\textbf{1}_{[-\pi/s,\pi/s]^n}(hD)L^2(\R^n)$. For future reference, we want to mention that for every $m\ge2$ integer, if $\sinc_k^{(m)}(\xi)$ is defined by 
\be{2.8x}
\mathcal{F}_h\sinc_k^{(m)}(\xi) =  (sh/m)^n  \mathbf{1}_{[-m\pi/s, m\pi/s]^n} (\xi)  e^{-\i s k\cdot \xi}
\ee
then 
\be{2.8a}
\sinc^{(m)} _k(x)=  \prod_j\sinc\left(\frac{m}{s h}(x_j-shk_j)\right)
\ee
instead.  Then 
\[
\phi_k^{(m)} := (sh/m)^{-n/2}\sinc_k^{(m)}
\]
  is an orthonormal system in $\textbf{1}_{[-m\pi/s,m\pi/s]^n}(hD)L^2(\R^n)$ but not a basis. To make it a basis, notice that $s$ was replaced by $s/m$ and $k$ was replaced by $mk$ there. Allowing the original $k$ to run over all integer points, i.e., replacing $k$ by $k/m$ there would complete \r{2.8a} to a basis. 
\end{remark}

\subsection{Constructing a \sc ly band limited function from a discrete sequence} 
The next question is how to associate a \sc ly band limited function to a set of numbers $\bm f_k$, $k\in\mathbf{Z}^n$, which we view as its samples. Without the band limited requirement, this can be done in infinitely many ways, of course, by various ways to interpolate between the samples. On the other hand, if we fix the band limit $B$, then for $B<\pi/s$, such a function, if exists, would be oversampled, and those samples can be shown to be dependent. We can only hope that this problem always has a solution when $B\ge\pi/s$. The next proposition, proven in \cite{S-Sampling}, shows that it can be done when $B=\pi/s$ with a sinc interpolation. 

\begin{proposition}\label{pr_f} 
Let $\Omega\Subset \Omega_1$ be both open.  
Fix $s>0$. 
For $0<h\ll1$, let $K(h)\subset \mathbf{Z}^n$ be the set of those $k$ for which $shk\in \Omega$. Then for every collection of complex numbers $\{\bm f_{k,h}\}$, $k\in K(h)$ with $\sum_k |\bm f_{k,h}|^2$ tempered, there exists a \sc ly band limited $f_h$ with $\WFH(f)\subset \Omega_1\times [-\pi/s,\pi/s]^n$ so that $f(shk) = \bm f_{k,h}$.

One such choice  is given by
\be{2.3''}
\tilde f_h=\psi f_h, \quad \text{with\ \ } 
f_h(x) =  \sum_{k\in K(h)} \bm f_{k,h} \sinc_k(x),  
\ee
where  $\psi\in C_0^\infty(\Omega_1)$  is equal to $1$ near $\bar\Omega$. 
Moreover, \r{2.5xx} holds. 
\end{proposition}

\begin{proof}    
With $f_h$ as above, we have $f_h(shk) = \bm f_h(k)$ and 
\be{2.d.2}
\mathcal{F}_hf_h(\xi) = \textbf{1}_{[-\pi/s,\pi/s]^n}(\xi)(sh)^n \sum_{k\in K(h)} \bm f_{k,h} e^{-\i sk\cdot\xi},
\ee
compare to \r{S1} and \r{2.8}. 
 Then $\WFH(f_h)\subset \R^n\times [-\pi/s,\pi/s]^n$. 
Let $\psi$ be as in the theorem. Then $\tilde f_h=\psi f_h$ has the required properties. 
\end{proof}

By Theorem~\ref{thm_sc}, \r{2.3''} is the only such representation when $f_h$ is restricted to $\Omega$, up to an $O(h^\infty)$ error, if we want to keep the band limit $B$ to be the sharp one $B=\pi/s$.  

The expansion \r{2.3''} has the usual downsides associated with the presence of the sinc functions there --- they decay too slowly at infinity allowing the influence of each term to extend too far. When $B>\pi/s$ (strictly), we can have the localized interpolation functions $\chi_k$ of Theorem~\ref{thm_sc} in principle. The situation is different than  that in Theorem~\ref{thm_sc} though. The functions $\chi_j$ in \r{2.3} do not necessarily satisfy $\chi_{k_1}(shk_2)= \delta_{k_1,k_2}$, where $\delta_{k_1,k_2}$ stands for the Kronecker symbol.  In the case under consideration, they have to (up to an $O(h^\infty)$ error). Also, when $B>\pi/s$, the corresponding function $f_h$ would be undersampled rather than oversampled. Next, in interpolations like these, the desire is to make it as smooth as possible. 

One way to enforce $f_h(shk) = \bm f_{k,h}$  is to replace the sinc function in \r{2.8a} with itself, multiplied by some $\phi\in \mathcal{S}$ with $\phi(0)=1$,  $\hat\phi\in C_0^\infty$, i.e., to put a product of $\sinc(x)\phi(x)$ for each point $x_j$. Then $\mathcal{F}(\phi\sinc)=(2\pi)^{-1} \mathbf{1}_{[-\pi,\pi]}*\hat\phi $ has support larger than $[-\pi,\pi]$ which corresponds to a band limit greater than $\pi/s$, compare with \r{2.d.2}. One can also have a rapidly decreasing $\hat\phi$ instead of a $C_0^\infty$ one, and the resulting error by replacing it with a suitable  a $C_0^\infty$ one can be estimated easily. 

\subsection{ Lanczos-3 interpolation and other convolution based interpolations} \label{sec_Lan3}
One practical and approximate realization of the idea above is the Lanczos-3 interpolation. 
It is part of the family of the Lanczos-k interpolations with the number $3$ below replaced by an integer $k$. In it, the functions $\chi_j$ in \r{2.3}  are taken to be
\[
\text{Lan3}(x) := H(3-|x|)\sinc(x) \sinc(x/3),
\]
where $H$ is the Heaviside function, and $x$ stands for each coordinate function $x_i$.  Its Fourier transform is not of compact support but decays like $O(|\xi|^{-2})$ with a small leading term; and it is very small outside $|\xi|\le 2\pi$,  see Figure~\ref{fig_3}; as opposed to $\sinc(x)$ which Fourier transform is supported on $|\xi|\le\pi$. The kernel $\text{Lan3}(x)$ is easy to compute numerically, has a small support, and preserves the property $f(shk)=\bm f_{k,h}$ because $\text{Lan3}(k) = \delta_{k,0}$ for $k$ an  integer.  So for all practical purposes, choosing $\chi$ in \r{2.3''} to be $\text{Lan3}$, provides an interpolation with a band limit no greater than $B=2\pi/s$ and even $B=1.5\pi/s$; $1.5$ to $2$ times that of \r{2.3''}, see Figure~\ref{fig_3}. 

\begin{figure}[ht]
\begin{center}
	\includegraphics[scale=0.3]{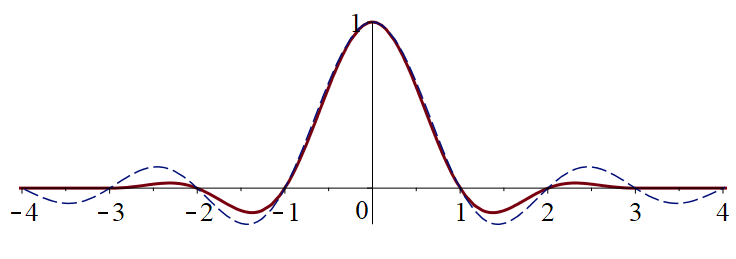}
	\includegraphics[scale=0.3]{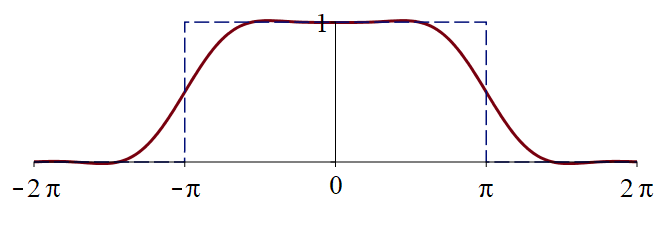}
\end{center}
\caption{\small The Lanczos-3 kernel $\text{Lan3}$ and its Fourier transform.  The  sinc kernel and its Fourier transform are shown as dashed lines. 
}  
\label{fig_3}
\end{figure}
If the samples in \r{2.4L} with $\chi$ a Lanczos-3 kernel, are those of a function with a band limit $B=\pi/s$ (the Nyquist limit for that step size), then the reconstruction will leave frequencies below $B/2$ mostly unchanged, and will attenuate and alias those between $B/2$ and $B$ as in Figure~\ref{fig_3}. 
The resulting aliasing will be ``small'' because the amplitude is ``small'' away from $|\xi|\le B$ (in Figure~\ref{fig_3}, $B=\pi$). 
The created $f_h$ will have an essential  band limit larger than $B$, as explained above, and this is true even if the samples are arbitrary.

The property of the Lanczos-3 kernel to be almost $1$ in $[-\pi/2,\pi/2]$ can be used to practical interpolations with an explicit kernel with small support approximating well enough $\chi_j$ in~\r{2.3}. For this, it is enough to oversample  twice or even $1.5$ times only in each coordinate and use the Lanczos-3 kernel.  
We use this technique in the numerical computations later.  This way, we work with a very well localized kernel rather than with the sinc one.

The Lanczos-3 interpolation belongs to the family of the convolution based interpolations of the type 
\be{2.4L}
f_h(x) =  \sum_{k\in K(h)} \bm f_{k,h} \chi_k(x), \quad \chi_k(x):=  \chi\Big(\frac{1}{sh}(x-shk)\Big)  ,
\ee
with various compactly supported kernels $\chi$. It is easy to see that this is the case when the interpolation is translation invariant, has a finite domain of influence,  and is a linear operator. The simplest examples are the nearest neighbor ($\chi_k$ are characteristic functions of boxes in $\R^n$) and  the linear interpolation. 
Some of the higher order ones are the third order cubic Catmull-Rom spline and a fourth order cubic  spline proposed by Keyes, see \cite{Meijering2003, Keys1981}. Without going into detail, we will mention that those two are very similar to Lanczos-2 and Lanczos-3, respectively with the Keyes one being a bit more smoothing that Lanczos-3.
The Fourier transforms $\hat\chi$  related to the cubic interpolations and the Lanczos-3 ones decay fast enough to be well approximated with compactly supported ones. Then we have the following.

\begin{proposition}\label{pr_P}
Let $\hat\chi\in L^\infty_\text{\rm comp}$. Then for $f_h$ given by \r{2.4L} we have
\[
\|f_h\|^2\le C (sh)^n \sum_{k\in K(h)} |\bm f_{k,h}|^2, \quad C:=\|\hat\chi\|^2_{L^\infty}. 
\]
\end{proposition} 

\begin{proof}
Let $1\le m\in\mathbf{Z}$ be such that $\supp\hat\chi\in [-m\pi, m\pi]^n$. Let $\sinc_k^{(m)}$ be defined by \r{2.8a}. Then  $\phi_k= (sh/m)^{-n/2}\sinc_k^{(m)}$ form an orthonormal system, see Remark~\ref{rem_1}. For
\begin{equation}\label{a1}
g_h (x) := \sum_{k\in K(h)} \bm f_{k,h} \sinc_k^{(m)}(x)
\end{equation}
we have
\be{pr_eq1}
\|g_h\|^2= (sh/m)^n \sum_{k\in K(h)}|\bm f_{k,h}|^2.
\ee
Multiply \r{2.8x} by $\hat\chi(s\xi)$, we get 
\be{pr_eq2}
\hat\chi(shD) \sinc_k^{(m)}=m^{-n}\hat\chi_k ,
\ee
 where we used \r{2.6}, valid for $\chi_k$ as in \r{2.4L} with every $\chi$ 
as in the proposition. Therefore,  $f_h= m^n\hat\chi(shD)  g_h$. 
 {Hence it is easily seen that $\|f_h\|\le m^n\|\hat\chi\|_{L^\infty}\|g_h\|$, 
 where we recall that $g_{h}$ is defined by \eqref{a1}}. 
 Combining this with \r{pr_eq1}, we complete the  proof.
\end{proof}

\subsection{Noisy samples} \label{sec_noisy_samples}
Let us say  we restore a \sc ly band limited function from noisy samples. Assume oversampling, i.e., $B<\pi/s$ (strictly). 
Without noise, we would use the formula \r{2.5x} where $f_h(\cdot)$ are the samples which we call here $\bm f_{k,h} $. In other words, we would take $f_h$ as in \r{2.4L} 
with $\chi$ so that
\be{2.5}
\text{$\hat\chi=1$ on  $[-sB,sB]^n$ }, \qquad \supp\hat\chi\subset (-\pi,\pi)^n ,
\ee
which we  take to be in the Schwartz class; and we can do this  since we can choose $\hat\chi$  to be in $C_0^\infty$.

If we do the same thing with the noisy samples, the added noise will be given by \r{2.4L} 
again. Then $f(shk)=\bm f_{k,h}$ would be true for the noise free samples since a priori, $f_h$ has a band limit $B$. This would not be true for the noisy samples, in general, because they are not necessary samples of such a function. {In fact, one of the  goals of the current contribution is to tackle this issue.}

{As in the proof of Proposition \ref{pr_P},} 
note that \r{2.4L} and the sinc reconstruction \r{2.3''} are closely connected: one can get the former from the latter by applying the convolution operator $\hat\chi(shD)$ to it. 

\subsection{Delta type of expansion}\label{sec_delta}
We can view the convolution based interpolation \r{2.4L} as a convolved delta type of expansion in which $\chi$ is formally replaced by the Dirac delta. Indeed, start with 
\be{2.4delta}
f_h^\delta(x) : =  (sh)^n\sum_{k\in K(h)} \bm f_{k,h}   \delta(x-shk)  ,
\ee
then $f_h=\chi_h*f_h^\delta$ with $\chi_h(x)=(sh)^{-n}\chi(x/(sh))$. On the Fourier side, we have \r{2.d.2} without the cutoff function $\mathbf{1}$ there. 

\section{Noise and defect measures}\label{sec_4}

\subsection{Microlocal defect measures as a generalization of power density}\label{sec_DM} 

We start this section by specifying the kind of \textit{white noise} considered in the sequel, see also section~\ref{sec_FB}. 

\begin{hypothesis}\label{hyp:noise}
For every  $h>0$, the noise is modeled by a family $\{ \bm f_{k,h}; \,  k\in \mathbf{Z}^n \}$ of independent and identically distributed (i.i.d.) real valued  random variables defined on the same probability space $(\mathbb{X},\mathcal{F},\mathbb{P})$.
The random variables $\bm f_{k,h}$ have zero expected values and a common finite variance $\sigma^2$. For our computations we also make the following technical assumption on the common  higher  moments: 
there exists a constant $\delta>0$ such that 
\begin{equation} \label{H1}
\mathbb{E}\left(  \bm f_{k,h}^{4} \, \left(\log(1+|\bm f_{k,h}|)\right)^{1+\delta} \right)   < \infty .
\end{equation}
\end{hypothesis}

The variables $\bm f_{k,h}$ model the noise at each cell/pixel $x_k=shk$, with the relative step $s>0$ fixed, and $h>0$ a small parameter. In Hypothesis~\ref{hyp:noise} we allow $\bm f_k$ to depend on $h$,  but $h$ will often be omitted for notational sake. In the numerical examples later, we use either normally distributed $\bm f_k$ or uniformly distributed ones. For a fixed bounded domain $\Omega$, the number of sampling points $x_k=shk$ in it (we called that index set $K(h)$ in Proposition~\ref{pr_f}) is $|\Omega|(sh)^{-n}(1+o(h))$. For each $h>0$, only that many $\bm f_{k,h}$'s will be used eventually; therefore, we have a triangular array of random variables $\bm f_{k,h}$, $h>0$,  $k\in K(h)$.

As explained in the introduction, there are two types of statistical properties we are interested in. First, what we call ``temporal'' mean, variance, etc., are the moments of each $\bm f_k$ as a random variable. They are determined by the process which creates them and in practical applications correspond to  repeated experiments, hence the term ``temporal''. We use the notation $\mathbb{E}(\bm f_k)$, $\mathbb{E}(\bm f_k^2)$, etc.,  for the expectation. The second, and the more interesting kind of properties are for a single experiment as $h\to0$, i.e., when the number $N~\sim h^{-n}$  of $\bm f_k$ grows. The mean is just the mean of those finitely many numbers, and the variance is the mean of their squares. We call them {empirical spatial} mean and variance, using the notation $\VAR$ for the latter and $\STD$ for the spatial standard deviation. {Limit theorems for averaged random quantities with certain invariances} are called sometimes  ergodic properties; we view them as ``spatial'' ones, interpreting $\bm f_k$ as samples of some function in space. By the strong law of large numbers, the mean of $\sim h^{-n}$ of $\bm f_k$'s converges to zero almost surely, and its spatial variance converges to $\sigma^2$ almost surely, as $N\to\infty$. 
Below, we define similar quantities for continuous {function-valued} random variables. 

Our terminology could be confusing since for random processes, 
{that is families $\{f(t)\}_{t\in\R}$ of real-valued random variables, $t$ is naturally interpreted as a time parameter. However, in our case the parameter (denoted as $x$) is a spatial variable 
and $\bm{f}$ has to be considered as a random field. 
}

{With Hypothesis \ref{hyp:noise} in hand, }
we think of each discrete noise as identified with a function  $f_h$ as in  \r{2.4L} with some $\hat \chi\in C_0^\infty$  without necessarily assuming \r{2.5} for now. Clearly, $\mathbb{E}(f_h)=0$, which is a temporal characteristic. 
We now state a lemma for the spatial mean and variance of $f_{h}$.

\begin{lemma}\label{lem:as-convergence-variance}
Let $\{ \bm f_{k,h}; \,  k\in \mathbf{Z}^n \}$ be a noise satisfying Hypothesis \ref{hyp:noise}, and define the function $f_h$ according to  \r{2.4L}. Then $\mathbb{P}$-almost surely we have 
\be{DM1}
\MEAN (f_h) :=\frac1{|\Omega|} \int_{\Omega} f_h\,\d x\to 0, \quad \text{as $h\to0$}.
\ee
As far as the spatial variance of $f_{h}$ is concerned, we get the following $\mathbb{P}$-almost sure limit,
\be{DM20}
\VAR_\Omega(f_h):= \frac1{|\Omega|}\int_\Omega f^2_h\,\d x 
\to \sigma^{2}, \quad \text{as $h\to0$}.
\ee
\end{lemma}

\begin{proof}
We will only prove \r{DM20}, the proof of \r{DM1} being similar. To this aim, starting from~\r{2.4L} and using the fact that $\{\chi_{k};\, k\in K(h)\}$ is an orthogonal system we get
\begin{equation}\label{b1}
\|f_{h}\|^{2}
=
\sum_{k\in K(h)} \bm{f}_{k,h}^2  \|\chi_{k}\|^{2}
=
c_{\chi} (sh)^{n} \sum_{k\in K(h)} \bm{f}_{k,h}^2  ,
\end{equation}
where $c_{\chi}=\|\chi\|^{2}$. Plugging \r{b1} into the definition \r{DM20} of $\VAR_\Omega(f_h)$, we obtain
\begin{equation}\label{b2}
\VAR_\Omega(f_h)
=
c_{\chi} \frac{(sh)^{n}}{|\Omega|} \sum_{k\in K(h)} \bm{f}_{k,h}^2  .
\end{equation}
Taking limits in \r{b2} now amounts to applying an almost sure limit theorem for the triangular array $\{\bm{f}_{k,h}^{2}; \,k\in K(h), h >0\}$. 
{This is ensured by the relation $\text{Card}(K(h))=|\Omega|(sh)^{-n}(1+o(h))$ and classical theorems on strong law of large numbers for triangular arrays (see e.g \cite[Corollary on p. 378]{hu-taylor97}), as soon as the random variables $\bm{f}_{k,h}^{2}$ satisfy Hypothesis~\ref{hyp:noise}.} 
The proof of our claim~\r{DM20} is now easily achieved. 
\end{proof}

By \r{DM20}, $f_h$ is $L^2$ bounded almost surely, therefore it almost surely has a microlocal defect measure (possibly not unique) associated to it. 
In this paper, we consider  every such \sc\ defect measures $\d\mu_f(x,\xi)\ge0$, defined in Section~\ref{sec_measures}, as a spectral density of $f_h$. In Theorem~\ref{thm_m} below however, we show that the limit is unique and it holds for every sequence $h\to0$ in the case we consider. 

One can see that $\d\mu_f$ makes sense as the variance density in the phase space. 
In fact for a domain $\Omega$, the quantity
\be{VAR}
\VAR_\Omega^0(f):= \frac1{|\Omega|}\iint_{T^*\Omega} \d\mu_f
\ee
corresponds formally to $p$   being the characteristic function of $\Omega$, divided by $|\Omega|$, which would correspond to the usual variance definition if $\lim_{h\to0} f_h$ existed. We are not claiming that the latter limit exists however but when $f$ is white noise, the defect measure exists as a limit in mean square sense, as we prove in Theorem~\ref{thm_m} below.  
The superscript $0$ in \r{VAR} is a reminder that this is a quantity in the limit $h\to0$. 
We want to emphasize that   $\VAR^0$ is just defined by \r{VAR} and \r{9b} below for any $f$ for which $\d\mu_f$ exists and it is not necessarily connected to any random $f$. When $f$ is random (noise), $\VAR^0(f)$ is related to it as in Theorem~\ref{thm_m} below. 
We define the standard deviation $\STD(f)$ as the square root of the variance $\VAR(f)$ (with or without the superscript $0$). 

Assume now 
\be{9aa}
\d\mu_f=\gamma_f\,\d x\,\d\xi
\ee
with some continuous $\gamma_f\ge0$. Then taking the limit as $\Omega$ converges to a point, we set
\be{9b}
\VAR_x^0(f): = \int  \gamma_f(x,\xi)\, \d\xi .
\ee
Hence $\VAR_x^0(f)$ can be viewed as the {asymptotic} variance density of the noise at $x$. 

\subsection{A remark about the Wigner function}  
In this section, we will relate the Wigner function to the defect measures {at a heuristic level}.
For a noise $\bm{f}$ satisfying Hypothesis~\ref{hyp:noise}, we set 
\be{de1}
(p^\text{w}(x,hD)f_h,f_h) = \int p(x,\xi) W_f(x,\xi)\,\d x\,\d\xi,
\ee
where $W_f$ is the Wigner function, see \cite{deVerdiere2011semi},
\[
W_f^h(x,\xi)= (2\pi h)^{-n} \int e^{-\i z\cdot\xi/h } f_h(x+z/2) \bar f_h(x-z/2)\,\d z.
\]
Note that $W_f^h\,\d x\,\d\xi$ is $h$-dependent and not a measure in general since it may take negative values. However, the existence theorem of defect measures says that there exits at least one sequence $h_j\to0$ for which $W_f^h$   converges to some $\d\mu$. Moreover, we have
\be{de2a}
\int W_f^h(x,\xi)\,\d \xi= |f(x)|^2 ,\qquad \int W_f^h(x,\xi)\,\d x=(2\pi h)^{-n} |\mathcal{F}_h f(\xi)|^2.
\ee
In \cite{deVerdiere2011semi}, de Verdi\`ere considers random vector fields $f(x)$, $x\in\R^n$, and defines their auto-correlation by
\[
\acor_f(x,y)=\mathbb{E}(f(x)\bar f(y))
\]
Then he defines the power spectrum of $f$ by
\[
P_h(x,\xi)= \mathbb{E}(W_f^h(x,\xi)).
\]
This lifts the notion of power spectrum to the phase space but the limit $h\to0$ is not taken. 

Following the steps of the forthcoming Theorem~\ref{thm_m} and using crucially the fact that $\mathbb{E}(\bm f_k \bm f_l)=\sigma^2 \delta_{k,l}$, we let the patient reader check that 
\be{EQ}
\begin{split}
\mathbb{E}(p^\text{w}(x,hD)f_h,f_h) &= (sh)^n\sigma^2 \tr ( Q(h))\\
& = \frac{s^n \sigma^2 }{(2\pi)^n}\Big( \iint |\hat\chi(s\xi)|^2 p(x,\xi)\,\d x\,\d\xi+O(h) \Big) ,
\end{split}
\ee
where $Q$ is defined by  \r{Q}. 
Thanks to \r{de1}, this leads to the expected value of the Wigner function $W_f^h$ up to an $O(h)$ error in a weak sense; and eventually, it could lead to the  expected value of the defect measure, {if we can take limits as $h\to 0$ in any reasonable probabilistic 
sense. 
There are several difficulties with this approach. We have to treat and estimate the remainder as a measure applied to $p$; different subsequences $h_j$ could converge to different defect measures for a fixed $\bm f_k$ while the expected value applies to all such sequences, etc. The latter is the  important reason we do not pursue this approach. 
In addition, the Wigner function method characterizes the power spectrum of the noise after repeated experiments (in temporal sense), while we want to study a single one (in ergodic sense).

\subsection{The defect measure of white noise}
Let $\bm f_k$, $k\in\mathbf{Z}^n$ have values in $\mathbf{R}$. As before, $\Omega\subset \R^n$ is a bounded domain. 
In the theorem below, given $h>0$, we associate a \sc ly band limited function $f_h$ to $\{\bm f_k\}$ by \r{2.4L}. This uses $|\Omega|(sh)^{-n}(1+o(1))$ terms of the sequence $\bm f_k$. We allow $\{\bm f_k\}$ to depend on $h$. Then we get a triangular array of random variables.

The following theorem is the main technical result of this paper. 

\begin{theorem}\label{thm_m} 
Assume  that $\{ \bm f_{k,h}; \,  k\in \mathbf{Z}^n \}$ is a noise satisfying Hypothesis \ref{hyp:noise}, with $L^{4}$ moments only. Namely the random variables  
$\bm f_k$, $k\in\mathbf{Z}^n$ take values in $\mathbf{R}$   and
are created by a white noise process with variance $\sigma^2>0$ and a bounded fourth moment. 

(a)  Let $f_h^\delta$ be the associated distribution given by \r{2.4delta} with some fixed $s>0$.   
Then  for every $p\in C_0^\infty(T^*\Omega)$,  
\be{de6adelta}
\left( p^{\rm w}(x,hD)f_h^\delta,f_h^\delta \right)_{L^2} \longrightarrow \int  p(x,\xi)\, \d\mu_{f^\delta}(x,\xi), \quad \text{as $h\to0+$ in mean square sense},
\ee
where 
\be{de6bdelta}
\d\mu_{f^\delta}(x,\xi) = \sigma^2s^n\frac{\d x\,\d\xi}{(2\pi)^n}  . 
\ee

(b) Let $f_h$ be the associated function given by \r{2.4L} with some fixed $s>0$ and with $\hat\chi\in C_0^\infty$ not necessarily satisfying \r{2.5}.  
Then  for every $p\in C_0^\infty(T^*\Omega)$,  
\be{de6a}
\left( p^{\rm w}(x,hD)f_h,f_h \right)_{L^2} \longrightarrow \int  p(x,\xi)\, \d\mu_f(x,\xi), \quad \text{as $h\to0+$ in mean square sense},
\ee
where 
\be{de6b}
\d\mu_f(x,\xi) = \sigma^2s^n  |\hat\chi(s\xi)|^2\frac{\d x\,\d\xi}{(2\pi)^n}.
\ee
\end{theorem}

\begin{proof}
Notice first that the l.h.s.\ of \r{de6adelta} is well-defined in distribution sense since the Schwartz kernel of  $p^{\rm w}(x,hD)$, see \r{Schw}, is Schwartz class. Let $\hat \chi\in C_0^\infty$ be such that $p(x,\xi)\chi(\xi)=p(x,\xi)$. Since we use the Weyl quantization, it is easy to see that $f_h^\delta$ can be replaced by $\psi_h*f_h^\delta$ as in \r{sec_delta}; which is \r{2.4}. 
Therefore, we need to prove (b) only. 

We start with the easier case when \r{2.5} is satisfied (with $B<\pi/s$). 
This corresponds to the practical situation of restoring an oversampled function with white noise added, and the theorem studies how the noise is added to the result.  
 
Recall that the functions $\sinc_k$ were defined in \r{sinc_k} and that $\phi_k= (sh)^{-n/2}\sinc_k$ 
 form an orthonormal basis in the space $\textbf{1}_{[-\pi/s,\pi/s]^n}(hD)L^2(\R^n)$, as mentioned earlier. The  interpolation function $\chi$ satisfies $\hat\chi \, \textbf{1}_{[-\pi,\pi]^n}=\hat\chi$ by \r{2.5}, therefore, 
\be{de7}
\chi_k= \hat \chi(shD) \sinc_k= (sh)^{n/2}\hat \chi(shD)\phi_k .
\ee
Since   $\Omega\supset\supp_x p$, we have
\begin{equation}\label{c1}
(p^\text{w}(x,hD)f_h,f_h)= \sum_{k,l\in K(h)^2} \bm f_k\bm f_l (p^\text{w}(x,hD)\chi_k,\chi_l   )
 =  \sum_{k,l\in K(h)^2}p_{kl}\bm f_k\bm f_l,
\end{equation}
where, as before,  $K(h)=\{k\in \mathbf{Z}^n,\, shk\in \Omega\}$, $K^2=K\times K$,   and
\be{pkl}
p_{kl} :=  (p^\text{w}(x,hD)\chi_k,\chi_l) = (sh)^n (Q\phi_k,\phi_l) ,
\ee
with  
\be{Q}
Q(h):= \bar{\hat\chi}(shD)  p^\text{w}(x,hD)\hat\chi(shD).
\ee
We shall prove in Lemma~\ref{lemma1} that
$|p_{kl}|\le C(sh)^{n}$. Our aim in~\eqref{de8a} is to prove that in the $L^{2}(\mathbb{X})$ sense we have 
\be{c2}
\lim_{h\to 0}\left( p^{\rm w}(x,hD)f_h,f_h \right)_{L^2} 
=
\frac{s^n}{(2\pi)^n}\sigma^2 \int  p(x,\xi)\, \d\mu_f(x,\xi).
\ee
We now split the proof of \r{c2} in several steps.

\smallskip

\noindent
\textit{Step 1: A decomposition:}
Split the summation in \r{c1} over elements $(k,l)$ on the diagonal $\Delta:= \{k=l\}$ and away from it: 
\be{c3}
(p^\text{w}(x,hD)f_h,f_h)_{L^2}  = W_1+W_2,   
\ee
where
\be{W12}
W_1:= \sum_{k\in K(h)}p_{kk}\bm f_k^2,\quad
W_2:= \sum_{k,l\in K(h)^2\setminus \Delta}p_{kl}\bm f_k\bm f_l.
\ee
Furthermore, according to \r{de8a} below we have
\begin{equation}\label{pkk}
\sum_{k\in K(h)}p_{kk} = 
 \frac{s^n}{(2\pi)^n}\int q(x,\xi)\,\d x\,\d\xi .
\end{equation}
{Thus owing to the fact that $q=p+O(h)$,} we can recast \r{c3} as 
\be{c4}
(p^\text{w}(x,hD)f_h,f_h)_{L^2}  -  \frac{s^n}{(2\pi)^n}\sigma^2 \int  p(x,\xi)\, \d\mu_f(x,\xi)
= 
W_{1,0}+W_2,
\ee
where the term $W_{1,0}$ is defined by 
\begin{equation*} 
W_{1,0} = \sum_{k\in K(h)} \bm (\bm f_k^2-\sigma^2) \, p_{kk} .
\end{equation*}
We are now reduced to prove that both $W_{1,0}$ and $W_{2}$ in \r{c4} converge to 0 in $L^{2}(\mathbb{X})$.

\smallskip

\noindent
\textit{Step 2: Analysis of $W_{1,0}$:}
Observe that the random variables $\bm f_k^2-\sigma^2$ are independent, have zero expectation and a finite variance $\tilde\sigma^2= \mathbb{E}\bm (f_k^4)-\sigma^4$ under our fourth moment assumptions. 
Then $\mathbb{E}(W_{1,0})=0$. Moreover, invoking the forthcoming inequality \r{de8a0} and the fact that $\text{Card}(K(h))\leq c |\Omega| (sh)^{-n}$, we get
\be{c51}
\mathbb{E}(W_{1,0}^2)=\sum_{k\in K(h)} \bm \tilde\sigma^2p_{kk}^2
\le
(sh)^{2n} \tilde\sigma^2 |\Omega| (sh)^{-n}
\le 
C h^n.
\ee
Therefore, $W_{1,0}$ converges to $0$ as $h\to0$, in the $L^{2}(\mathbb{X})$ sense.

\smallskip

\noindent
\textit{Step 3: Analysis of $W_{2}$:}
The random variables $\bm f_k\bm f_l$, $k\not=l$,  have expected values zero and variance $\sigma^4$. Next, $\bm f_k\bm f_l$  and $\bm f_{k'}\bm f_{l'}$
are not independent unless neither $k'$ nor $l'$ are equal to $k$ or $l$ but they are uncorrelated.  Indeed, we only need to check that when, say $k=k'$ and even then, $\mathbb{E}\left( (\bm f_k \bm f_l) (\bm f_{k} \bm f_{l'})\right)=\mathbb{E} (\bm f_k^2)  \mathbb{E}(\bm f_l) \mathbb{E} (\bm f_{l'}) =0$ because all $\bm f_k$ have expectation zero. 
Therefore some elementary $L^{2}(\mathbb{X})$ considerations, together with \r{de8a0}, reveal that
\be{c6}
\mathbb{E}(W_2^2)=\sigma^4\sum_{k,l\in K(h)^2\setminus \Delta}p_{kl}^2\le Ch^{n}. 
\ee
Therefore, $W_2\to0$ in mean square sense.

Summarizing our considerations so far,  the proof of the case when \r{2.5} is  easily achieved by plugging 
\r{c51} and \r{c6} into \r{c4}.

\smallskip

\noindent
\textit{Step 4: Dropping  the assumption \r{2.5}.} 
Let $m$ be such  that $\supp\hat\chi \subset (-m\pi, m\pi)$. Let $\sinc_k^{(m)}$ be as in \r{2.8a}. Then \r{de7} takes the form, see also \r{pr_eq2},
\be{de12}
\chi_k = m^n\hat \chi(shD) \sinc_k^{(m)}= m^n(sh/m)^{n/2}\hat \chi(shD)\phi_k^{(m)} = (shm)^{n/2}\hat \chi(shD)\phi_k^{(m)}.
\ee
The necessary modifications of the proof above in this case are as follows. For the  deterministic term  featuring in \r{W12} we have the same formula but now,
\[
p_{kl} :=  (p^\text{w}(x,hD)\chi_k,\chi_l) = (shm)^n (Q\phi_k^{(m)},\phi_l^{(m)}) .
\]
The set $\{\phi_k^{(m)}\}$ is an orthonormal system in $\textbf{1}_{[-m\pi/s,m\pi/s]^n}(hD)L^2(\R^n)$ but not a basis, see Remark~\ref{rem_1}. The missing elements are those with fractional indices in $\mathbf{Z}^n/m$. Then    there are many ``gaps'' in the sum $W_{1,0}$ compared to the one with a basis, giving us a trace as in Lemma~\ref{lemma1}. On the other hand, the extra factor $m^n$ in \r{de12} allow us to think of each term $m^np_{kk}$ as an approximation of all $m^n$ terms in a box around $k$ of size one, which would add the missing terms. The error is $O(h^{n+1})$ (multiplied by the constant $m^n$), by \r{de10}.  Since $K(h)/m$ has $O((m/h)^{-n})$ points, this introduces an $O(h)$ error, thus \r{de6a} is preserved.
\end{proof}

The following lemma was used in the proof above. Below, $\|\cdot\|_{\rm HS}$ stands for the Hilbert-Schmidt norm. 

\begin{lemma}\label{lemma1} 
For $p_{kl}$ defined by \r{pkl}, we have
\begin{align} \label{de8a0}
|p_{kl}|&\le C(sh)^n,\\ \label{de8a}
\sum_{k}p_{kk } &= (sh)^n \tr Q = \frac{s^n}{(2\pi)^n}\int q(x,\xi)\,\d x\,\d\xi,\\
\sum_{k,l}|p_{kl}|^2 &= (sh)^{2n} \|Q\|_{\rm HS}^2 = \frac{s^{2n}h^n }{(2\pi)^n}\int |q(x,\xi)|^2 \,\d x\,\d\xi, \label{de8b}
\end{align}
where $q$ is the complete symbol of the \HPDO\ $Q$ in \r{Q}. Next,
\be{de10}
\begin{split}
p_{kl}&= (p^\textnormal{w}(x,hD)\chi_k,\chi_l)\\
&= s^{2n}h^{n}\iint  \check p\Big(\frac{sh}2(x+y+k+l),s (x-y+k-l)\Big)\chi(x)\chi(y)\,\d x\, \d y,
\end{split}
\ee
where $\check p$ is the inverse Fourier transform of $p$ w.r.t.\ $\xi$
\end{lemma}

\begin{proof}
Inequality \r{de8a0} follows directly from the fact that $\|P(h)\|$ is bounded uniformly in $h$, see, e.g., \cite[Theorem~4.21]{Zworski_book}. 
If we add the basis elements of     $(\Id-\textbf{1}_{[-\pi/s,\pi/s]^n}(hD)) L^2(\R^n)$   to the $\phi_k$ terms in \r{pkl}, we will get zero contribution, so we consider it done.  Then the first equality in \r{de8a} follows by the definition of a trace. The second part follows from \cite[Ch.~9]{DimassiSj_book}.

To prove \r{de8b}, write
\be{L1a}
\|Q\|_{\rm HS}^2 = \tr(Q^*Q)= \sum_k \|Q\phi_k\|^2= \sum_{k,l} |(Q\phi_k,\phi_l  )|^2= (sh)^{-2n}  \sum_{k,l}|p_{kl}|^2,
\ee
see also the proof of \cite[Theorem~VI.23]{Reed-Simon3}. This proves the first part of \r{de8b}. For the second part, notice that by \cite[Ch.~9]{DimassiSj_book} again, the Hilbert-Schmidt norm of a classical \PDO\ $R:=r(x,D)$ is given by
\[
\|R\|_{\rm HS}^2= \frac1{(2\pi)^n} \int |r(x,\xi)|^2\,\d x\,\d\xi. 
\]
We can turn $R$ into a classical \PDO\ by setting formally $r(x,\xi)= q(x,h\xi)$ to get 
\[
\|Q(h)\|_{\rm HS}^2= \frac1{(2\pi h)^n} \int |q(x,\xi)|^2\,\d x\,\d\xi. 
\]
Combining this with \r{L1a}, we complete the proof of \r{de8b} as well. 

 Finally, the Schwartz kernel of $p^{\rm w}(x,hD)$ is given by
\be{Schw}
h^{-n} \check p((x+y)/2, (x-y)/h),
\ee
and $\check p$   is in the Schwartz class. 
Then 
\[
(p^\text{w}(x,hD)\chi_k,\chi_l)= h^{-n}\iint  \check p\Big(\frac{x+y}2, \frac{x-y}h\Big)\chi\Big(\frac1{sh}(x-shk)\Big)\chi\Big(\frac1{sh}(y-shl)\Big)\,\d x\, \d y.
\]
Make the change of variables $\tilde x= (x-shk)/(sh)$, $\tilde y= (x-shl)/(sh)$; then 
$x=sh(\tilde x+k)$, $y=sh(\tilde y+l)$
 to get \r{de10}. 
\end{proof}

\begin{remark}\label{rem_std}\ 

(a) The presence of the parameter $s$ in \r{de6b} is to be expected. The random sequence $\bm f_k$ is not related to any distance scale, while $sh$ is the distance between two adjacent points on the sampling grid after we associate  $\bm f_k$ to $f_h$. Then $s$ reflects the choice of that scale.

(b) For every $x$, we have, see \r{VAR},
\be{rem_std_eq}
\VAR_x^0(f) = \int\gamma_f(x,\xi)\,\d\xi = \frac{\sigma^2}{(2\pi)^n} \|\hat\chi\|^2= \sigma^2\|\chi\|^2,
\ee
in mean square sense, see also \r{de2a}. In particular, if $\chi$ is a product of sinc functions, we get $\sigma^2$, i.e., $f_h$ has the same variance as that of  $\bm f_k$, in a limit. If $\chi=\text{LAN}3$, then $\|\chi\|^2\approx 0.888$ in one dimension. In dimension $n$, we have a product of such $\chi(x_j)$'s, then the factor would be $\|\chi\|^{2n}$ instead,  therefore, $\STD^0_x(f)\approx 0.94^n\sigma$. 
Note that there is no dependence on $s$ here.  For the linear interpolation, $\|\chi\|^2= 2/3$,  therefore, $\STD_x^0(f)= (2/3)^{n/2}\sigma\approx 0.816^n\sigma$. All those equalities are mean square limits in the sense of the theorem. 

(c) If we are interested in the expected value of the variance in repeated experiments,  the equivalent of \r{rem_std_eq}   is easy to get. We can think of $f_h$ as a linear operator, say $\Psi$, applied to $\bm f=\{\bm f_k\}$, i.e., $f_h=\Psi\bm f$. Then 
\[
\mathbf{E}(\|f_h\|^2) = \mathbf{E}(\Psi^*\Psi\bm f,\bm f)=\sigma^2 \tr (\Psi^*\Psi)= \sigma^2 \|\Psi\|_{\rm HS}^2,
\]
where the latter norm is the Hilbert-Schmidt one. 
Then the equivalent of \r{rem_std_eq}  can be derived from this formula. That requires repeated experiments however. 

(d) The variance \r{VAR} is like the l.h.s.\ of \r{de6a} with $p$ being the characteristic function of $\Omega$ divided by its volume. The theorem requires $p$ to be smooth though, so we may think of \r{VAR} as an approximation of $(pf_h,f_h)$ with $p\in C_0^\infty(\Omega)$ (independent of $\xi$) approximating that normalized characteristic function. 

(e)  Theorem \ref{thm_m} says that the noisy $|\hat{f}_{h}|^2$ in \eqref{2.4L} converges in weak sense to $\frac{s^{2n}h^n }{(2\pi)^n} \sigma^{2} |\hat{\chi}(s\xi)|^2$. 

(f) We can assume that the noise is not homogeneous, for example that ${\bm f}_{k,h}$ are replaced by $\zeta(shk){\bm f}_{k,h}$ with some smooth $\zeta$. This case can be handled as explained in section~\ref{sec_7.1}, where $g=\zeta$ and the problem with $\nabla g$ described there does not exist in this case. This would introduce the extra factor $|\zeta(x)|^2$ in \r{de6b}. In principle, one can consider noise inhomogeneous in phase space, i.e., $\zeta$ being a suitably sampled \PDO\ or an \HPDO.
\end{remark}

In Figure~\ref{fig_demo_thm_t}, we present an one dimensional numerical example. In sections \ref{sec_R_p} and \ref{sec_FB} we show two-dimensional ones. We take a discrete $\bm f$ with $N=100$ components, 
upsize it to a $200$ point grid with the Lanczos3 algorithm,  and plot $|\hat{\bm f}|$, where the hat stands for the Discrete Fourier Transform, then the same quantity computed as a square root of $|\hat{\bm f}|^2$ averaged over $10^2$ and $10^5$ experiments, for frequencies in $[0,100]$.  
This illustrates \r{EQ}. The limiting profile looks very close to the profile in Figure~\ref{fig_3}, right, as expected {from our Remark (e) above}. At the right hand side of the plot, it is not as close to zero as the profile in Figure~\ref{fig_demo_thm_e}} because of the $O(1/N)$ error in \r{EQ}; here $N=100$ only.  The plot on the right is essentially the expected value of the Wigner function $W_f^h$. 

\begin{figure}[ht]
\begin{center}
	\includegraphics[trim = 0 10 0 10  ,scale=0.5]{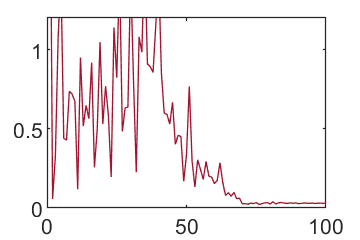}
	\includegraphics[trim = 0 10 0 10  ,scale=0.5]{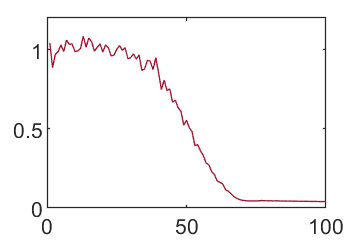}
	\includegraphics[trim = 0 10 0 10  ,scale=0.5]{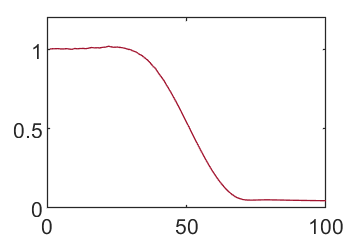}
\end{center}
\caption{\small  Plot of   $|\hat f|$ for $N=100$, with $|\hat f|^2$ averaged over $1, 10^2$, and $10^5$  experiments.}
\label{fig_demo_thm_t}
\end{figure}

In Figure~\ref{fig_demo_thm_e}, the setup is as above but we show the smoothing effect of averaging the power spectrum within a single experiment, {illustrating relation~\r{de6a}}. To this aim we consider $\bm f$ with $N=10^2$, $10^4$, and $10^6$ components. The frequency interval is divided into $25$ subintervals and averaged there, similarly to Figure~\ref{fig_power_sp}. The plot on the left is very close to the plot of the modulus of the Fourier transform of the Lanczos3 filter in  Figure~\ref{fig_3}. 
\begin{figure}[ht]
\begin{center}
	\includegraphics[trim = 0 10 0 10  ,scale=0.5]{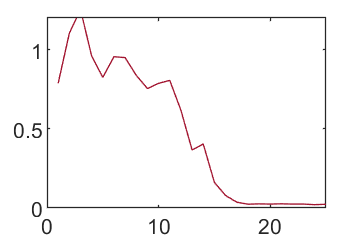}
	\includegraphics[trim = 0 10 0 10  ,scale=0.5]{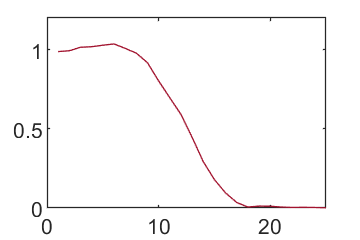}
	\includegraphics[trim = 0 10 0 10  ,scale=0.5]{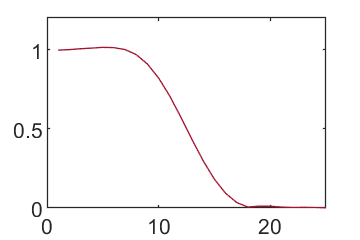}
\end{center}
\caption{\small  Plot of   $|\hat f|$ with a single experiment, for $N=10^2, 10^4, 10^6$, with averaging over $25$ subintervals.}
\label{fig_demo_thm_e}
\end{figure}

\subsection{Micorlocal defect measure of more general noise} 
We consider more general noise now. First, we assume that the random variables $\bm f_{k,h}$ might be correlated with the neighboring ones; and second, we assume that this correlation might be position dependent. Since the position of $\bm f_{k,h}$ would be at $x_k=shk$, this more general noise would be assumed to satisfy the following. 

\begin{hypothesis}\label{hyp:noise_g}
For every  $h>0$, the noise is modeled by a family $\{ \bm f_{k,h}; \,  k\in \mathbf{Z}^n \}$ of  real valued  random variables defined on the same probability space $(\mathbb{X},\mathcal{F},\mathbb{P})$ with zero expected values. They are all assumed to satisfy \r{H1} with a uniform bound. For the autocorrelation $\acor( \bm f_{k,h}, \bm f_{k+m,h})$ we assume
\be{Acor_f}
\acor( \bm f_{k,h}, \bm f_{k+m,h}) = \beta(skh,m),
\ee
where $\beta(x,k)$, $x\in\R^n$, $k\in\mathbf{Z}^n$, is smooth in $x$, and supported in a bounded set w.r.t.\ both variables.
\end{hypothesis}

Note that we are no longer requiring, in particular,  $ \bm f_{k,h}$ to have the same variance. They are not identically distributed, in general. 

Let
\be{beta}
\check \beta(x,\xi)  =\sum _m e^{\i s m\cdot\xi}\beta(x,m) 
\ee
be the inverse Fourier series of $\beta$ with respect to the $m$ variable. This is essentially the Wigner distribution related to the auto-correlation, in the limit $h\to0$.  Since $\beta(x, s(k+m)h,-m) =\beta(x, skh,m) $, we must have $\beta(x,m)=\beta(x,-m)$ for all $(x,m)$. Then \r{beta} is just a cosine series, and in particular real. The theorem above shows that it is in fact non-negative. 

The generalization of Theorem~\ref{thm_m} to this case is the following.

\begin{theorem}\label{thm_m2} 
Assume  that $\{ \bm f_{k,h}; \,  k\in \mathbf{Z}^n \}$ is a noise satisfying Hypothesis \ref{hyp:noise_g}, with $L^{4}$ moments only.  
 Let $f_h$ be the associated function 
given by \r{2.4L} with some fixed $s>0$ and with $\hat\chi\in C_0^\infty$ not necessarily satisfying \r{2.5}.  
Then  \r{de6a} remians true with 
\be{de6b_g}
\d\mu_f(x,\xi) = \frac{s^n}{(2\pi)^n}  \check\beta(x,s\xi) |\hat\chi(s\xi)|^2\,\d x\,\d\xi. 
\ee
\end{theorem}

\begin{proof}
We follow the proof of Theorem~\ref{thm_m}. We replace the diagonal $\Delta$ in it by $\Delta=\{(k,l);\; |k-l|\le M\}$, where $M$ is so that $\beta(\cdot,m)=0$ for $|m|>M$. The off--$\Delta$ terms do not contribute to the limit \r{de6a} as above. For the rest, we estimate their contribution for every fixed $m$, and then sum up the results. The analog of $W_1$ now, depending on $m$, is
\be{W1_g}
W_1= \sum_{k\in K(h)}p_{k\, k+m}\bm f_k\bm f_{k+m} = \sum_{k\in K(h)} \beta(skh,m)p_{k\, k+m}  + W_{1,0},
\ee
where 
\begin{equation*} 
W_{1,0} = \sum_{k\in K(h)}  (\bm f_k\bm f_{k+m}-\beta(skh,m)) \, p_{k\, k+m} .
\end{equation*}
The analysis of $W_{1,0} $ is similar: the random variables $\bm f_k\bm f_{k+m}-\beta(skh,m)$ have zero expectation, thus $\mathbb{E}(W_{1,0})=0$. They have a uniformly bounded variance. To estimate $\mathbb{E}(W_{1,0}^2)$, notice that only $O(m^2h^{-n})$ terms in the expansion would have a non-zero expectation; and by  \r{de8a0}, $\mathbb{E}(W_{1,0}^2)=O(h^n)$ again. It remains to compute the $\beta$  term in \r{W1_g}. 

Recall the definition \r{pkl} of $p_{kl}$. With $l=k+m$ there, an easy calculation shows that $Q\phi_{k+m} = q_m(x,hD)\phi_k$ for any \HPDO\ $Q=q(x,hD)$, with $q_m(x,\xi)= e^{\i s m\cdot\xi}q(x+shm,\xi)$ (which is a symbol as well, notice that there is no $h$ in the phase). The principal symbol of that is just $e^{\i s m\cdot\xi}q(x,\xi)$. 
Then the $\beta$  term in \r{W1_g} takes the form
\[
(sh)^n \sum_{k\in K(h)} ( \beta(skh,m) \phi_k,q_m(x,hD)\phi_k) .
\]
By the properties of $\phi_k$, recall \r{2.4L}, where $\hat\chi\in C_0^\infty$, replacing $\beta(skh,m)$ above with $\beta(x,m)$ would result in an $O(sh)$ error in each term, and a total error $O(h)$. Considering this done, and moving the $\beta$ factor to the right, we get a quadratic form with $q$ multiplied by $\beta(x,m)$:
\[
(sh)^n \sum_{k\in K(h)} (   \phi_k,\tilde q_m(x,hD)\phi_k),
\]
where $\tilde q_m(x,\xi) = e^{\i s m\cdot\xi}\beta(x,m) q(x,\xi)$. 

So far, $m$ was fixed. Summing over $m$ (the number of those terms is $2M+1$), we get to the situation of the proof of Theorem~\ref{thm_m} with $q$ replaced by
\[
\sum _m e^{\i s m\cdot\xi}\beta(x,m) q(x,\xi)=\check\beta (x, s\xi) q(x,\xi).
\]
The theorem then follows as in the proof of Theorem~\ref{thm_m}.
\end{proof}

\subsection{Spectral density under an FIO} 
We want to find out how a spectral density transforms under an action of a classical FIO of order $m$. It is easier to answer this question for \sc\ FIOs since the defect measures are a \sc\ object, and we will reduce the classical case to the \sc\ one.

\begin{theorem}\label{thm_dmu} 
Let $A$ be a classical FIO of order $m$ on $\R^n$ with a homogeneous principal symbol associated with a canonical relation which is a graph of a local diffeomorphism $\kappa$. Let $f=f_h$ be \sc ly band limited and uniformly bounded in $L^2$. 
Then for every defect measure $\d\mu_f$ given as the limit \r{9} for some $h=h_j\to0$, the defect measure $\d\mu_{h^m Af}$ associated to the same sequence $h_j$ exists as well and it satisfies
\[
\d\mu_{h^m Af}=  {\kappa^{-1}}^*(b\,\d\mu_f) \quad \text{on $T^*\R^n\setminus 0$},
\]
where $b$ is the (classical) principal symbol of $A^*A$.
\end{theorem}

\begin{proof}
By \r{9}, 
\be{10}
\begin{split}
\int  p(x,\xi)\,\d\mu_{h^mAf} &=  \lim_{h=h_j\to 0}\left( p(x,hD)h^m Af_h, h^m Af_h \right)_{L^2} \\
&=  \lim_{h=h_j\to 0}\left(h^m A^* p(x,hD)h^m Af_h,f_h \right)_{L^2} .
\end{split}
\ee 
Since we need to find $\d\mu_{h^mAf}$ away from the zero section, it is enough to assume that $p=0$ near $\xi=0$. 

If for a moment we ignore the need to cut near $\xi=0$, then we can think of $A$ as in \r{A00} as an  $h$-FIO with symbol $a(x,\xi/h)=h^{-m}a(x,\xi)$ for $|\xi|\gg1$. Then by the \sc\ Egorov's theorem \cite[Theorem~5.5.5]{Martinez_book}, which an analog of the classical one, 
 (Theorem~25.3.5 in \cite{Hormander4}), we would get
\be{11}
\left(h^m A^* p(x,hD)h^m Af_h, f_h \right)_{L^2} = \left(Q  f_h,f_h \right)_{L^2} ,
\ee
where $Q$ is an \HPDO\ with a principal symbol $b(p\circ \kappa) $, where $b$ is the (classical) principal symbol of $A^*A$ and $\kappa$ is the canonical relation (as a map) of $A$. Note that the canonical relations of $A$ and its \sc\ version after the change $\xi\mapsto\xi/h$ are the same.

To deal with the fact that we have a classical FIO and a \sc\ \PDO, we apply Theorem~\ref{thm_c-sc}. Let $A=A_{h,\eps}+ R_{h,\eps}$ be as in \r{AR}. 
For $\eps\ll1$, the remainder $R_{h,\eps}$ would contribute an $O(h^\infty)$ error to \r{10} if we replace $A$ there by $A_{h,\eps}$ because $p=0$ near $\xi=0$. Therefore, we can consider this done. Then $A_{h,\eps}$ is an $h$-FIO, see \r{AR} with symbol $\tilde a:=a(x,\eta/h) (1-\psi(\eta/\eps))\in h^{-m}S^0$ supported where $\eta\ge\eps$. On the support, $|\eta|/h\ge\eps/h$, and there, $a(x,\eta/h)$ is homogeneous for $h\ll1$; therefore $\tilde a=h^{-m}a(x,\eta) (1-\psi(\eta/\eps))$

Then we can apply the \sc\ version of Egorov's theorem \cite[Theorem~5.5.5]{Martinez_book}. For that, we need to compare the principal symbol of the \HPDO\ $A_{h,\eps}^*A_{h,\eps}$ to that of the classical \PDO\ $A^*A$ and see how the cutoff $(1-\psi(\eta/\eps))$ near the zero section affects that.

The principal symbol of $A_{h,\eps}^*A_{h,\eps}$ is given by
\[
c(x,\xi,h) = \left|\tilde a(\pi_1\circ\kappa(x,\xi)), \xi,h) \right|^2 J(x,\xi),
\]
where $\pi_1$ is the projection on the fist variable, and $J>0$ is a smooth Jacobian, homogeneous of order zero w.r.t.\ $\xi$,  depending on the phase function only. For $|\xi|>2\eps$ we have $\tilde a=  h^{-m}a(x,\eta)$; therefore  
\[
c(x,\xi,h) = h^{-2m}\left| a(\pi_1\circ\kappa(x,\xi)), \xi) \right|^2 J(x,\xi),\quad |\xi|\ge2\eps. 
\]
This is the principal symbol of $A^*A$ as a classical \PDO\ as well without the factor $h^{-2m}$. 
Therefore, the limit of \r{11}, as $h=h_j\to0$, would be
\[
\int  b(p\circ \kappa)  (x,\xi)\, \d\mu_{f}
\]
as long as $p=0$ for $|\xi|\le 2\eps$. 
Make the change of variables $\kappa(x,\xi)=(y,\eta)$, and using the fact that $\kappa$ is symplectic, in particular an isometry, we would get
\[
\begin{split}
\int  p(x,\xi)\,\d\mu_{h^m Af} = \int  p(x,\xi)   {\kappa^{-1}}^*(b\d\mu_{f}),
\end{split}
\]
when $p=0$ for $|\xi|\le 2\eps$, 
where ${\kappa^{-1}}^*$ is the pull-back under $\kappa^{-1}$. Since  $\eps>0$ is arbitrary, this holds when $0\not\in \supp_\xi p$. Then 
\be{12aa}
\d\mu_{h^m Af}=  {\kappa^{-1}}^*(b\d\mu_f).  
\ee
So far $A$ was microlocalized near pair of points, where $\kappa$ is a (global) diffeomorphism. Since it is only a local one, we can do the same for each branch, and add the results. Then $b$ would be the principal symbol of $A^*A$ with all branches combined, as stated. 
\end{proof}

Note that in particular, if \r{9aa} holds, then ${\kappa^{-1}}^*(\gamma_f \,\d x\,\d\xi) = \gamma_g\circ \kappa^{-1} \,\d y\,\d\eta$. 

\begin{remark}\label{remark_mu}
The proof also implies that if $Q=q(x,hD)$ and $R=r(x,hD)$ are \HPDO s, then 
\be{Rem4.2} 
\d\mu_{h^m QARf}=  |q|^2{\kappa^{-1}}^*(b|r|^2\,\d\mu_f) \quad \text{on $T^*\R^n\setminus 0$},
\ee
{where $b$ still denotes the principal symbol of $A^*A$.}
\end{remark}

\begin{example}
Take $R=r(x)$  (i.e., a multiplication) with $r$ smooth. Then, up to $O(h)$, equality \r{2.4L} for $rf_h$ takes a similar form but now $\bm f_{k,h}$ are replaced by   $r(shk)\bm f_{k,h}$. This is an example of non-homogeneous noise, depending on the position, for which Theorem~\ref{thm_m} applies but then the measure is as in \r{Rem4.2}.  
\end{example}

\begin{example}
Let $R$ be a convolution with $h^{-n} \psi(x/h)$ with some $\psi\in C_0^\infty$. This is an \HPDO\ with symbol $\check \psi(\xi)$, therefore we get the factor $|r|^2=|\check \psi(\xi)|^2$ in \r{Rem4.2}.  An elementary computation shows that, up to $O(h)$, $Rf_h$ is obtained from $\tilde{\bm f}_{k,h} = \sum_m\psi(s(k-m)) {\bm f}_{m,h}$. Those are correlated (in general) random variables. They model sensors with cross-talk. Then Theorem~\ref{thm_m} applies with the measure is as in \r{Rem4.2}.  
\end{example}

Both examples are covered by Theorem~\ref{thm_m2} as well if you think of $Rf$ as $f$ but generated by correlated noise $\bm f_{k,h}$.  

\subsection{Back to the inverse problem} 
We return to the inverse problem \r{1} now. Let  $A$ be an FIO as in Theorem~\ref{thm_dmu}, and elliptic.  More precisely, let $\Omega\subset\R^n$ be a bounded domain, and let $\Omega'$ be another such domain so that the canonical relation $\kappa$ of $A$ maps $T^*\Omega$ into $T^*\Omega'$. By a compactness argument, if $A$ is defined first as $A:\mathcal{E}'(\Omega)\to \mathcal{E}'(\R^n)$, then the range of $\kappa$ projected to its base variable is a bounded set, thus such an $\Omega'$ exists. Outside $\bar\Omega'$, the image of $A$ is smooth. The measurement $g$, supposedly equal to $Af$ for some $\mathcal{E}'(\Omega)$ but corrupted by noise, is a function defined in $\Omega'$. 
 Then \r{1} is microlocally solvable: $f=A^{-1}g$ (we do not have problems with $g$ not being in the range because $A^{-1}$ is a parametrix) and we are in the situation above with $A$ replaced by $A^{-1}$.  The added noise is given by \r{5}. Dropping the subscript ``noise'' as we already did, we assume that $g$ is given first as discrete noise $\{\bm g_k\}$ and then converted to a \sc ly band limited function $g$ as in \r{2.4L}. 
Then
\[
A^{-1}g= \sum_{shk\in\Omega'} \bm g_k A^{-1}\chi_k.
\]
We have not defined what noise is but we can think of this as noise because it is a linear combination of $\{A^{-1}\chi_k\}$ with random coefficients. It has zero mean in the sense of \r{DM1}. Then
\be{13}
\d\mu_{h^{-m} A^{-1}g}=  \kappa^*(b^{-1}\d\mu_{g}) \quad \text{on $T^*\R^n\setminus 0$},
\ee
where $\kappa$ is the canonical relation of $A$ and $b$ is the principal symbol of $AA^*$. By Egorov's theorem again applied to the operator $A^* (AA^*)A=(A^*A)^2$, the principal symbol of it is that of $A^*A$ multiplied by $b\circ \kappa$. Therefore, $b\circ \kappa$ is the principal symbol of $A^*A$. 

The defect measure \r{13} then describes the power spectrum of the noise in the reconstruction \textit{away from the zero section $\xi=0$}. We cannot expect to get an estimate near the zero section in this case since $A$ may not be even injective. 
For example, the interior region of interest problem for the Radon transform in the plane has no unique solution and the practical solution is a parametrix. Then every element in the kernel would be smooth and could be considered as noise with zero frequency.  

Next theorem is a direct consequence of \r{13}. The operator $Q$ is needed to cut the  zero section, and $R$ is a filter which we may want to apply to the data, see also next section. Below, $\sigma_p(Q)$ stands for the principal symbol of $Q$. 
\begin{theorem}\label{thm_IP}
Let $A$ be as above, and elliptic, and let $g=g_h$ be \sc ly band limited with $\WFH(g)\subset T^*\Omega'$,  uniformly bounded in $L^2(\Omega')$. 
If $R=r(x,hD)$ is any \HPDO\ in $\Omega'$ with an $h$-independent symbol, and if $Q=q(x,hD)$  is a similar \HPDO\ in $\Omega$ with $q=0$  near the zero section, then 
\be{thm_IP1}
\begin{split}
\VAR^0_\Omega(Qh^{-m}A^{-1}Rg)&= \frac1{|\Omega|} \int_{T^*\Omega} |q|^2 \sigma_p(A^*A)^{-1} \kappa^*\left(|r|^2 \,\d\mu_g\right)\\
& = \frac1{|\Omega|} \int_{T^*\Omega'}  |q\circ \kappa^{-1}|^2    \sigma_p(AA^*)^{-1} |r|^2\, \d\mu_g 
\end{split}
\ee
for every $g$ (called there $f$) as in Theorem~\ref{thm_dmu}. 
\end{theorem}

\begin{proof}
By Remark~\ref{remark_mu} about Theorem~\ref{thm_dmu} and \r{VAR}, 
\[
\VAR^0_\Omega(Qh^{-m}A^{-1}Rg)= \frac1{|\Omega|} \int_{T^*\Omega}  |q|^2 \kappa^*\left(   b^{-1}|r|^2 \d\mu_g\right).
\]
Make the change of variables $(y,\eta)=\kappa(x,\xi)$, where $(y,\eta)$ are the variables in the phase space of $g$, using the fact that $\kappa$ is symplectic, and therefore an isometry, to get the second equality of the theorem.
\end{proof}

A typical use of this theorem is to take $q$ to cut off smoothly a small neighborhood of the zero section. Then, for $g$ being white noise, for example, the effect of that on the r.h.s.\ would be small. Then if we formally take $q=1$, hence $Q=\Id$, we get a good approximation of the variance of the noise in the reconstruction away from the zero frequency noise, by Theorem~\ref{thm_m}. The operator $R$ plays a role of a filter before the inversion.

We want to emphasize that $g$ in Theorem~\ref{thm_IP} does not need to be white noise; we just need a well-defined $\d\mu_g$, which is the case for noise satisfying Hypothesis~\ref{hyp:noise_g},  by Theorem~\ref{thm_m2}. 

\begin{remark}\label{rem_R}
In some situations, like in the next two sections, the requirement $q=0$ near the zero section can be removed, and the whole operator $Q$ can be removed (replaced by $\Id$). Assume that the filter $r$ is compactly supported in the dual variable. Since we deal with \sc ly band limited $g$, we can always assume that. Assume that $  \sigma_p(A^*A)^{-1}\kappa^*\d\mu_g$ is absolutely continuous near the zero section. In the case of the Radon transform in parallel geometry in the next section, for example, with $g$ being white noise, that measure is $C|\xi|\d x\,\d\xi$, so this assumption is satisfied. Then the first integral in \r{thm_IP1} has a limit when $q$ (a priori vanishing near $\xi=0$) tends to $1$, and that limit is given by the same formula with $q=1$. Then the l.h.s.\ has the same limit, too,, because we just defined it by that equality, see \r{VAR}. A similar remark applies to the second integral. 
\end{remark}

\section{The Radon transform in ``parallel geometry''}\label{sec_R_p}
We apply the theory to the Radon transform now. We study the parallel geometry parameterization first, where each (directed) line is parameterized by its signed distance $p$ to the origin $p$ and its normal $\omega$, see \r{7}. 
For 
\be{3.0a}
\omega(\varphi)=(\cos\varphi, \sin\varphi),
\ee
we choose the natural measures $\d\varphi$; and the standard measure $\d p$  for $p$. Based on that, we a define the microlocal defect measure $\d\mu_g(\varphi,p,\hat\varphi, \hat p)$ of $g=g_h(\varphi,p)$. If we restrict $p$ to  $|p|\le R$, corresponding to Radon transforms of functions supported in $B(0,R)$, since $\varphi$ naturally belongs to $|\varphi|\le\pi$ (modulo $2\pi$) (call that $\Omega$), then
\be{14}
\VAR^0_\Omega(g)= \frac1{4\pi R}\int \int_\Omega \d\mu_g(\varphi,p,\hat\varphi, \hat p).
\ee
The Radon transform is an FIO of order $-1/2$ with canonical relation  $\kappa=\kappa_+\cup \kappa_-$, where
\[
\kappa_\pm:  (x, \xi)\longmapsto \bigg(\underbrace{\arg (\pm\xi)}_\varphi, \underbrace{\pm x\cdot\xi/|\xi|}_p, \underbrace{- x\cdot\xi^\perp }_{\hat\varphi},\underbrace{\pm |\xi|}_{\hat p}\bigg).
\]
The ranges of $\kappa_\pm$ intersect in the zero section only, and in particular, $\pm \hat p\ge0$ on the range of $\kappa_\pm$. Next, each branch is a local  diffeomorphism. Indeed, $(x,\xi) = \kappa_\pm^{-1}(\varphi,p,\hat\varphi,\hat p)$ is given by
\[
 x= p\omega(\varphi) - (\hat\varphi/\hat p)     \omega^\perp(\varphi), \quad   \xi= \hat p\omega(\varphi). 
\]
It is well defined for $\hat p\not=0$ but if we want $x$ in the image to be in $|x|<R$, we need to require $p^2+(\hat\varphi/ \hat p)^2<R^2$; therefore $\kappa_\pm^{-1}$ are well defined away from the zero section. Then $\mathcal{R}^{-1}$ is associated with $\kappa^{-1}$, which is a local diffeomorphism as well. What prevents it from being global is that it is 2-to-1, i.e., and in particular, it is not injective.

\subsection{The unfiltered inversion} 
The symbol of $\mathcal{RR}^*$ is $b= 4\pi |\hat p|^{-1}$, where $\hat p$ is the dual of $p$. Applying the canonical relation, we get $b\circ \kappa= 4\pi/ |\xi|$. We could have obtained this as  the principal (and full) symbol $4\pi/|\xi|$ of $\mathcal{R}^*\mathcal{R}$. Therefore, by \r{13}, 
\be{3.1}
\d\mu_{h^{1/2}\mathcal{R}^{-1}g}(x,\xi)= \frac{|\xi|}{4\pi}\kappa^* \d\mu_g \quad \text{on $T^*\R^2\setminus 0$}.  
\ee
The fact that $\kappa$ is 1-to-2 presents some subtlety here, already accounted for in the proof of Theorem~\ref{thm_dmu}. Microlocally, one can express $\mathcal{R}$ as $\mathcal{R}= \mathcal{R}_++ \mathcal{R}_-$; then each $\mathcal{R}_\pm$ has normal operator $\mathcal{R}_\pm^*\mathcal{R}_\pm$ with principal symbols one half of that $\mathcal{R}^*\mathcal{R}$; then we apply \r{13}, and the combined result would be still the principal symbol of $\mathcal{R}^*\mathcal{R}$.

Let us say that we have $f$ supported in $B(0,R)$ with a certain \sc\ band limit $B\ge |\xi|$. We take its Radon transform $\mathcal{R} f$. Here, $f$ is not discretized, we can think of $\mathcal{R}f$ as the physical X-ray transform. The assumption on the band limit will be satisfied if the X-rays are not really ideal lines but have some thickness. 
Then we sample $\mathcal{R}f$ densely enough to satisfy the Nyquist requirements and add noise to it. The noise will have higher frequencies than those coming from $f$ if $\mathcal{R}f$ is oversampled. When we invert $\mathcal{R}f$, we will get higher frequencies for $f$ as well that do not originally belong to the set where the frequency set of $f$ lies. We can apply a filter, cutting them 
 to $|\xi|\le B$. Note that this is a filter not affecting $f$, that is why we think of those as a unfiltered inversion. One way to do this is to restrict $\hat p$ to $|\hat p|\le B$ before applying $\mathcal{R}'$ in \r{8}. 

More precisely, let $\supp f\subset B(0,R)$ and 
\be{3.3}
\WFH(f)\subset \{(x,\xi);\; |x|\le R, \, |\xi|\le B\}.
\ee
Then the frequency set $\Sigma(\Ra f)$ (the projection of the \sc\ wave front set on the fiber variable) of $\mathcal{R}f$ is the double cone
\be{3.4}
\big\{(\hat\varphi, \hat p);\; |\hat\varphi|\le R|\hat p|,\, |\hat p|\le B \big\},
\ee
included in the box $\mathcal{B}:=\{|\hat\varphi|\le RB,\; |\hat p|\le B\}$, see Figure~\ref{fig_0} and \cite{S-Sampling} for more details.  
The set \r{3.4} is the ``worst scenario case'' over all points $(\varphi,p)$. For $|p|\gg0$, the opening of the cone is much smaller: $|\hat\varphi|\le |\hat p|\sqrt{R^2-p^2}$. We refer to \cite{S-Sampling} and Figure~3 there. This describes the range of $\kappa$. Therefore, some portion of the noise will not propagate back to the reconstructed $f$. 

We assume that we sample $g=\mathcal{R}f$ at a rate  smaller than the Nyquist requirement for the box $\mathcal{B}$. Moreover, we assume an interpolation kernel $\chi$ in \r{2.4L} (with $f$ replaced by $g$) is chosen so that $\hat\chi=1$ in a neighborhood of $\mathcal{B}$. As we explained in the introduction, we assume that the data is (white) noise, since the problem is linear. Then the power spectrum of the noise (more precisely, the Wigner function) converges in mean sense to a defect measure $\d\mu_g$ that is absolutely continuous by  Theorem~\ref{thm_m}, i.e., it has the form $\d\mu_g=\gamma_g\,\d x\,\d\xi$ of the kind \r{9aa} on $\mathcal{B}$, with $\gamma_g$ as in \r{de6b}. 
Then on $\mathcal{B}$, we have $\gamma_g=s^n\sigma^2/(2\pi)^n=:\gamma^\sharp$, and 
\be{3.2}
\gamma_{h^{1/2} \mathcal{R}^{-1}g}(x,\xi)= \frac{|\xi|}{4\pi}\gamma^\sharp\quad \text{for $|\xi|\le B$}.  
\ee
This is ``blue noise''. 
Here and below, all equalities about the statistics of $f$ are in the limit sense of Theorem~\ref{thm_m}, see \r{VAR} and \r{9b}. 
\begin{figure}[ht]
\begin{center}
	\includegraphics[page=1 ]{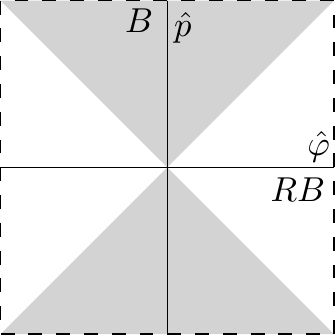}
\end{center}
\caption{\small The frequency set of $\mathcal{R} f$.}
\label{fig_0}
\end{figure}
An important observation is that there is no $x$ dependence in this case. The dependence on $\xi$ is rotationally invariant. This is not the case with the Radon transform in fan-bean coordinates as we will see below.

By \r{14},
\[
\VAR^0_\Omega(g) = \VAR^0_{p,\varphi}(g)= 4B_\varphi B_p\gamma^\sharp, \quad \forall (\varphi, p)\in  S^1\times [-R,R]. 
\]
The two variances are equal because $\gamma_{\mathcal{R}^{-1}g}$ is independent of the position. 

Assume that the sampling rates of $g$ are based on $B_\varphi$ and $B_p$  which take their sharp values not to allow undersampling: $B_p=B$, $B_\varphi=RB$, where $B$ is the band limit of $f$ as in \r{3.3}. Then 
\[
\VAR^0_{p,\varphi}(g) =  4RB^2\gamma^\sharp.
\]
Note that this is actually the sharp lower bound of the variation when the oversampling becomes asymptotically sharp sampling but it is not achievable in our theory; this would require a sinc interpolation while we need  a rapidly decreasing kernel.

For the variance of $f= \mathcal{R}^{-1}g$, we have, see Theorem~\ref{thm_IP} and  Remark~\ref{rem_R},
\be{3.5}
\begin{split}
\VAR^0_x(h^{1/2}\mathcal{R}^{-1}g) &= \int_{  |\xi|\le B } \gamma_{\mathcal{R}^{-1}g}(x,\xi)\, \d\xi\\
&  =  \frac{1}{4\pi}\gamma^\sharp   \int_{|\xi|<B}|\xi|\,\d\xi = \frac{1}{4\pi}\gamma^\sharp  2\pi \int_0^B \rho^2\,\d\rho\\
&=\frac{ B^3\gamma^\sharp}{6}.
\end{split}
\ee
We get the following theorem.

\begin{theorem}[unfiltered inversion]\label{thm1}
Under the assumptions above, in particular assuming that $g$ is white noise, 
and no undersampling, we have 
\be{3.5b}
\STD^0(  \mathcal{R}^{-1}g)= \frac{B^{3/2}}{\sqrt{24 B_\varphi B_ph}}\STD^0(g).
\ee
If $g=\mathcal{R}f$ is sampled sharply, then
\be{3.5c}
\STD^0( \mathcal{R}^{-1}g)= \left(\frac{B}{24 Rh}\right)^{1/2}\STD^0(g).
\ee
\end{theorem}
Recall that we defined $\VAR^0$, see \r{VAR}, and similarly, $\STD^0$, as integral of the defect measure. The implication of this theorem is that when we have $g$ created by a white noise process, then for every $Q=q(hD)$ with $q=0$ near the origin,  $\STD^0(h^{1/2}Q\mathcal{R}^{-1}g)$ converges in mean square sense to a quantity (see \r{3.12}), which itself converges to the r.h.s.\ of \r{3.5b}, respectively \r{3.5c}, when $q\to1$. In other words, the cutoff near $\xi=0$ is removable at the expense of taking a double limit: first $h\to0$, then $q\to1$ (in $L^1$ sense). 

\subsection{The filtered inversion} 
The Radon transform is inverted often with a low-pass filter before applying $\mathcal{R}'$ in \r{8}, i.e., 
\be{3.8}
f= \frac1{4\pi} \mathcal{R}' \nu(D_p) |D_p|g,
\ee
where $\nu$ is an even function decaying away from the origin. Assuming a band limit $B_p$ for the $p$ variable, determined by the sampling rate $s_p$, for example, one popular filter is the Hann filter:
\be{3.8a}
\nu_\textrm{Hann}(\hat p) = \frac12\left(1+\cos\frac{\pi \hat p}{B_p}\right)= \cos^2\frac{\pi \hat p}{2B_p}, \quad |\hat p|\le B_p,
\ee
and $\nu_\textrm{Hann}(\hat p) =0$ otherwise. Another commonly used filter is the cosine one
\[
\nu_\textrm{cosine}(\hat p) = \cos\frac{\pi \hat p}{2B_p}, \quad |\hat p|\le B_p.
\]
They are plotted in Figure~\ref{fig_filters}. 
\begin{figure}[ht]
\begin{center}
	\includegraphics[trim = 0 30 0 30  ,scale=0.3]{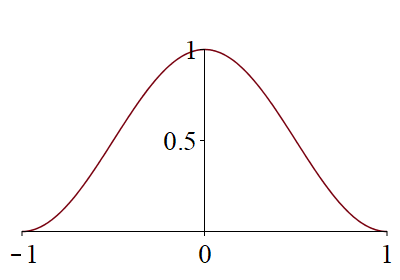}\hspace{25pt}
	\includegraphics[trim = 0 30 0 30  ,scale=0.3]{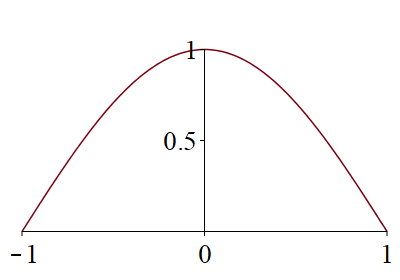}
\end{center}
\caption{\small The Hann and the cosine filters with $B=1$.}
\label{fig_filters}
\end{figure}

There are many other filters (windows) used in signal processing and imaging. We assume that $\nu$ is continuous and supported in $|\hat p|\le B_p$. If the shape of the filter is fixed, say Hann, then $\nu(t)=\nu_0(t/B_p)$ with some fixed $\nu_0$ supported in $[0,1]$, see, e.g., \r{3.8a}.  
Then \r{3.1} takes the form
\be{3.9}
\gamma_{h^{1/2} \mathcal{R}_\nu^{-1}g}(x,\xi)= \frac{|\xi|\nu_0^2(|\xi|/B_p)}{4\pi} \gamma_g\circ \kappa(x,\xi),
\ee
where $\Ra_\nu^{-1}= \Ra\nu(D_p)$ is the filtered inversion, defined as the operator applied to $g$ in \r{3.8}. Then the equivalent to \r{3.2} is
\[
\gamma_{h^{1/2} \mathcal{R_\nu}^{-1}g}(x,\xi)= \frac{|\xi|\nu_0^2(|\xi|/B_p)}{4\pi}\gamma^\sharp.
\]
 Taking $B_p=B$ as before,  similarly to \r{3.5} we get the following analog of \r{3.5}
\be{3.12}
\begin{split}
\VAR^0_x(h^{1/2} \mathcal{R_\nu}^{-1}g) &= \int_{ |\xi|\le B } \gamma_{\mathcal{R}_\nu^{-1}g}(x,\xi)\, \d\xi\\
&  =  \frac{1}{4\pi}\gamma^\sharp  \int_{|\xi|<B}|\xi|\nu_0^2(|\xi|/B) \,\d\xi = \frac{1}{4\pi}\gamma^\sharp  2\pi  \int_0^B \rho^2\nu_0^2(\rho/B)\,\d\rho\\
&=  \frac{1}6  B^3 \gamma^\sharp c_\nu,
\end{split}
\ee
where
\be{3.11a}
c_\nu:= 3\int_0^1\rho^2\nu_0^2(\rho)\,\d\rho.
\ee

We proved the following.

\begin{theorem}[filtered inversion]\label{thm2}
Under the assumptions above, in particular assuming white noise and no undersampling, with a filter $\nu_0(|D_p|/B) $, we have 
\be{3.5bb}
\STD^0(\mathcal{R}^{-1}g)= \frac{B^{3/2}\sqrt{c_\nu}}{\sqrt{24 B_\varphi B_p h}} \STD^0(g).
\ee
If $\mathcal{R}f$ is sampled sharply, then
\[
\STD^0(\mathcal{R}^{-1}g)=  
 \left(\frac{B c_\gamma}{24 Rh}\right)^{1/2}\STD^0(g).
\]
\end{theorem}

If there is no filter ($\nu_0=1$), we have $c_\nu=1$, which explains the appearance of the factor $1/3$ in the definition of $c_\nu$. 
For the Hann filter, $c_\nu=3/8-45/(16\pi^2)\approx 0.0900$, then  $\sqrt{c_\nu}\approx 0.3000 $. For the cosine filter, $\sqrt{c_\nu}\approx 0.4427$.  In \r{3.6b} below, the constant would be approximately $0.07676$ for the Hann filter and $0.11327$ for the cosine one. 

\subsection{Numerical experiments} We use MATLAB and the built in {\tt radon} and {\tt iradon} routines to compute and invert numerically the Radon transform in the plane. The default angular step is one degree but it can be changed. Assume that $f$ is given on an $N\times N$ lattice. Then by default, {\tt radon} computes $\mathcal{R}f(\varphi,p)$ on a $360\times N\sqrt{2}$ lattice, with $N\sqrt 2 $ rounded; the actual formula is $2\,\text{ceil}\big(\sqrt{2}(N-\text{floor}((N-1)/2)-1)\big)+3$. Then {\tt iradon} inverts the data to the original grid (with $N$ replaced by $N+1$ or $N+2$ which does not matter in view of our asymptotic setup). 

As we showed in \cite{S-Sampling}, this choice of the discretization of $\mathcal{R}f$ is suboptimal for $N\gg1$; we need to compute $\mathcal{R}f$ on an $N_\varphi\times N_p$ lattice with $N_p=2N$, $N_\varphi=2\pi N$ at least, and some oversampling would be beneficial, see Figure~\ref{fig_1}. With most test images, the (dominating) frequencies are well below the Nyquist limit, that is why most of the time the inversion is satisfactory. When we add, say white noise, the Nyquist limit is reached, and the inversion with {\tt iradon} will alias some of those frequencies.

\subsection{Discretization} \

Let us say we have $f$ on an $M\times N$ grid. We think of that as discrete samples of $f$ originally defined on, say, $[-a,a]\times [-b,b]$. 
This we have the steps $s_{x_1}=2a/M$, $s_{x_2}=2b/N$. Assume for a moment that we apply the classical sampling theory (no small parameter $h$) in a formal way at this point. Then those steps  have to be $\pi/B_{x_1}$, respectively $\pi/B_{x_2}$ at most, where $B_{x_j}$ are the band limits in the $x_j$ variable. Then 
we get $B_{x_1}=M\pi/(2a)$, $B_{x_2}=N\pi /(2b)$ as the least upper bounds of the band limits of $f$.  For the band limit of $|\xi|$, we have $B= \big(B_{x_1}^2+B_{x_2}^2\big)^{1/2}$, and the maximum is achieved at the vertices of the box $[-B_{x_1},B_{x_1}] \times [-B_{x_2},B_{x_2}]$. Note that the disk $|\xi|\le B$ contains more frequencies than can be properly sampled on the $M\times N$ grid; the extra ones lie outside that inscribed  box. 

We can connect the classical sampling theory to the \sc\ one as follows. Denote for a moment the \sc\ quantities with tildes over them. Let $M=\tilde M/h$, $N=\tilde N/h$, with $\tilde M$, $\tilde N$ fixed. The steps $s$ ($s_{x_1}$, etc.) are equal to the \sc\ \textit{relative} steps $\tilde s$ but since in our sampling theorems the absolute  steps are $sh$, this means that the absolute steps are multiplied by $h$. Then our analysis holds as $h\to0$, i.e., as $M\to\infty$, $N\to\infty$ (keeping the ratio constant) and the steps going to zero at a rate $\sim h$. This is the usual setup in numerical analysis where $\tilde s=1$, i.e., the step is $h$.

For each such $f$ we define the $L^2$ norm as 
\[
\|f\|^2= \frac{4ab}{MN}\sum_{i=1}^{M} \sum_{j=1}^{N}|f_{ij}|^2.
\]
This is consistent with formula (16) in  \cite{S-Sampling} and approximates the $L^2$ norm of a continuous function on that box with samples $f_{ij}$. Then
\[
\STD(f) = \Big(\frac{1}{MN}\sum_{i=1}^{M} \sum_{j=1}^{N}|f_{ij}|^2\Big)^\frac12 =\frac{\|f\|}{2\sqrt{ab}}
\]
 is  the standard deviation $\STD(f)$ of $f$ when the mean of $f$ is zero. 

We will apply this to both $f$ defined on $[-a,a]^2$ for some $a>0$,  and to $\mathcal{R}f$ on $[-\pi,\pi]\times [-R,R] $. 

Assume that $g$ is a discrete representation of a function on $[-\pi,\pi]\times [-R,R] $ sampled on an $N_\varphi\times N_p$ lattice. Assume $g$ is obtained by a white noise process (with zero mean)  and  variance $\sigma^2$. 
{Then a slight extension of Lemma~\ref{lem:as-convergence-variance} shows that}
$\VAR(g)\to \sigma^2 $ almost surely.

The sampling steps are $s_\varphi=2\pi/N_\varphi$, $s_p=2R/N_p$; hence   to avoid aliasing, we need $B_\varphi \le N_\varphi/2$, $B_p\le \pi N_p/(2R)$.

Let $f$, to which $\mathcal{R}$ will be applied, represent a discretization of a function on $[-a,a]^2$, and assume that it is sampled on an $N\times N$ lattice. Then, similarly, the sharp  band limit in each variable is $B_{x_1}=B_{x_2}=\pi N/(2a)$. 

As we showed in \cite{S-Sampling}, and it follows easily from \r{3.4}, to avoid aliasing, we need 
\be{3.5N}
N_p\ge 2N, \quad N_\varphi\ge 2\pi N. 
\ee
This inequality, as well as the inequalities and the equalities below are meant in asymptotic sense, i.e., one should multiply, say the r.h.s.\ in this case by $(1+o(1))$, as $N\to\infty$. 
Note that \r{3.5N} follows from viewing $f$ as supported in $B(0,\sqrt{2}a)$, i.e., $R=\sqrt2R$, with frequency set in $|\xi|\le B:=\sqrt2 B_{x_1}$. As we mentioned above, that ball contains more frequencies than those in its inscribed square. For every $g=\mathcal{R}f$ in the range of $\mathcal{R}$ with $f$ as above, after an inversion we get $f$, of course, and then the frequencies will fall inside the inscribed square $[-B_{x_1},B_{x_1}]^2$. If we take $g$ to be ``noise'', not in the range of $\mathcal{R}$, then by the mapping property of $\kappa^{-1}$, see \cite{S-Sampling}, formula (51), the frequency set of $\mathcal{R}^{-1}g$ will generically fill the disk $|\xi|\le B= B_{x_1}\sqrt{2}$. If we want to avoid aliasing (without applying a filter), we would need to reconstruct $f=\mathcal{R}^{-1}g$ on an $N\sqrt2\times N\sqrt2$ grid or better. On the other hand, for all practical purposes, we would want to apply a filter. 

\begin{figure}[ht]
\begin{center}
	\includegraphics[page=2 ]{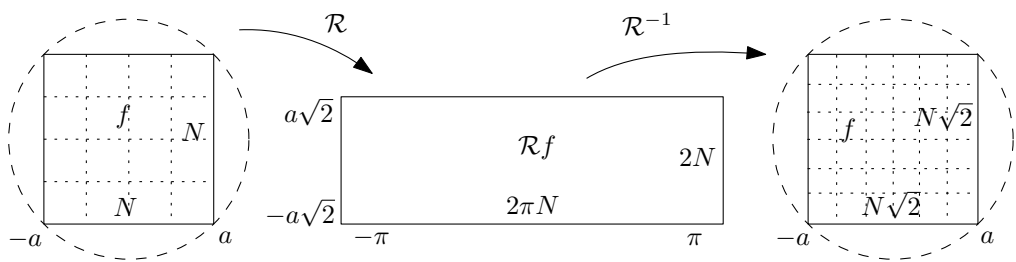}
\end{center}
\caption{\small The sampling sets of $f$, $\mathcal{R} f$ and the reconstructed $f$ with sharp sampling requirements.  
}
\label{fig_1}
\end{figure}

Therefore, the discrete version of \r{3.5bb}, including a filter now, is
\be{3.6}
\VAR(\mathcal{R}_\nu^{-1}g) = \frac{\pi^2N^3c_\nu }{12a^2N_\varphi N_p}\VAR(g),
\ee
where the formula has the same asymptotic and probabilistic meaning as explained after Theorem~\ref{thm1}. 

Assume now that we sample sharply, i.e., we have equalities in \r{3.5N}. Then $N_p=2N$, $N_\varphi= 2\pi N$ and we get
\be{3.6a}
\VAR(\mathcal{R}_\nu^{-1}g) = \frac{\pi Nc_\nu}{48a^2 }\VAR(g).
\ee
Therefore, 
\be{3.6b}
\STD(\mathcal{R}_\nu^{-1}g) \approx 0.2558 \sqrt{c_\nu}\frac{\sqrt{N}}a \STD(g).
\ee

We can make the following conclusions from \r{3.6}, \r{3.6a} and \r{3.6b}.
\begin{itemize}
	\item With a sharp sampling rate, the noise ratio, measured as its standard deviation relative to that of $g$, increases as $\sqrt N$. This is understandable since we are allowing for higher frequencies, and $\mathcal{R}^{-1}$ is of order $1/2$. At the same time, we can handle $f$ with higher frequencies because $N$ is proportional to the Nyquist bound. 
	\item The noise ratio, for a fixed $N$, is minimized when we sample sharply. 
	\item In many applications, increasing $N_\phi$ and $N_p$ decreases the size of the detectors, and then the discrete samples $g_{ij}$ are scaled down by constants times $N_\phi$ and $N_p$.  If the added noise is expressed in units relative to that, then the quotient in \r{3.6} would be proportional to $N^3N_\varphi N_p$, i.e., the noise ratio increases with $N_\phi$ and $N_p$. This is known in engineering.
\end{itemize}

\textbf{Default {\tt iradon} inversion.} 
First we present an inversion with the default one degree angular step. We choose $N=601$, $N_\varphi=360$ by default and $N_p=853$ is chosen by {\tt radon} as an approximation to $601\sqrt{2}$. We choose $g$ to be normally distributed (Gaussian) noise  with standard deviation one. Then we invert it with {\tt iradon}. A plot of the  modulus $|\hat f|$ of the Fourier transform $\hat f$ of the inversion $f$ is shown in Figure~\ref{fig_2}. 
\begin{figure}[ht]
\begin{center}
	\includegraphics[scale=0.2]{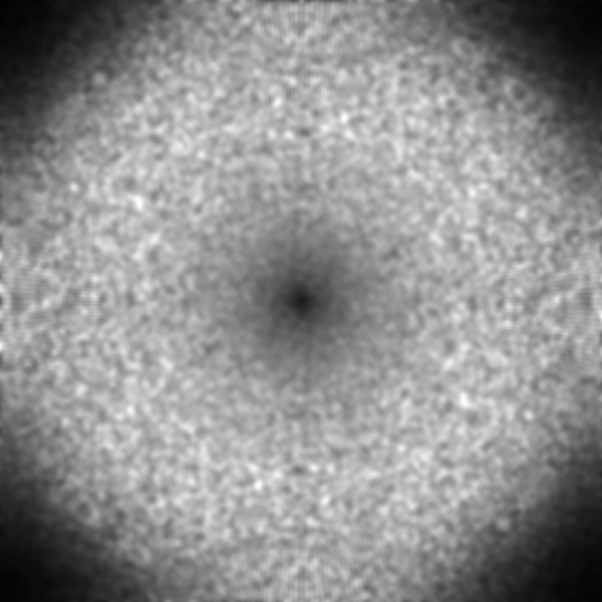}\hspace{20pt}
	\includegraphics[trim = 0 63 0 60  ,scale=0.235]{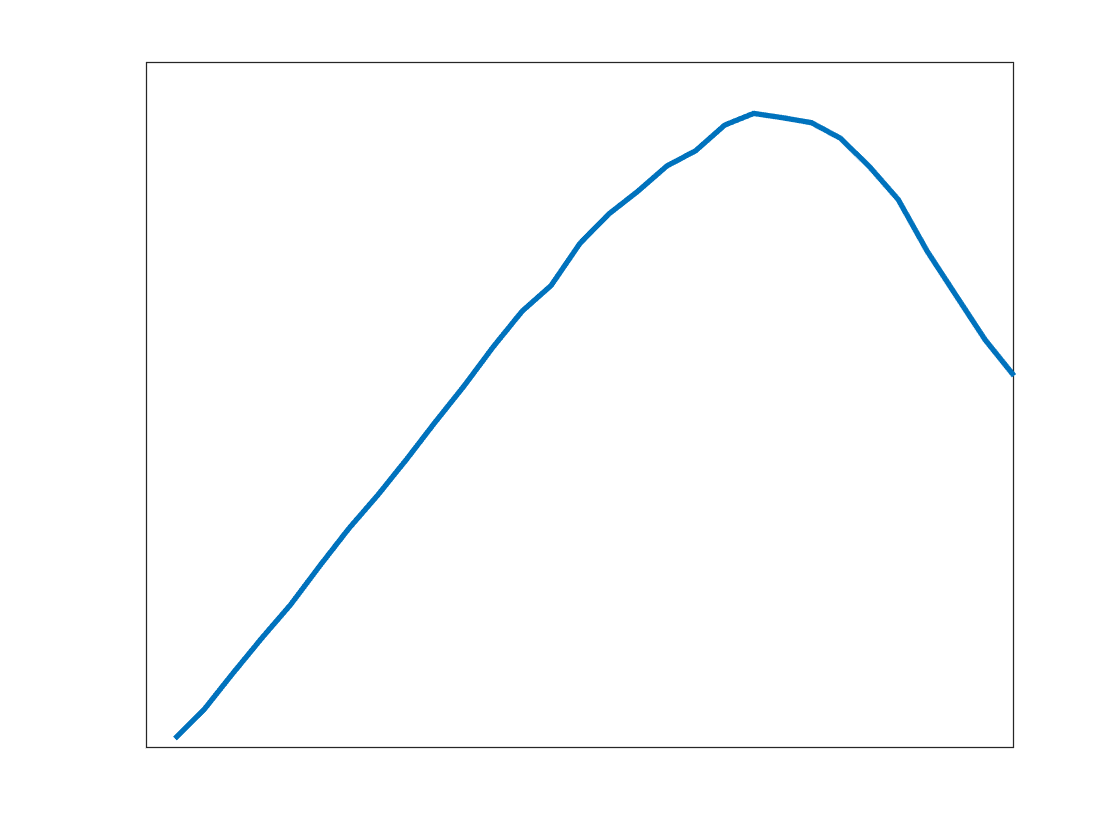}
\end{center}
\caption{\small Left: $|\hat f|$ where $f=   \mathcal{R}^{-1} g $ and $g$ is white noise. Right: radial profile of  $|\hat f|^2$ from the center to one of the sides (but not all the way along the diagonal to a vertex).}
\label{fig_2}
\end{figure}
We chose to plot here and below $|\hat f|$ rather than $|\hat f|^2$ for clarity. 
With an exact inversion, as $N\to\infty$, we should be seeing a density plot of square root of \r{3.2}, i.e., $c|\xi|^{1/2}$, filling the whole square. We see is that the density increases in the radial variable $|\xi|$ from the center but at some point starts to decease until it visibly becomes zero when $|\xi|$ is slightly larger than  a  half of the side, and it is radially symmetric. This behavior can be explained by the following. The default choice $N_p=N\sqrt 2$ (rounded) of $N_p$ actually lowers the Nyquist limit of the reconstructed $f$ to $1/\sqrt2$ of its original value. Without that, the boundary of the disk in Figure~\ref{fig_2} would be the circumscribed circle of that square but with that choice, it is the inscribed one. The gradual decrease close the border can be explained by an effectively low pass filter when inverting $\mathcal{R}$. Our numerical experiments below at much higher resolutions for $\mathcal{R}f$ confirm that. 

A similar experiment with a uniformly distributed noise $g$ in a symmetric interval around the origin produces virtually the same plot of $|\hat f|$, not shown. 
In both cases, the values of $f$ look normally distributed.

\textbf{High precision inversion.} We present numerical inversions with a proper discretization. We want to model adding noise to discrete measurements of the ``continuous'' $\Ra f$; inverted with high precision; i.e., by upsampling first the discrete data several times to mimic inversion in the ``continuous domain''. 
We do the following.

\begin{itemize}
\item[(i)] The function $f$ is assumed to be defined on $[-1,1]^2$ and sampled on an $N\times N$ lattice. 
\item[(ii)] We compute a high accuracy $\mathcal{R}f$ on a $N_\varphi\times N_p$ lattice, where $N_\varphi\ge 2\pi N$, $N_p\ge 2N$. To do that, we perform the computations on a finer grid. 
\item[(iii)] We add noise to the so-computed $\mathcal{R}f$. 
\item[(iv)] We invert the noisy data by upsampling it first. The reconstructed $f_\text{noisy}$ is either left sampled on a finer grid or downsampled to the original $N\times N$ one. 
\end{itemize}

We  give more details below. 
To do (ii), we upsample $f$ on an $mN\times mN$ lattice with {\tt Lanczos-3} with some $m\ge1$. Typical $m$'s we use are $m=2$ and $m=3$. 
Then we compute $\mathcal{R} f(\varphi,p)$ with {\tt radon} on a $2\pi mN\times 2mN$ lattice which we view as $\mathcal{R} f(\varphi,p)$ on $[-\pi,\pi]\times [-\sqrt2, \sqrt2]$ sampled uniformly in each variable. The parameter $m$ represents the degree of oversampling: $m=1$ corresponds to the sharp lower bound for proper sampling. Since computing $\mathcal{R}f$ involves interpolation of $f$ for computing the line integrals (we use the option 'spline' in {\tt radon}), such an oversampling allows us to reduce the errors in such interpolation compared to the sinc inversion. Then we downsample the computed $\mathcal{R}f$ to a lower resolution $N_\varphi\times N_p$ (without interpolation; we take every  $m$-th value in each row and column). 
This simulates a high precision  $\mathcal{R}f$ computed on the $N_\varphi\times N_p$ grid. To do (iii), we add noise. 

In (iv), we 
invert $\mathcal{R}$ on that lattice. We could resize to a different (but high enough resolution) before that but the results do not look much different. The resulting $f=\mathcal{R}^{-1}g$ is computed on an $mN\times mN$ lattice, which is viewed as $f$ on $[-1,1]^2$ sampled uniformly. If needed, that $f$ could be resampled to an $N\times N$ lattice but since it does not contain frequencies higher than the Nyquist limit $B=N\pi/2$ corresponding to $N$, this is not needed for computing the standard deviation, for example. 

We want to emphasize that it is possible to do (close to) ideal  upsampling, say from $N_x\times N_y$ to $\tilde N_x\times \tilde N_y$ with $\tilde N_x>N_x$ and $\tilde N_y>N_y$ which preserves the band limits $B_{x_1}=N_x\pi/2$ and $B_{x_2}=N_y\pi/2$ by using the Fourier transform. On the other hand, this is not what is usually done. When we use  {\tt  Lanczos-3}, for example, the interpolation kernel is the inverse Fourier transform of a smoothened version of $\nu_{[-1,1]}$, see Figure~\ref{fig_3}, which is close to be equal to one in $[-0.5, 0.5]$ at least as explained in section~\ref{sec_Lan3}. 
On the other hand, Theorem~3.1 in \cite{S-Sampling} requires some oversampling, and an interpolation kernel to be the Fourier transform of a function similar to that in Figure~\ref{fig_3}, equal to one on the (smaller) frequency band. Therefore if we choose $m\ge 2$ we are in this regime.

To do experiments with noise only, we take $f=0$ in \r{2}. Then steps (i) and (ii) are trivial, since $\mathcal{R}f=0$. So our starting point is  (iii), where we take $g$ to be generated by either a normally or a uniformly distributed noise, on  an $N_\varphi\times N_p$ grid. We upsample by a factor of $m$, i.e., to an $mN_\varphi\times mN_p$ grid an do the inversion there. We take $m=2,3,4,5$ in our experiments.

\textbf{Non-filtered inversion.} 
We test \r{3.5b} now. To this end,  we take $g$ to be either Gaussian or uniformly distributed noise with zero mean on an $N_\varphi\times N_p$ grid as in (ii), with equalities there, i.e., $N_\varphi=2\pi N$, $N_p=2N$.  
Then we cut the Fourier transform of the result sharply  to $1/m$-th of the frequency box corresponding to the original resolution $N_\varphi\times N_p$, $m=2,3,\dots$; denote this by $\nu_m(D)g_m$, and apply $\mathcal{R}^{-1}$ to it without changing the grid size. 
This procedure provides more precise computation than just inverting the noise 
because it avoids the smoothing which happens in the part we cut off.  If we had $\Ra f$ of a non-zero $f$ polluted with noise, we would have upsized  the data $m$ times in each dimension first, and then would have performed that procedure. 

\begin{figure}[hb]\,\hfill
\begin{subfigure}[t]{.12\linewidth}%
\includegraphics[height=.1\textheight]{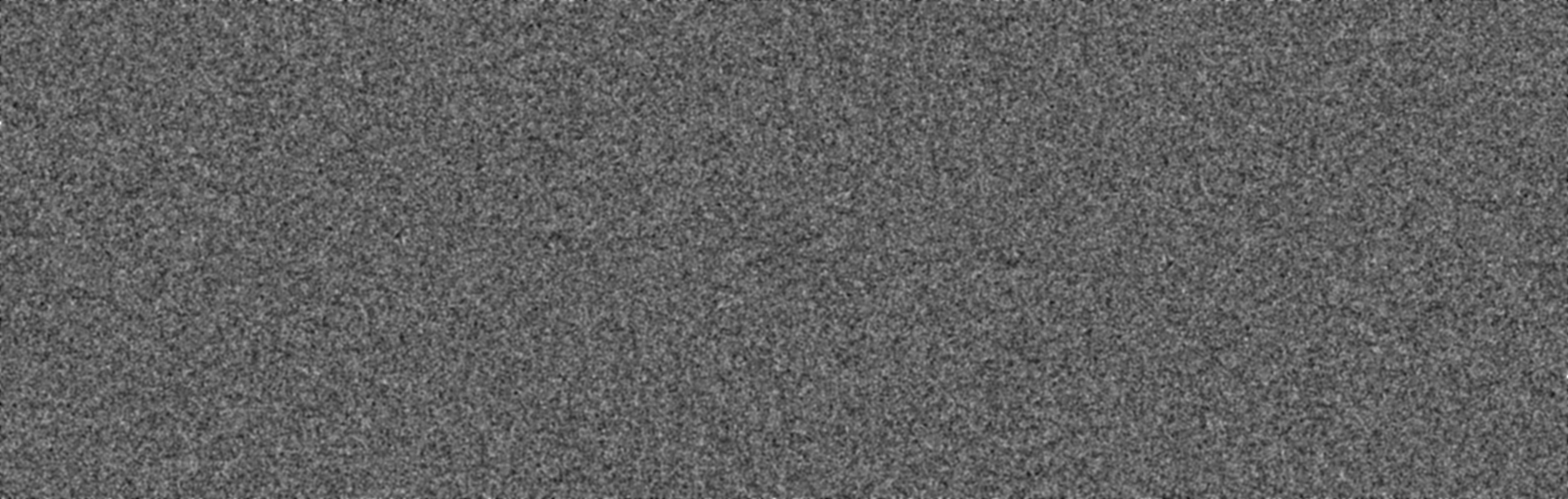}
 \captionsetup{width=2.2\linewidth}
\caption{$|\hat g_m|$ with $m=2$: white noise }
\end{subfigure}\hfill\hspace{2in}
\begin{subfigure}[t]{.19\linewidth}
\centering
\includegraphics[height=.15\textheight]{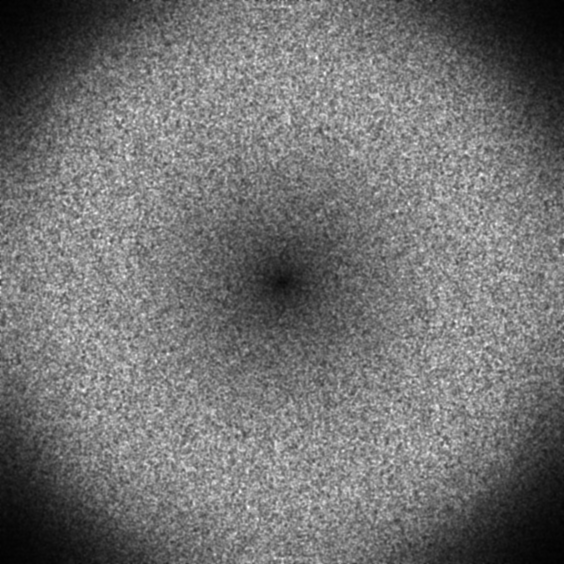}
\caption{reconstruction }
\end{subfigure}\hfill \,\\ \vspace{10pt}
\,\hfill
\begin{subfigure}[t]{.12\linewidth}%
\includegraphics[height=.1\textheight]{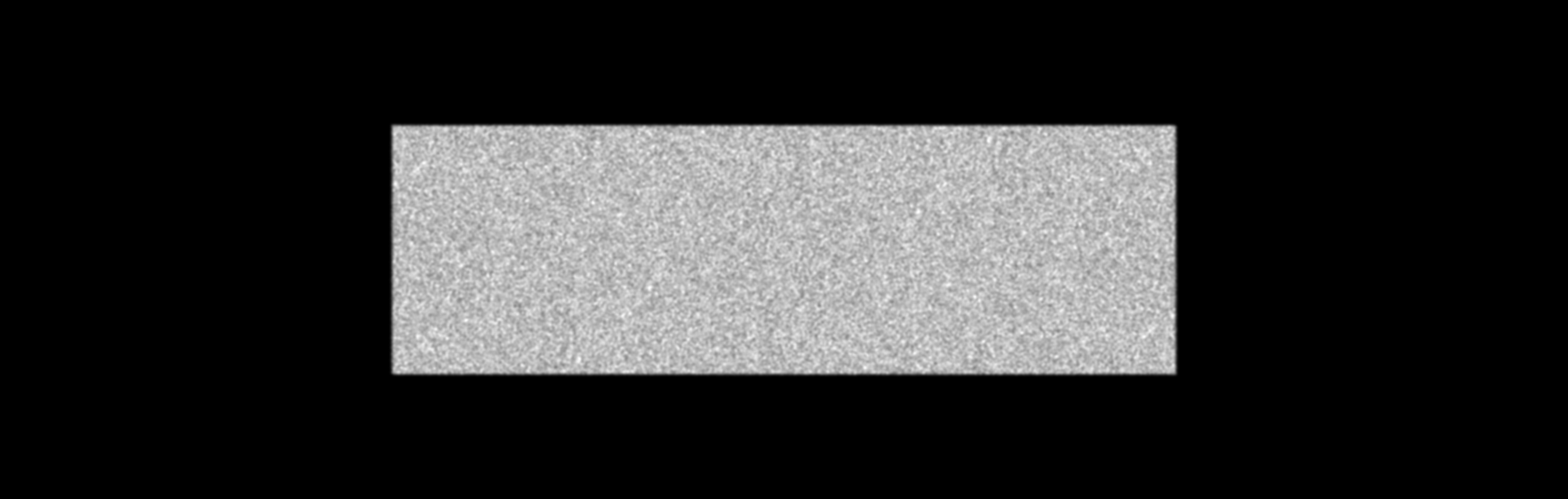}
 \captionsetup{width=3.7\linewidth}
\caption{$|\hat g_m|$ with $m=2$  cut to a half in each dimension}
\end{subfigure}\hfill\hspace{2in}
\begin{subfigure}[t]{.19\linewidth}
\centering
\includegraphics[height=.15\textheight]{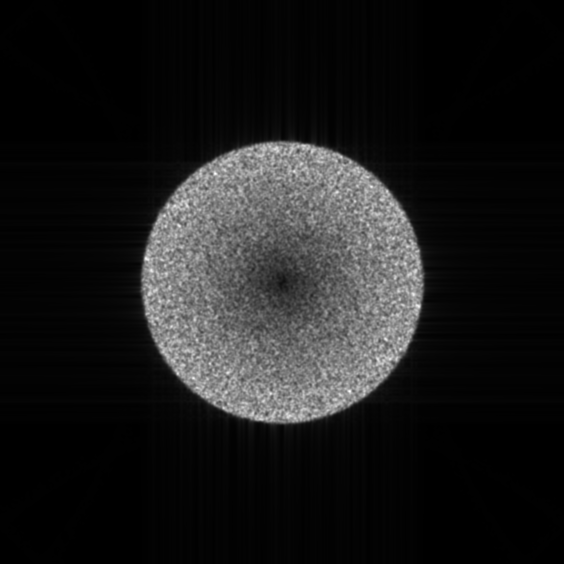}
\caption{reconstruction}
\end{subfigure}\hfill \,\\ \vspace{10pt}
\,\hfill
\caption{Top: We choose $g_m$, $m=2$ to be white noise; then $|\mathcal{R}^{-1}g_m|$ looks like in Figure~\r{fig_2}. Bottom: $\nu_m\hat  g_m$ and  $\mathcal{R}^{-2}\nu_m g_m$, i.e., the Fourier transform of the reconstruction after the frequency cut-off of the noise. 
}\label{fig_8}
\end{figure}

Since we effectively multiply both $N_\varphi$ and $N_p$ by $m$,
 by \r{3.6}, we see that \r{3.6b} can be written in terms of the noise ratio as
\be{3.13}
\text{Noise ratio}: =   m\STD(\mathcal{R}^{-1}g)\Big/ \frac{\sqrt{N}}a \STD(g)\approx 0.2558 .
\ee
We take $g$ first to be a Gaussian noise with several choices of $N$ and $m$; doing five experiments for each choice. The results are in  Table~\ref{t1} below and in Figure~\ref{fig_8}, we illustrate the inversion with $m=2$. 
\begin{table}[ht]
\begin{tabular}{|l|l|l|l|}
 \hline
 \multicolumn{4}{|c|}{Noise ratio with Gaussian noise. Theoretical ratio: $\mathbf{0.2558} $ } \\
 \hline & $N=100$ & $N=200$ & $N=300$   \\ \hline 
m=1  &  $  0.2224  \pm 0.61\%$& $0.2223\pm 0.32\%$ &  $0.2226\pm 0.16\%$  \\ \hline
m=2  & $0.2552\pm 0.38\%$  & $0.2572\pm 0.25\%$  &  $0.2578\pm 0.17\%$   \\ \hline
m=3  & $0.2569\pm 0.70\%$  & $0.2584\pm 0.41\%$  &  $0.2591\pm 0.07\%$   \\ \hline
\end{tabular}
\caption{Noise experiments}\label{t1}
\end{table}
Similar experiments with a uniformly distributed noise with mean zero generate similar numbers, not shown. 

\textbf{Filtered inversion.} We perform similar experiments with the Hann and the cosine filter. Since the Hann filter is very small near the band limit $B$, see Figure!\ref{fig_filters}, the smoothing effect of the interpolation used by {\tt iradon}, see Figure~\ref{fig_2}, plays a negligible role. Modeling that smoothening by the Lanczos-3 profile, for example, see Figure~\ref{fig_3}, by introducing an extra factor in \r{3.11a} shows an error of less than $1\%$ in $\sqrt{c_\nu}$.  Then even with $m=1$, we get a  result close to the theoretical one, which is approximately $0.07676$ for the Hann filter and $0.11327$ for the cosine one, as we computed above. For $N=300$, for example, we get $0.0767\pm 0.11\%$ for Hann and $0.1105\pm 0.21\%$ for cosine, where the smoothing effect of {\tt iradon} is a bit less compensated for. The numbers for normally and uniformly distributed noise are very close. 

For the cosine filter, we plot $|\hat f|$ (instead of $|\hat f|^2$ for clarity), the computed radial profile of $|\hat f|^2$, and its theoretical one $\rho\nu_\text{cosine}^2(\rho) =\rho\cos^2(\pi\rho/2)$ in Figure~\ref{fig_cos} below. The radial profile is computed as $|\hat f|^2$ averaged over $25$ concentric rings. 
 In this case, $|\hat f (\xi)|^2$ is proportional to  the microlocal defect measure of $f$ at any fixed $x$ (it does not depend on $x$). 
\begin{figure}[ht]
\begin{center}
	\includegraphics[trim = 0 79 0 60  , scale=0.22]{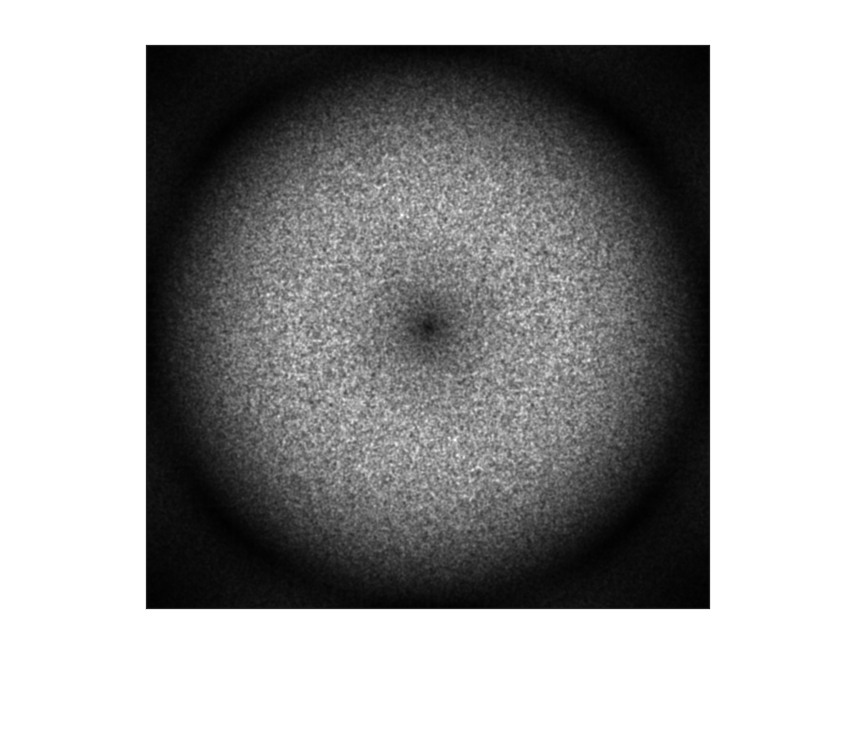}
	\includegraphics[trim = 0 50 0 60  ,scale=0.18]{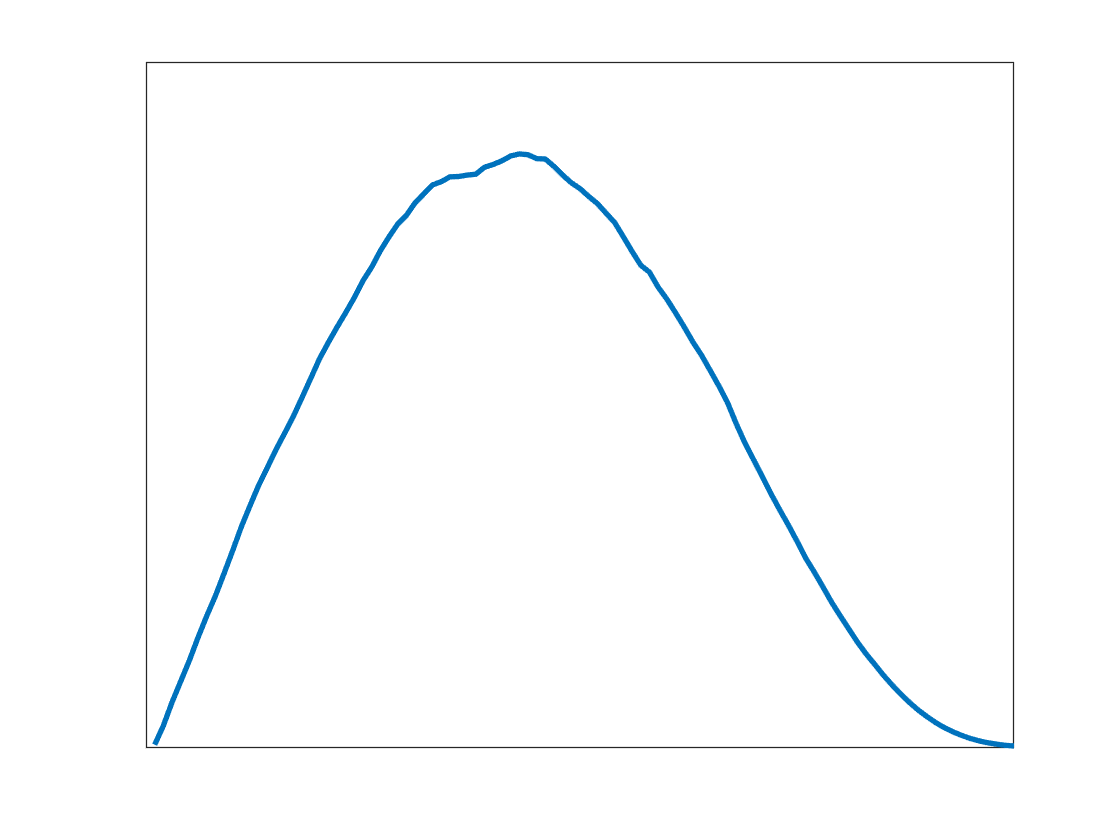}
	\includegraphics[trim = 0 40 0 0 ,scale=0.315]{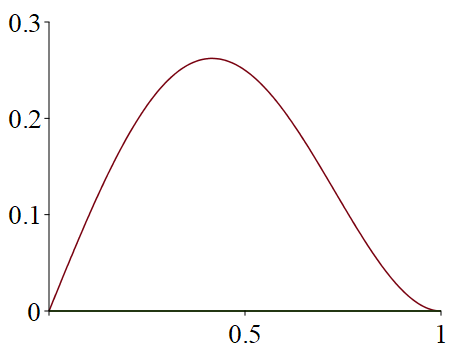}
\end{center}
\caption{\small Cosine filter.  Left: $|\hat f|$ where $f=   \mathcal{R}^{-1}\nu_\text{cosine} g $ and $g$ is white noise. Center: The computed radial profile of  $|\hat f|^2$ from the center to one of the sides. Right: The theoretical profile $\rho\nu_\text{cosine}^2(\rho) =\rho\cos^2(\pi\rho/2)$.}
\label{fig_cos}
\end{figure}
The Hann filter behaves similarly, with the computed radial profile of $|\hat f|^2$ very close to its theoretical one $\rho\cos^4(\pi\rho/2)$.

\subsection{Percentage of added noise} In many numerical simulations, we add   noise to the data, as a percentage of a certain norm of the data, and measure the percentage of the noise in the reconstruction. This is especially interesting in (mildly or not) ill-posed problems. 

There is a lot of flexibility in choosing those norms. Let us say that we choose the $L^2(B(0,R))$ norm for $f$ and the $L^2(S^1\times (-R,R))$ norm for $\Ra$. Then the left inverse $\Ra^{-1}$ is not bounded in those spaces but on \sc ly bounded functions (which are smooth), it is; we refer to \cite{S-Sampling} for \sc\ estimates. 

Let $g_\text{noise}$ be the noise added to $g=\Ra f$, see \r{2}. Its percentage is given my $\|g_\text{noise}\|/\|\Ra f\|$ (converted to percentage). We are interested in $\|f_\text{noise}\|/\|f\|$, where $f_\text{noise} = \Ra^{-1}g_\text{noise}$ is the noise in the reconstruction. We have
\be{3.14}
\frac{\|f_\text{noise}\|}{\|f\|} =  K\frac{\|g_\text{noise}\|}{\|\Ra f\|}, \quad K:=  \frac{\|f_\text{noise}\|}{\|g_\text{noise}\|} \cdot  \frac{\|\Ra f\|}{\|f\|}.
\ee
The coefficient $K$ is the multiplier which relates the two percentages. Its first factor is proportional to the noise ratio we studied earlier since the $L^2$ norms are proportional to the standard deviations. The second one depends on $f$. To analyze  it, write
\[
\|\Ra f\|^2  \longrightarrow\int\gamma_{\Ra f} \, \d \varphi\,\d p\,\d \hat\varphi\,\d\hat p,
\]
where the convergence is in the sense of Theorem~\r{thm_m}.
Then we integrate over the \sc\ wave front. 
By \r{12aa}, 
\[
\begin{split}
\|\Ra f\|^2 &\longrightarrow \int(b\gamma_{f} )\circ \kappa^{-1}\, \d \varphi\,\d p\,\d \hat\varphi\,\d\hat p = 4\pi \int \frac{\gamma_{f}\circ \kappa^{-1}}{|\hat p|} \, \d \varphi\,\d p\,\d \hat\varphi\,\d\hat p \\
&= 4\pi \int \frac{\gamma_f(x,\xi)}{|\xi|}\,\d x\,\d\xi= 4\pi \big\||D|^{-1/2}f\big\|^2. 
\end{split}
\]

We used again the fact that $\kappa$ is an isometry. 
This works for general operators but for $\Ra$ we actually know that $4\pi \|f\|^2=\||D_p|^{1/2}\Ra f\|^2 $. 
We can write $\tilde f = |D|^{1/2}f$, intertwine $|D|^{1/2}$ with $|D_p|^{1/2}$, to get the formula above as an exact one, not just a limit.  
Therefore, \r{3.14} yields
\be{3.15}
K=  \frac{\|f_\text{noise}\|}{\|g_\text{noise}\|} \cdot  \frac{4\pi \big\||D|^{-1/2}f\big\|}{\|f\|}.
\ee
Since the noise ratio is independent of $f$, we see that $K$ would be large if, roughly speaking, $f$ is low frequency. Most conventional images (with $f\ge0$) have a very large zero frequency $\hat f(0)$ relative to the rest of the spectrum and the second quotient in \r{3.15} does not vary much. When $\int f(x)\d x=0$, we have $\hat f(0)=0$ and functions the variation of this quotient is higher. Then we do not need to isolate the zero section. 

In Figure~\ref{fig_Rf/f} we demonstrate this effect. We choose $N=300$ and the dimensions of the grid for $\Ra f$ is chosen with equalities in \r{3.5N}, see also Figure~\ref{fig_1}. We add the same amount of normally distributed noise, $20\%$ of $\|\Ra f\|$, to $\Ra f$. We measure different percentages of added noise to the reconstructed $f$ depending on the frequency distribution of $f$, i.e., on the ratio in \r{3.15}. Images with mostly lower frequencies suffer from noise more. On the other hand, given the a priori knowledge of their frequency band, that noise can be filtered out, unless we are looking for small high frequency detail in an overly lower frequency image. We chose non-negative $f$'s in that figure only. Numerical experiments with $f$ of mean value zero show lower added noise on a few examples. If in Figure~\ref{fig_Rf/f}(c) we allow random positive and negative amplitudes as well, for example (not shown), the added noise in (g) drops to $41\%$. 
\begin{figure}[ht]\,\hfill
\begin{subfigure}[t]{.22\linewidth}
\centering
\includegraphics[height=.15\textheight]{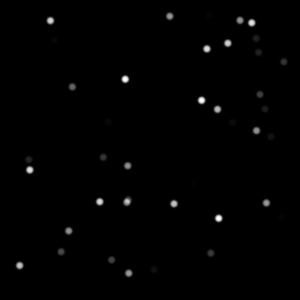}
\caption{$f_1$ }
\end{subfigure}\hfill
\begin{subfigure}[t]{.22\linewidth}
\centering
\includegraphics[height=.15\textheight]{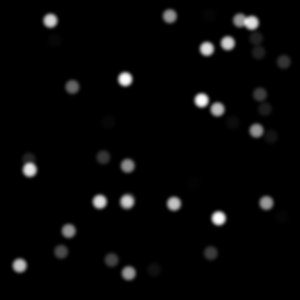}
\caption{$f_2$}
\end{subfigure}\hfill
\begin{subfigure}[t]{.22\linewidth}
\centering
\includegraphics[height=.15\textheight]{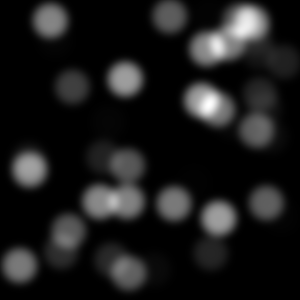}
\caption{$f_3$}
\end{subfigure}\hfill
\begin{subfigure}[t]{.22\linewidth}
\centering
\includegraphics[height=.15\textheight]{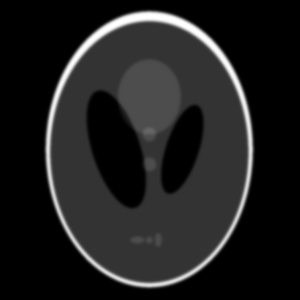}
\caption{$f_4$}
\end{subfigure}\hfill \,\\ \vspace{10pt}
\,\hfill
\begin{subfigure}[t]{.22\linewidth}
\centering
\includegraphics[height=.15\textheight]{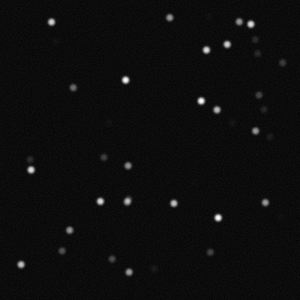}
\caption{$25\%$, $0.66$}
\end{subfigure}\hfill
\begin{subfigure}[t]{.22\linewidth}
\centering
\includegraphics[height=.15\textheight]{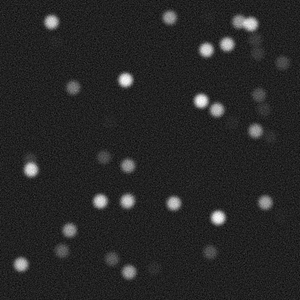}
\caption{$39.8\% $, $1.08$}
\end{subfigure}\hfill
\begin{subfigure}[t]{.22\linewidth}
\centering
\includegraphics[height=.15\textheight]{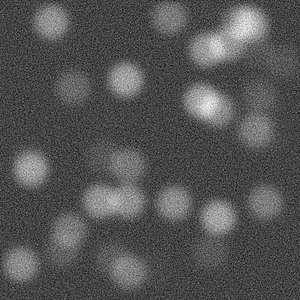}
\caption{$74.4\%$,$2.06$ }
\end{subfigure}\hfill
\begin{subfigure}[t]{.22\linewidth}
\centering
\includegraphics[height=.15\textheight]{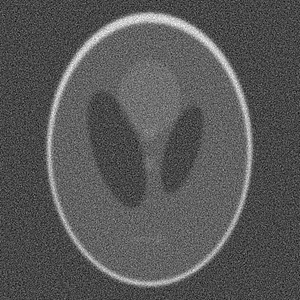}
\caption{$79.8\%$, $2.18$ }
\end{subfigure}\hfill \,
\caption{Top: Four different choices of $f\ge0$, $N=300$. Bottom: 
 $f$ reconstructed with $20\%$ noise added to $\Ra f$. The numbers show the added noise to $f$, and $\|\Ra f\|/\|f\|$.
}\label{fig_Rf/f}
\end{figure}
It is worth mentioning that with many conventional images, the values we are getting are close. In fact, statistically, such images share similar power spectra distributions \cite{van1996modelling}. 

Therefore, measuring the sensitivity of a particular inversion to noise this way can be quite misleading. The added noise to the image depends on the noise ratio \r{3.13} which in turn depends on the grid chosen to discretize $\Ra f$; and also depends on the choice of the test image.

\section{The Radon transform $\Ra$ in the plane in fan-beam coordinates} \label{sec_FB}
\subsection{$\Ra$ in fan-beam coordinates} 
We parametrize $\mathcal{R}$  by the so-called fan-beam coordinates. Recall \r{3.0a}. 
Each line is represented by an initial point $R\omega(\alpha)$ on the boundary of $B(0,R)$, where $f$ is supported, and by an initial direction making angle $\beta$ with the radial line through the same point,  see Figure~\ref{fig_FB}. It is straightforward to see that this direction is given by  $\omega(\alpha+\beta)$. Then the lines through $B(0,R)$ are given by 
\be{R2a}
x\cdot\omega(\alpha+\beta-\pi/2) = R\sin\beta, \quad \alpha\in [-\pi,\pi], \; \beta\in [-\pi/2,\pi/2].
\ee

\begin{figure}[ht]
\begin{center}
	\includegraphics[page=3 ]{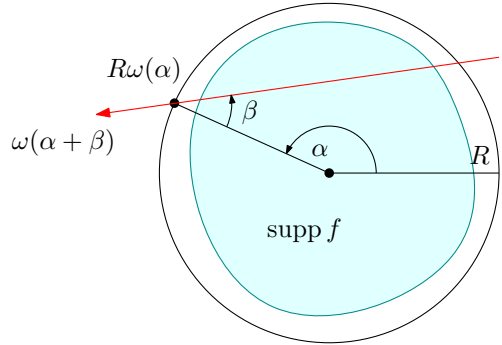}
\end{center}
\caption{\small The fan-beam coordinates.}
\label{fig_FB}
\end{figure}
The  canonical relation is $\kappa_\text{FB}=\kappa_+\cup \kappa_-$, where $\kappa_\pm$ are given by, see \cite{S-Sampling}, 
\[
\beta = \pm \sin^{-1} \frac{x\cdot\xi}{R|\xi|}, \quad \alpha = \arg\xi 
-\beta\pm \frac\pi2 \quad
\hat\alpha = x\cdot\xi^\perp, \quad   \hat\beta= \pm |\xi|\sqrt{R^2-(x\cdot \xi/|\xi|)^2}+\hat\alpha.
\]
Then $\kappa_\pm$ are isomorphic under the symmetry mentioned above lifted to the tangent bundle
\[
(\alpha, \beta, \hat\alpha,\hat\beta) \quad \longmapsto \quad (\alpha+2\beta-\pi , -\beta, \hat\alpha,2\hat\alpha-\hat\beta).
\]
The inverses $\kappa_\pm^{-1}$ are given by
\be{4.0c}
x = R\sin\beta \,\omega(\alpha+\beta-\pi/2) - \frac{\hat\alpha } { \hat\beta-\hat\alpha}R\cos\beta \,\omega(\alpha+\beta), \quad 
\xi =\frac{ \hat\beta-\hat\alpha}{R\cos\beta} \omega(\alpha+\beta-\pi/2) .
\ee
In particular, we recover the well known fact that $\kappa$ is 1-to-2, as in the previous case.

Set $(\varphi,p) = \Phi(\alpha,\beta)$,  where 
\[
\varphi =\alpha+\beta-\pi/2, \quad p=R\sin\beta.  
\]
We have $\det \d\Phi = R\cos\beta$. 
Then $\Ra_\text{FB}= \Ra\circ\Phi$. To compute $\Ra_\text{FB}^*\Ra_\text{FB}$, write 
\[
(\Ra_\text{FB} ^* \Ra_\text{FB} f, f) = \int |\Ra_\text{FB} f(\alpha,\beta)|^2\, \d\alpha\, \d\beta = \int |\Ra  f(\varphi,p)|^2\frac1{\cos\beta} \, \d\varphi\, \d p.
\]
Since $\sin\beta=p/R$, we have $\cos\beta= \sqrt{1-p^2/R^2 }$. Therefore,
\[
\Ra_\text{FB} ^* \Ra_\text{FB} = \Ra^* (1-p^2/R^2)^{-1/2} \Ra .
\]
The factor in the middle of the r.h.s.\ is a multiplication operator, and applying Egorov's theorem (one can actually do it even directly and without a remainder), one gets for  the principal symbols, at least,
\[
\sigma_p(\Ra_\text{FB} ^* \Ra_\text{FB}) = \bigg(1-\frac{(x\cdot\xi)^2}{R^2|\xi|^2}\bigg)^{-1/2}\sigma_p(\Ra^*  \Ra )= \frac{4\pi}{|\xi|}\bigg(1-\frac{(x\cdot\xi)^2}{R^2 |\xi|^2}\bigg)^{-1/2}. 
\]
The equivalent to \r{3.1} then is 
\be{4.3}
\gamma_{h^{1/2} \mathcal{R}_\text{FB}^{-1}g}(x,\xi)= \frac{|\xi|}{4\pi} \bigg(1-\frac{(x\cdot\xi)^2}{R^2 |\xi|^2}\bigg)^{1/2} \gamma_g\circ \kappa_\text{FB}(x,\xi), \quad \xi\not=0. 
\ee
Therefore, the noise spectral distribution depends on $x$ now, and it depends on the direction of $\xi$ relative to $x$. For $x, |\xi|$ fixed, it is maximized when $\xi\perp x$, and minimized when $\xi \parallel x$.

\subsection{Sampling} As above, if $\supp f\subset [-1,1]^2$ is sampled on an $N \times N$ grid, we have $B_{x_1}=B_{x_2}=N\pi/2 $. As before, set $B=\sqrt2 B_{x_1}=N\pi/\sqrt2$. Then we consider $f$ having $\WFH(f)$ in $B(0,R)\times B(0,\sqrt2 B)$ with $R=\sqrt 2$. The image of this product under the canonical map, projected to the dual variable $(\hat \alpha,\hat\beta)$ has the following smallest box containing it: $[-RB,RB]\times [-2RB, 2RB]$, see \cite{S-Sampling}. The means taking at least $2RB\times 2RB$, i.e., $2N\pi\times 2N\pi$ samples over the intervals indicated in \r{R2a}. Compared to \r{3.5N}, this requires $\pi$ times the number of samples, which makes it a less efficient sampling geometry, as  shown in \cite{S-Sampling}. 

In Figure~\ref{fig_ell}, we present a numerical experiment to validate \r{4.3}. We take $g$ to be  Gaussian noise and invert it with $\Ra_\text{FB}$. Then we crop a small rectangle in the top left corner and take the modulus of its Fourier Transform. Then $x$ is close to $x_0=0.8(-1,1)$ and the small black elongated oval in the center has a major axis along the same vector, formula \r{4.3} predicts. 

\begin{figure}[ht]
\begin{center}
	\includegraphics[scale=0.672]{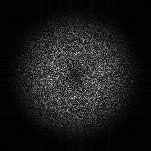}\hspace{50pt}
	\includegraphics[trim = 0 4 0 0  ,scale=0.26]{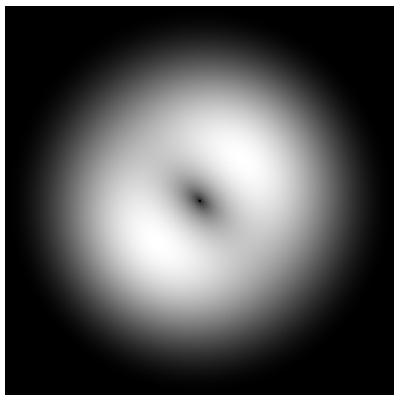} 
\end{center}
\caption{\small Spectral density of the noise in $f$ with the Hann filter. 
Left: measured in the top left corner. Right: theoretical profile \r{4.3} at that corner.}
\label{fig_ell}
\end{figure}

\subsection{Noise Ratio} 
We study the noise ratio with a filtered inversion. In {\tt ifanbeam} in MATLAB, for example, $\Ra_\text{FB}f$ is converted to parallel coordinates and the filter is applied after that. By \r{4.0c}, the filter $\nu(\hat p)$, with $\nu$ even, takes the form $F:= \nu\big( |\hat\beta-\hat\alpha|/(RB\cos\beta)\big) $, where $B$ is the band limit of $|\xi|$. The inversion operator then is $\Ra_{\text{FB},\nu}^{-1} =\Ra_{\text{FB}}^{-1}F $ which equals $( \Ra_{\text{FB},\nu}^* \Ra_{\text{FB},\nu} )^{-1}\nu(|D|)  \Ra_{\text{FB}}^*$ modulo lower order operators by Egorov's theorem. We get, similarly to \r{3.9} that  \r{4.3} modifies as 
\[
\gamma_{h^{1/2} \mathcal{R}_{\text{FB},\nu} ^{-1}g}(x,\xi)= \frac{|\xi|}{4\pi} \bigg(1-\frac{(x\cdot\xi)^2}{R^2 |\xi|^2}\bigg)^{1/2} \nu_0^2(|\xi|/B) \gamma_g\circ \kappa_\text{FB}(x,\xi), \quad \xi\not=0. 
\]

Assume that $g$ is oversampled (related to $B$, see \cite{S-Sampling} for the sampling requirements),  and it is white noise. Then the variance at a point, see \r{9a} is given by
\[
\begin{split}
\VAR_x(h^{1/2} f) &= \gamma^\sharp\int_{|\xi|\le B} \frac{|\xi|}{4\pi} \bigg(1-\frac{(x\cdot\xi)^2}{R^2 |\xi|^2}\bigg)^{1/2} \nu_0^2(|\xi|/B_p)\,\d\xi,\\ 
&= \frac{ \gamma^\sharp B^3}{4\pi}\int_0^{2\pi } \int_0^1\bigg(1-\frac{|x|^2}{R^2}\cos^2\theta \bigg)^{1/2} \rho^2 \nu_0^2(\rho)\,\d\rho\, \d\theta\\
&= \frac{ \gamma^\sharp B^3}{6}\frac{c_\nu}{2\pi} \int_0^{2\pi }  \bigg(1-\frac{|x|^2}{R^2}\cos^2\theta \bigg)^{1/2}    \, \d\theta,
\end{split}
\]
compare with \r{3.12}. The integral is of elliptic type and varies between $2\pi$, when $|x|=0$, and $4$ when $|x|=R$. To connect this to \r{3.12}, the integrand in \r{3.12} there corresponds to $|x|=0$ formally; and then we get \r{3.12}. Taking a square root, we see that the standard deviation would be higher in the center, the same as in the parallel geometry case, and will decrease slightly to about $80\%$ at $|x|=R$, which corresponds to the four corners of the square in our numerical simulations. 

Similarly to  \r{3.5} and \r{3.12}, we integrate over $x$ in the inscribed disk $|x|\le 1$ in $[-1,1]^2$ and divide by its area $\pi$ to get the variation in that disk. Then $R=\sqrt2$ and
\[
\VAR_{B(0,1)}( f)   \approx \frac{2.93}\pi \frac{ \gamma^\sharp B^3}{6h} c_\nu  \approx {0.9328}  \frac{ \gamma^\sharp B^3}{6h} c_\nu.
\]
This is within 6-7\% of the parallel geometry variance, and about $3\%$ difference for the standard deviation.


\section{Non-additive noise} \label{sec_NA_noise}
In this section we discuss some types of non-additive noise. The exposition here will be more sketchy, we will point out how to fit those cases into the general framework we developed but will not go into detail.

\subsection{Multiplicative noise} \label{sec_7.1}
Assume the data $g=Af$ is subject to a multiplicative noise. This can happen if the detectors are not perfectly calibrated and each one reports a signal somewhat larger or smaller than it should be (non-uniform response). In imaging systems, photo response non-uniformity (PRNU) is an example of such noise. 
{A generic way to model the kind of multiplicative noise we have in mind is the following: consider a sequence of discrete noise samples $\{\bm{w}_{k,h}; \, h>0, k\in K(h)\}$, where 
\be{MN0}
\bm{w}_{k,h} = 1 + \bm{f}_{k,h},
\quad\text{or}\quad
\bm{w}_{k,h} = \exp\left( \bm{f}_{k,h} \right),
\ee
and $\{\bm{f}_{k,h}; \, h>0, k\in K(h)\}$ is the white noise considered in Hypothesis~\ref{hyp:noise}. Then we set
}
\be{MN1}
g_\text{noise}(x) =  \sum_{k\in K(h)} \bm w_{k,h} g(shk) \chi_k(x), \quad \chi_k(x)=  \chi\Big(\frac{1}{sh}(x-shk)\Big)  , 
\ee
where $\bm w_{k,h} $ are the discrete noise samples, playing the role of $w_\text{\rm noise}$ above, and $g$ is the noise-free continuous signal. 
{We will compare the noise $g_h(x)$ defined by \r{MN1}, i.e., by the first formula in \r{MN0} (the second one can be treated similarly and one has to take into account that $\bm{w}_{k,h}$ is not necessarily centered), to a noise of the form 
}
\be{MN2}
\tilde g_\text{noise}(x) :=  \sum_{k\in K(h)} \bm w_{k,h} g(x) \chi_k(x) = g(x)  \sum_{k\in K(h)} \bm w_{k,h} \chi_k(x).
\ee
We have
\[
 \tilde g_\text{noise}(x)-g_\text{noise}(x) =  \sum_{k\in K(h)} \bm w_{k,h}\left( g(x)- g(shk)\right) \chi_k(x). 
\]
Since
\[
|g(x)- g(shk)|\le C|x- shk|,
\]
with $C\le C_0\|\nabla g\|_{L^\infty}$, we get
\[
\left| \left( g(x)- g(shk)\right) \chi_k(x)\right|\le Csh\max(|x||\chi(x)|)\le C'h. 
\]
That factor $h$ allows us to estimate, using Proposition~\ref{pr_P}, the error when replacing $g(shk)$ in \r{MN1} by $g(x)$. We would get an $O(h)$ error. 
The problem here is that we want to apply this to $g=Af$, all dependent on $h$, and in general, $\nabla Af$ grows like $h^{-1}|Af|$. This cancels the decay above. If we oversample a lot, the error will be ``small''. Also, in regions with $\WFH(Af)$ far away from the Nyquist limit, that term will be small. If we ignore it for a moment, the noise added to $Af$ is $\tilde g_h$ given by \r{MN2}. It is white noise as above but multiplied by $g=Af$. The defect measure of the noise added to the data then is like in \r{de6b} with the additional factor $|Af(x)|^2$. 

One important case which allows to overcome the difficulty above is when $g=\psi_h*g_0$, where $\psi_h(x) = h^{-n}\psi(x/h)$ with $\int\psi=1$ (a  Friedrichs mollifier) with $\hat\psi\in C_0^\infty$. This corresponds to averaged measurements of an $h$-independent function $g_0$. We refer to \cite{S-Sampling} for the sampling theory for such measurements. 
Then $\nabla g=\psi_h*\nabla g_0$, and assuming $g_0\in C^1$ (either $h$-independent or uniformly bounded there   in $h$), we have $|\nabla g|\le C$ (rather than $C/h$), which is the estimate we needed in the previous paragraph. Then the machinery we developed works and we need to multiply the noise measure by the additional factor $|g_0(x)|^2$ in \r{de6b}, i.e., we get there
\[
\d\mu_{g_\text{noise}}(x,\xi) = \frac{s^n}{(2\pi)^n}\sigma^2 |g_0(x)|^2 | \hat\psi (\xi)   \hat\chi(s\xi)|^2\,\d x\,\d\xi
\]
under the assumption $g=\psi_h*g_0$ in \r{MN1}. Then \r{3.12} takes the form 
\be{MN5}
\begin{split}
\VAR^0_x(h^{1/2} \mathcal{R_\chi}^{-1}g_\text{noise}) &  =  \frac{1}{4\pi}\gamma^\sharp  \int_{|\xi|<B} \big|\Ra f\big(x\cdot\xi/|\xi|, \arg(\xi)\big )  \big|^2  |\xi|| \nu_0(|\xi|/B)  \hat\psi (\xi)  |^2 \,\d\xi\\
& = \frac{1}{4\pi}\gamma^\sharp  \int_0^{2\pi} \int_0^B \big|\Ra f\big(x\cdot \theta, \theta\big )  \big|^2   \rho^2 |\hat\psi(\rho\theta)|^2 \nu_0^2(\rho/B)\,\d\rho\, \d\theta. 
\end{split}
\ee
This shows that the standard deviation of the noise at $x$ depends in particular on the line integrals of $f$ along lines thorough $x$. Line integrals with large values would create stronger noise at $x$.

\subsection{Modeling noise in CT scan} 
In CT scan tomography, what is {measured} is the attenuation along each ray. If $I_0$ is the initial intensity, and $I$ is the one after the ray crosses the object, then the measurement is $I = \exp(-\Ra f)I_0$, by the Beer-Lambert law. Assuming an additive noise $g_\text{noise}$, we measure $I_\text{noisy} = \exp(-\Ra f)I_0+ g_\text{noise}$. If we invert this the same way as if there were no noise (which may not be the best strategy), we would get
\[
f_\text{noisy} = -\Ra^{-1}\log(I_\text{noisy}/I_0)= -\Ra^{-1}\log(\exp(-\Ra f)+ g_\text{noise}/I_0).
\]
Obviously, increasing $I_0$ will decrease the effect of the added noise but in many applications this is not desirable and/or the noise level may depend on $I_0$. We take $I_0=1$, i.e., $g_\text{noise}$ is the added noise relative to $I_0$. Then 
\[
\begin{split}
-\log(\exp(-\Ra f)+ g_\text{noise}) &= -\log\left(\exp(-\Ra f)(1+ \exp(\Ra f)g_\text{noise})\right)\\
& = \Ra f - \log\left(1+ \exp(\Ra f)g_\text{noise} \right).
\end{split}
\]
Therefore,
\be{CT2}
f_\text{noisy} = f - \Ra^{-1} \log\left(1+ \exp(\Ra f)g_\text{noise} \right)
\ee
If the noise is small enough, we can pass to a linearization to get 
\be{CT3}
f_\text{noisy} \approx f- \Ra^{-1}\left(\exp(\Ra f)g_\text{noise} \right). 
\ee
This is the multiplicative noise model above with $g w_\text{noise}$ replaced by $e^g w_\text{noise}$. In \r{MN5}, for example, the factor $|\Ra f|^2$ would be replaced by $\exp(2\Ra f)$.  

\subsection{Modeling Poisson noise} In SPECT, we measure the attenuated X-ray transform but the particle count at each detector is low. In this case, the predominant noise is of Poisson type: the number of particles at each detector is randomized by a Poisson distribution with probability of taking value $k$ being $P(k,\lambda) = e^{-k}\lambda^k/k!$, where $\lambda\ge0$ is the expected value at that detector, see  \cite[sec.~4.5]{Papoulis2002}. 
 Both the expected value and the variance of $P$ equals $\lambda$. Then the particle count at each detector equals $\lambda+w\sqrt{\lambda}$, where $w$ is a random variable with zero expected value and variance $1$. Note that the probability distribution of $w$ depends on $\lambda$ and approximates a Gaussian one when $\lambda\gg1$ and they are independent. Assuming locally averaged measurements, as above, we would get added noise $\bm w_{k,h}|\psi_h*\Ra f|^{1/2}$ when the units for $\Ra f$ are the number of particles; and $\alpha$ times that in general with some $\alpha>0$. Note that $\bm w_{k,h}$ are not identically distributed (but Theorem~\ref{thm_m} still applies) and  are well approximated by Gaussian  distributions when $\Ra f$ is not very small. The microlocal measure then would have the factor  $\alpha \Ra f$ (we assume $f\ge0$, thus $\Ra f\ge0$). This is similar to multiplicative noise, where the factor was proportional to $|\Ra f|^2$. 

\subsection{Numerical examples} 
We present numerical simulations with the three types of non-additive noise in Figures~\ref{fig_multSL} and \ref{fig_multD}. The phantoms are the Shepp-Logan one and three disks of different size and intensity, not shown there,  both phantoms having ranges between $0$ and $1$. They are both rendered on a $300\times300$ grid discretizing the square $[-1,1]^2$. Their Radon transforms are computed with 1,884 angular steps and $600$ steps in the $p$ variable covering the diagonal of the square. To simulate multiplicative noise, we choose the variance of $\bm w$ in \r{MN1} to be $0.2$. To simulate CT noise, we use the non-linear model \r{CT2} (rather than the linearization \r{CT3}) with $\VAR(g_\text{noise})=0.03$. In the Poisson noise case, each value of $\Ra f$ is randomized as follows: $\texttt{poissrnd}(80*\Ra f)/80$; it is worth noticing that $\Ra f$ ranges from $0$ to $0.51$ in the Shepp-Logan case and to $0.56$ in the disks case. We chose the noise parameters so that the noise would be of similar strength, visibly,  in all three cases, and the distribution is Gaussian.  Hann filter is applied to the inversion. 

\begin{figure}[ht]
\begin{center}
	\includegraphics[scale=0.3]{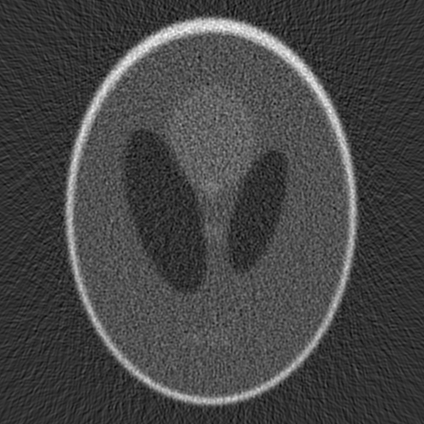}\hspace{20pt}
\includegraphics[scale=0.3]{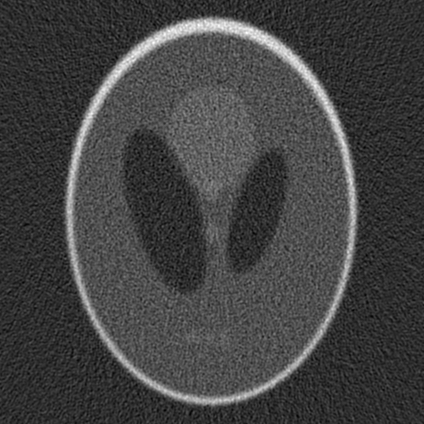}\hspace{20pt}
\includegraphics[scale=0.3]{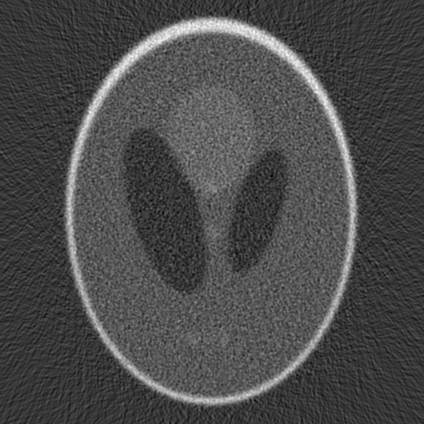}\hspace{20pt}
\end{center}
\caption{\small Shepp-Logan with the Hann filter and with (a)   multiplicative noise; (b) CT type of noise; (c) Poisson noise} 
\label{fig_multSL}
\end{figure}

\begin{figure}[ht]
\begin{center}
	\includegraphics[scale=0.3]{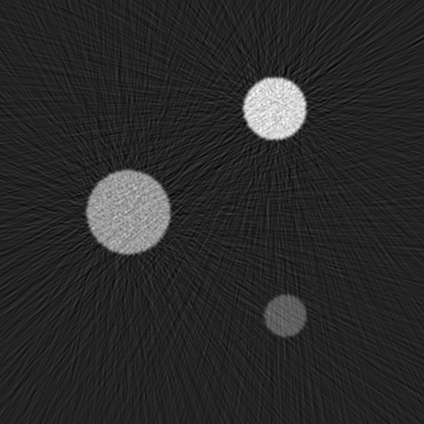}\hspace{20pt}
\includegraphics[scale=0.3]{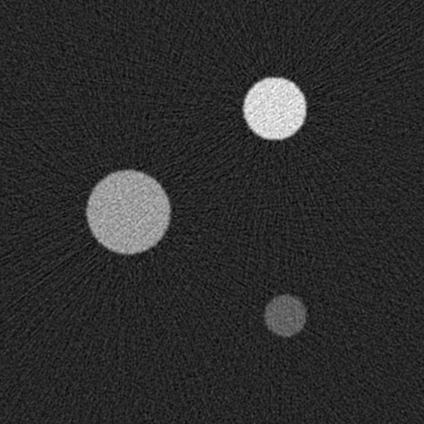}\hspace{20pt}
\includegraphics[scale=0.3]{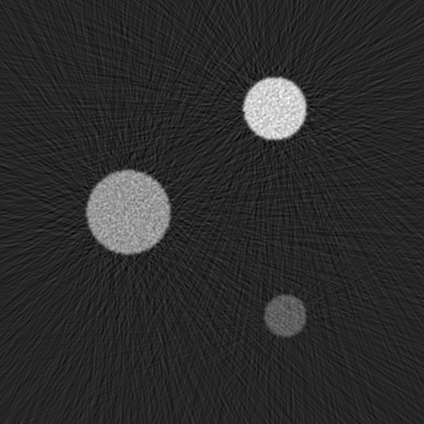}\hspace{20pt}
\end{center}
\caption{\small Three disks with the Hann filter and with (a) multiplicative noise; (b) CT type of noise; (c) Poisson noise} 
\label{fig_multD}
\end{figure}

Note that the noise has different character compared to Figure~\ref{fig_Rf/f}(h), for example, and one can see   individual lines (more precisely, line segments) in it. The multiplicative and the Poisson noise characters are somewhat similar; while the CT noise in the middle looks different. Our analysis shows that in the latter case, the standard deviation of the added noise in the linearization regime has range from $e^0=1$ to about $e^{0.5}\approx 1.65$ times $\STD(g_\text{noise})$, see \r{CT3}, while in the other two cases, the range is from $0$ to a certain positive constant, which allows for almost zero noise locally before inversion. For this reason, individual lines are harder to distinguish in the CT case. 
 
\section{Discrete noise and its power spectrum}\label{sec_8}
In this section, we analyze discrete white noise directly, without converting it to a continuous function. Here, $f(k)$ is a random vector on an $N\times \dots\times N$ grid which we denoted by $\bm f_k$ before. We will denote by $\delta(k)$ the discrete delta function on $\mathbf{Z}^n$. In section~\ref{sec_D_1}, we follow mainly  \cite[Chapter~12]{Papoulis2002}, where $f$  is a random variable depending on a (continuous) variable $t$; but most of it adapts to the discrete setting easily. We do a temporal analysis of the power spectrum for each fixed (discrete) frequency, with $N$ fixed. We show that the spectrum of white noise is flat in the sense of expected value over repeated experiments, and we consider more general noise. On the other hand, for each experiment, the spectrum is quite, well, noisy and does not appear to smoothen as $N\to\infty$ numerically. 
In the second part, we study the ergodic properties of the power spectrum, with a single experiment, as $N\to\infty$. We show in Theorem~\ref{thm_d1} that the power spectrum is flat \textit{on average}. That theorem is an analog of Theorem~\ref{thm_m}. 

We want to emphasize that $f=\{f(k)\}$ depends on $N$, so we have a ``triangular'' array of random variables depending  on the random outcome and increasing their size with $N$.

\subsection{Temporal analysis}\label{sec_D_1}\label{sec_D} 
The discrete analog of the Fourier transform is the Discrete Fourier Transform (DFT) described below. It lives naturally on the discrete torus  $\mathbf{T}^n_N = \mathbf{Z}^n/N\mathbf{Z}^n$ with period $N$. This shows that any time the DFT is used for spectral analysis, the original $f$ is actually regarded as the restriction of a periodic function on a fundamental domain. 
We consider   $f:\mathbf{T}^n_N  \to \mathbf{C}$, and we denote $f=\{ f(k) \}$;  with each element $f(k)$, $k=(k_1,\dots,k_n)\in  \mathbf{T}^n_N$ a random variable in the same probability space. 
In other words, $m$ varies on the $n$-dimensional discrete torus with $N$ points in each variable. First, $N$ will be fixed but eventually, we will take $N\gg1$.  We denote by $fg$ the vector defined by $(fg)(k)=f(k)g(k)$, i.e., this is the multiplication of the functions of a discrete argument. Similarly, $|f|$ is the vector with components $|f(k)|$, while $\|f\|$ is the norm of $f$. 

We define the (unitary) Discrete Fourier Transform (DFT) $\hat f=\mathcal{F}f$ by
\[
\hat f(k^*)= \frac1{N^{n/2}}\sum_{k\in \mathbf{T}^n_N}  f(k) e^{-2\pi \i k \cdot k^*/N}, \quad k^*\in \mathbf{T}^n_N .  
\]
Its inverse is the adjoint one 
\[
f(k)= \frac1{N^{n/2}}\sum_{k^*\in \mathbf{T}^n_N} \hat f(k^*) e^{2\pi \i k\cdot k^*/N}, \quad k\in  \mathbf{T}^n_N . 
\]
Parseval's equality takes the form
\[
f\cdot g=  \hat f\cdot\hat g
\]
for complex-valued $f$ and $g$, where the dot-product is the natural one in $\mathbf{C}^N$. In particular, $\mathcal{F}$ is unitary. There is a natural (circular) convolution $f*g$ defined, and we have
\[
\mathcal{F}(f*g)= N^{n/2} \hat f \hat g, \quad \mathcal{F}(f  g)= N^{-n/2} \hat f* \hat g.
\]
Next, we have
\be{2.0.1}
\mathcal{F}\delta=N^{-n/2}, \quad \mathcal{F}(1)= N^{n/2}\delta. 
\ee

For each $f$ with random entries, as above,  define the \textit{auto-correlation }
\[
\acor_f(m,k)=\mathbb{E}\{f(m) \bar f(k)\}. 
\]
The \textit{auto-covariance} is defined  as the auto-correlation  of the centered $f$, i.e., of $f-\mathbb{E}(f)$, and it is easy to see that
\[
\acov_f(m,k)= \acor(m,k)-\mathbb{E}\{f(m)\}  \overline{\mathbb{E}\{f(k)\}}.
\]

The process $f$ is called \textit{stationary},\footnote{this terms comes from 1D processes, where $x$ is the time} if $\acor_f(m,k)$ is a function of $m-k$ only:
\be{2.1.3}
\acor_f(m,k)=\acor_f(m-k),
\ee
where, with some abuse of notation, we used the same notation $\acor$ on the right. A process $f$ is called \textit{white noise} if 
\be{2.1.4}
\acor_f(m,k)=0 \quad \text{for $m\not=k$}. 
\ee
Then we must have
\be{2.1.5}
\acor_f(m,k)=\sigma^2(m)\delta(m-k) 
\ee
with $\sigma^2(m)=\VAR(f(m))\ge0$. We always assume that white noise has a zero mean. The process is wide-sense stationary (WSS) if it is stationary and its mean is constant. Then white noise is WSS if $\sigma$ is constant. Note that WSS does not mean that $f(m)$ are independent from each other but if they are independent, they are uncorrelated, i.e., \r{2.1.4} holds.

Let $\Gamma(m^*,k^*)$ be the  DFT of the auto-correlation of $f$, see \r{2.1.3}, with respect to $(m,k)$:  
\[
\Gamma(m^*,k^*)= \mathcal{F}(\acor_f)(m^*,k^*).
\]
 Then
\[\begin{split}
\mathbb{E}\big\{ \hat f(m^*) \bar{\hat f}(k^*)\big \}= \mathbb{E} \frac1{N^n}\sum_{m,k }  f(m)\bar f(k)e^{-2\pi \i (m\cdot m^*-k \cdot k^*  )/N}
= \Gamma(m^*,-k^*).
\end{split}
\]

In case of white noise satisfying \r{2.1.5}, we have $\Gamma(m^*,k^*)=N^{-n/2}\widehat{\sigma^2}(m^*+k^*)   $, thus we recover Theorem~11.2 in  \cite{Papoulis2002}:
\[
\mathbb{E}\big\{ \hat f(m^*+k^* ) \bar{\hat f}(m^*)\big \}= N^{-n/2}\widehat{\sigma^2}(k^*), \quad \forall m^*. 
\]
This shows that even when $f$ is not stationary,  $\hat f$ is stationary with auto-correlation $\widehat{\sigma^2}$. If $\sigma=\text{const.}$, then each $f(m)$ has standard deviation $\sigma$, and $N^{-n/2}\widehat{\sigma^2}=\sigma^2 \delta$, i.e.,
\be{d1-2}
\mathbb{E}\big\{ \hat f(k^* ) \bar{\hat f}(m^*)\big \}= \sigma^2  \delta(k^*-m^*).
\ee
In particular, $\mathbb{E}\{ |\hat f(m^*)|^2\}=  \sigma^2 $ for all $m$, which shows a flat (expectation of a) spectrum. By Theorem~11.3 in  \cite{Papoulis2002}, if $f$ is real and Gaussian, then the covariance of $|\hat f(m^*)|^2$ and $|\hat f(k^*)|^2$ equals $N^{-n} (\widehat{\sigma^2})^2(m^*+k^*)+ N^{-n} (\widehat{\sigma^2})^2(m^*-k^*)$, as we also show below. In particular, if $\sigma=\text{const.}$ in \r{2.1.5}, we get covariance $\sigma^4  (\delta(m^*+k^*)+\delta(m^*-k^*))$. Therefore, they are correlated when $k^*=m^*$ and $k^*=-m^*$ (because  $\hat f$ is even) with standard deviation $\sigma^4 $   for each Fourier coefficient except for the zeroth one when it is $2\sigma^4$. 
 In fact, we do not need $f$ to be Gaussian to have the same conclusion on asymptotic sense. We assume $f$ real from now on.

\begin{proposition}\label{pr_d}
Let $f$ be real valued white noise with a  finite fourth moment called $\mu_{4}$. Then
\be{pr81}
\acov\left\{|\hat f(k^*)|^2, |\hat f(m^*)|^2 \right\}= \sigma^4 \delta(k^*- m^*)+\sigma^4 \delta(k^*+ m^*) + \frac{\mu_4-3\sigma^4}{N^n}.
\ee
\end{proposition}
\begin{proof} We have 
\[
\begin{split}
\acov &\left\{|\hat f(k^*)|^2, |\hat f(m^*)|^2 \right\}\\
&=   \frac1{N^{2n}}\sum_{m_1,m_2,k_1,k_2 }\mathbb{E}\{f(k_1)f( k_2)f(m_1)f( m_2)\}  e^{-2\pi \i ((k_1-k_2) \cdot k^* +(m_1 -m_2 )\cdot m^*  )/N}- \sigma^4.
\end{split}
\]
The only non-zero expectation terms are those with two (equal or not) pairs of equal indices. Assume first that $k_1=k_2$, $m_1=m_2$. Then the expectation term equals $\sigma^4$ if $m_1\not=m_2$, and the fourth moment $\mu_4$, when $ k_1=k_2=m_1=m_2$. The latter number of terms is $N^n$, while the former is $N^{2n}-N^n$. Therefore,  
this set of indices contributes
\[
\left(1-\frac1{N^n}\right)\sigma^4 +  \frac1{N^n}\mu_4
\]
to the sum. 

Consider the terms with $k_1=m_1$, $k_2=m_2$. Then the corresponding sum is
\[
\begin{split}
\frac{\sigma^4 }{N^{2n}} \sum_{k_1=m_1\not= k_2=m_2 } e^{-2\pi \i (k_1-k_2) \cdot( k^* + m^* )/N}
 &= \frac{\sigma^4 }{N^{2n}}\sum_{k,m} e^{-2\pi \i k \cdot( k^* + m^* )/N} -\frac{\sigma^4 }{N^{n}} \\
&=\sigma^4 \delta (k^*+ m^*) -{\sigma^4 }/{N^{n}} .
\end{split}
\]
We performed the change $k=k_1-k_2$, $m=k_2$ above, used \r{2.0.1} and compensated for the added terms corresponding to $k=0$ in the second sum which are missing from the first one.

Finally, when $k_1=m_2$, $k_2=m_1$, the dot product in the phase function becomes $(k_1-m_1)\cdot(k^* - m^*)$ and the same argument gives us 
\[
\begin{split}
\frac{\sigma^4 }{N^{2n}}\sum_{k_1\not= m_1}  e^{-2\pi \i (k_1-m_1)\cdot(k^* - m^* )/N} &= \frac{\sigma^4 }{N^{2n}}\sum_{k\not=0} \sum_m e^{-2\pi \i k\cdot(k^* - m^* )/N}  \\
& =\sigma^4\delta(k^* - m^*) -{\sigma^4 }/{N^{n}} . 
\end{split}
\]
The analysis of those three cases completes the proof. 
\end{proof}
 
\begin{corollary}\label{corD}
If $f$ in Proposition~\ref{pr_d} is normal, then the last term in \r{pr81} vanishes. 
\end{corollary}

The proof follows from the well know fact that $\mu_4=3\sigma^4$ for normal distributions.  

\begin{remark}
The results in Proposition~\ref{pr_d} can be interpreted as follows. Up to an error $O(N^{-n})$, we get auto-covariance $\sigma^4$ if $k^*=m^*\not=0$ and when $k^*=-m^*\not=0$ (symmetry, because $f$ is real), and 
$2\sigma^4$ if $k^*=m^*=0$. 
If we stay in a fundamental domain of the type $k_j\in \{0,1,\dots,N-1\}$ then the symmetry becomes $|\hat f(k^*)|^2 = |\hat f(N-k_1^*,\dots N-k_N)|^2$.
\end{remark}

\subsection{Ergodic analysis. Flatness of the power spectrum on average.} 
Let $\alpha$ be a locally Riemann integrable function on $\R^n$, periodic of period $1$ in each variable. Assume that $f$ is real valued. We are interested in the following linear functional
\[
\mu_N(\alpha) := \frac1{N^n} \sum_{k^*\in \mathbf{T}^n_N}\alpha(k^*/N) |\hat f(k^*)|^2.
\]
This is a discrete analog of \r{de6a}   with $p$ there depending on the dual variable only. It is a weighted (not normalized) average of the power spectrum. What we do there is effectively rescale the spectrum from the integer points in $[0,N-1]^n$ (and then extended by periodicity) to the ones with fractional components of the kind $k^*/N$, forming a dense set in $[0,1]^n$ asymptotically. In statistics, this is done routinely in the study of peridograms, and $k^*/N$ is replaced by a continuous variable. 

Assume a white noise process \r{2.1.5} with $\sigma=\text{const}$. Then $\mathbf{E}(\mu_N(\alpha))\to \sigma^2\int \alpha(\xi)\,\d \xi$, as $N\to\infty$ 
by \r{d1-2}, where the integration is taken over the continuous torus in $\R^n$ with period one.  

The random variables $|\hat f(k^*)|^2-\sigma^2$  have zero expectation, correlation  given by  Proposition~\ref{pr_d},  and variance $\sigma^4$. 
Write 
\[
\mu_N(\alpha)= \sigma^2\frac1{N^n} \sum_{k^*\in \mathbf{T}^n_N}\alpha(k^*/N) + \frac1{N^n} \sum_{k^*\in \mathbf{T}^n_N}\alpha(k^*/N)\left(|\hat f(k^*)|^2 - \sigma^2\right).
\]
The first term is a Riemannian sum. 
The  second term has zero expectation and variance
\[
 \frac{\sigma^4}{N^{2n}} \sum_{k^*\in \mathbf{T}^n_N} \left( |\alpha(k^* /N)|^2         + \alpha(k^* /N)     \bar \alpha(-k^* /N)     \right)+O\Big(\frac1{N^n}\Big)
\le\frac{C}{N^n}. 
\]
The error terms come from the cross terms which are products of $\sigma^4$ and $O(N^{-n})$, by Proposition~\ref{pr_d}. There are $N^{2n}$ of them. 
Therefore, we proved the following.

\begin{theorem}\label{thm_d1} Let $f(k)$ be a white noise process on $\mathbf{T}^n_N$  (depending on $N$), with variance $\sigma^2$ and a finite fourth momentum. Then for every Riemann integrable function  $\alpha$ on ${T}_N^n$ we have
\be{thm_d1:eq}
\mu_N(\alpha) \longrightarrow \sigma^2\int \alpha(\xi)\,\d \xi \quad \text{in mean square sense}, 
\ee
where the integral is taken over the torus in $\R^n$ with period one. 
\end{theorem} 

Therefore, the measure $N^{-n}\sum_{ \mathbf{T}^n_N} |\hat f(m^*)|^2\delta(\xi-m^*/N)$ converges weakly to $\sigma^2\, \d\xi $ in mean square sense. In particular, if we take $\alpha$ to be the characteristic function of, say a box $U$ in $\mathbf{T}^n_N$, then the average of the power spectrum on $U$ tends to $\sigma^2$ in mean square sense. 

\subsection{More general noise} We assume now that the random variables $f_h(k)$ depend on $h$, have zero mean and have uniformly bounded fourth momenta but are not necessarily independent or equally distributed. If we assume \r{Acor_f}, then the power spectrum is expressed in Theorem~\ref{thm_m2}. One special but important case is when the auto-correlation is space independent (stationary, see \r{2.1.3}), then $\beta$ in \r{Acor_f} is independent of $x$ and we have 
 \be{Acor_f_d}
\acor( f_{h}(k),  f_h(k+m)) = \beta(m)
\ee 
with some $\beta(m)$. Then \r{thm_d1:eq} takes the form 
\[
\mu_N(\alpha) \longrightarrow \int \check\beta(\xi) \alpha(\xi)\,\d \xi \quad \text{in mean square sense}, 
\] 
where $\check\beta$ is as in \r{beta}. In other words, the limit measure is $\check\beta(\xi)  \,\d \xi$.

\subsection{Numerical examples}
We illustrate the temporal behavior of the spectrum first. In Figure~\ref{fig_power_sp_seq_2}, we take a random normally distributed vector $f$ with $N=200$ and variance $\sigma^2=1$. The power spectrum is plotted next to it. As we can see, it looks flat on average with mean value close to one but the variation is substantial. On the right we plot the histogram of the (spatial) standard deviation $\STD(|\hat f|^2)$ over $1,000$ experiments; it appears to have mean $1$. We recall that $\STD(|\hat f|^2)$ is the square root of 
\[
\VAR\big( |\hat f|^2\big)=\frac1{N^n} \sum_{k^*\in \mathbf{T}^n_N} \left(|\hat f(k^*)|^2-\sigma^2\right)^2.
\]
We have not proved a limit for it though. That would require estimating the auto-correlation of the summands above similarly to Proposition~\ref{pr_d}.

\begin{figure}[ht]
\begin{center}
	\includegraphics[trim = 40 10 40 10  ,scale=0.35]{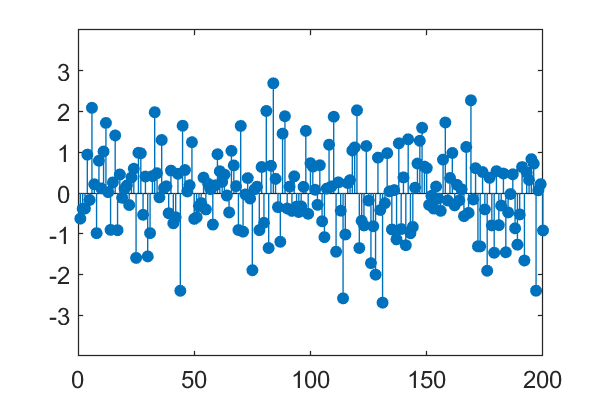}
\includegraphics[trim = 0 10 0 10  ,scale=0.35]{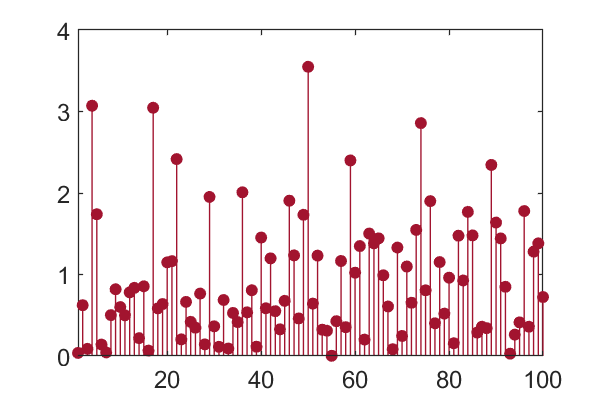}
\includegraphics[trim = 0 15 0 10  ,scale=0.325]{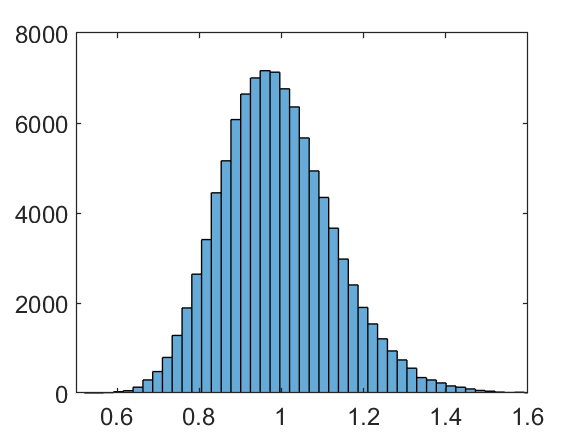}
\end{center}
\caption{\small Left: a random normally distributed vector, $N=200$, $\sigma^2=1$. Center:  plot of $|\hat f|^2$ for indices from $0$ to $100$ ($|\hat f|^2$ is an even function with period $200$). Right: the histogram of $\STD(|\hat f|^2)$ over $100,000$ experiments; it appears centered around $1$.}
\label{fig_power_sp_seq_2}
\end{figure}

Next, we illustrate the spatial (ergodic) behavior of the power spectrum. 
The averaged power spectrum for a normally distributed vector is shown in Figure~\ref{fig_power_sp}. We divide the interval $[0,N/2]$ into $25$ subintervals and average in each one of them. We take $N=10^2, 10^3, 10^4$ and $N=10^5$. As we can see, the averaged spectrum gets flatter and flatter. This illustrates Theorem~\ref{thm_d1}. 
\begin{figure}[ht]
\begin{center}
	\includegraphics[trim = 0 10 0 10  ,scale=0.4]{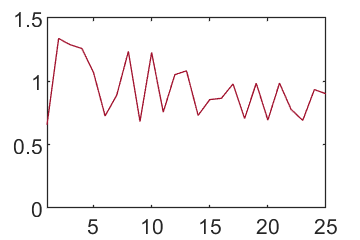}
\includegraphics[trim = 0 10 0 10  ,scale=0.4]{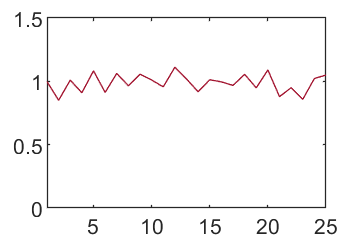}
\includegraphics[trim = 0 10 0 10  ,scale=0.4]{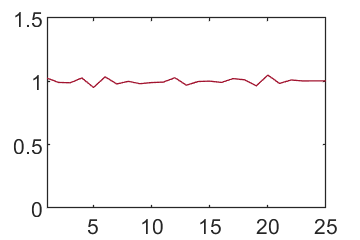}
\includegraphics[trim = 0 10 0 10  ,scale=0.4]{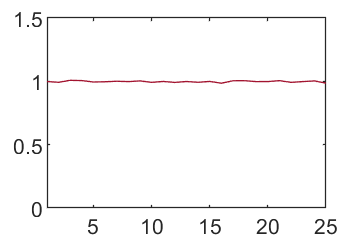}
\end{center}
\caption{\small  Plot of the averaged $|\hat f|^2$ for $N=10^2, 10^3, 10^4$ and $N=10^5$.}
\label{fig_power_sp}
\end{figure}

If we keep $N$ fixed but average over many experiments, the spectrum gets flatter as well numerically, as \r{d1-2} suggests.



\end{document}